%% file: SOC-in-limited-aggreg.tex
\documentclass{enarticle}

\usepackage{appendix}
\numberwithin{equation}{section}

\author[1]{Mathieu \sc{Merle}}
\author[2]{Raoul \sc{Normand}}
\location[1]{Laboratoire de Probabilit\'es et Mod\`eles Al\'eatoires, Universit\'e Paris 7, Paris, France}
\location[2]{Institute of Mathematics, Academia Sinica, Taipei, Taiwan}
\email{merle@math.univ-paris-diderot.fr; rnormand@math.sinica.edu.tw}

\title{Self-organized criticality in a discrete model for Smoluchowski's equation with limited aggregations}

\input{0-Shortcuts}

\begin{document}
\maketitle 

\abstract{
We introduce and study a discrete random model for Smoluchowski's equation with limited aggregations. The latter is a model of coagulation introduced by Bertoin which may exhibit gelation. 

In our model, a large number of particles are initially given a prescribed number of arms. These arms are activated independently after exponential times and successively linked together. However, when the size of a cluster goes above a fixed threshold, it falls instantaneously into the gel, meaning that it no longer interacts with the other particles in solution.

The concentrations in this model asymptotically obey Smoluchowski's equation with limited aggregations. In this article, we study the discrete features of this model. We are able to argue that it remains closely related with a configuration model. We specifically obtain explicit expressions for the parameters of this configuration model, which show that it is subcritical before gelation, but remains critical afterwards. As a consequence, the asymptotic distribution of a typical cluster in solution is that of a subcritical Galton-Watson tree before gelation, while it is that of a critical Galton-Watson tree after gelation. Our model therefore exhibits self-organized criticality.

Our study relies heavily on the study of the configuration model in the critical window. One of the main results of the paper is an extension of a result of Janson and Luczak on the size of the largest components in this critical window.

Finally, as a consequence of our explicit expressions, we provide an explanation of analytic formulas established by Normand and Zambotti regarding the limiting concentrations in Smoluchowski's equation.
}


\MSC Primary: 60K35; Secondary: 05C80

\input{1-Intro}

\input{2-Outline}

\input{3-Smolu}

\input{4-Prelim} 

\input{5-JL}

\input{6-GeltimesDCM}

\input{7-Altmodel} 

\input{8-Limit}

\input{9-Conclusion}

\appendix

\input{A-Index-notation}

\bibliographystyle{abbrv}
\bibliography{bibli}

\end{document}

%% file: 0-Shortcuts.tex
\newcommand{\cinf}{c_{\infty}}
\newcommand{\taugel}{\tau_{\mathrm{gel}}}
\newcommand{\GW}{\bbG\W}
\renewcommand{\N}{^{(N)}}
\newcommand{\n}{^{(n)}}

\newcommand{\MNf}{M^{(N),f}}
\newcommand{\An}{A^{(n)}}

\newcommand{\gel}{\mathrm{gel}}

\newcommand{\Tgel}{T_{\gel}}

\newcommand{\bN}{\mathbb{N}}

\renewcommand{\cE}{{\mathcal E}}
\renewcommand{\cF}{{\mathcal F}}
\renewcommand{\cG}{{\mathcal G}}

\newcommand{\beq}{\begin{eqnarray*}}
\newcommand{\feq}{\end{eqnarray*}}
\newcommand{\beqn}{\begin{eqnarray}}
\newcommand{\feqn}{\end{eqnarray}}

\newtheorem{assume}[thm]{Assumption}

\newcommand{\Bin}{\mathrm{Bin}}

\renewcommand{\d}[2]{\dfrac{\partial #1}{\partial #2}}

\renewcommand{\fm}{\mathfrak{m}}

\usepackage{xcolor}

\newcommand{\tK}{\widetilde{\K}}

\newcommand{\Ctyp}{\cC_{\mathrm{typ}}}

%% file: 1-Intro.tex
\section{Introduction} \label{sec:intro}

\subsection{Smoluchowski's equation}

The standard Smoluchowski's equation was first introduced in 1916 in \cite{Smolu}, to model coalescing particles. As time passes, clusters of particles are created. One assumes that each particle has unit mass, so a cluster has mass $m$, the number of particles it contains. As time passes, particles, and hence clusters, coalesce pairwise. The phenomenon is characterized by a symmetric kernel $\k(m,m')$, modeling the ``rate'' at which two clusters of mass $m$ and $m'$ coalesce. When they do, a cluster of mass $m + m'$ is formed. Then, the evolution of the concentration $c_t(m)$ of clusters of mass $m$ is given by Smoluchowski's equation, actually an infinite system of nonlinear ODEs,
\begin{equation} \label{eq:smolunoarm} 
\dfdt{} c_t(m) = \frac12 \sum_{m'=1}^{m-1} \k(m,m') c_t(m') c_t(m-m') - \sum_{m'\geq 1} \k(m,m') c_t(m) c_t(m'),
\end{equation}
for $m \geq 1$.

The behavior of this equation depends heavily on the kernel $\k$. Three kernels are of particular interests, namely the constant, additive $\k(m,m') = m + m'$ and multiplicative $\k(m,m') = m m'$ kernels. For the first two kernels, the system is quite easy to solve, for instance by considering a PDE involving the generating function of $(c_t(m))$. For the multiplicative kernel, three facts have been known for a long time. First, there is a unique solution on $[0,T)$, where $T = (\sum_{m \geq 1} m^2 c_t(m))^{-1}$, and it is mass-conservative, in that $\sum_{m \geq 1} m c_t(m)$ remains constant \cite{McLeod}. Secondly, if there is actually a solution on a larger interval, then the mass has to decrease after time $T$. This phenomenon is called \textit{gelation}. Physically, it is interpreted as the appearance of {\em infinite-size clusters} accounting for a positive fraction of the total mass. Finally, Kokholm \cite{Kokholm} has shown, for monodisperse initial conditions $c_t(m) = \unn{m=1}$, that there exists a unique solution for all $t \geq 0$, given explicitly for $m \geq 1$ by
\begin{equation} \label{eq:solsmolunoarm}
c_t(m) =
\begin{cases}
m^{-1} B(t,m) & t \leq 1, \\
(mt)^{-1} B(1,m) & t \geq 1,
\end{cases}
\end{equation}
where
\[
B(\l,m) = (\l m)^{m-1} \frac{1}{m!} e^{- \l m}.
\]
The function $B(\l,\cdot)$ is a probability distribution for $\l \leq 1$, called the Borel distribution. It is the law of the total progeny of a Galton-Watson (GW) process with reproduction law Poisson($\l$), which we denote $PGW(\l)$.

This shows in particular that the total mass $\sum_{m \geq 1} m c_t(m)$ is conserved up to time 1, but afterwards decreases and is equal to $1/t$. Moreover, the proportion of particles which are in clusters of mass $m$ is exactly the probability for a $PGW(t)$ to have size $m$, when $t \leq 1$. On the other hand, when $t > 1$, the proportion of particles in clusters of mass $m$ matches the probability for a $PGW(1)$ to have size $m$. Since a $PGW(\l)$ is subcritical for $\l < 1$, and critical for $\l = 1$, this already suggests that the model exhibits a form of \emph{self-organized criticality} (SOC). Note also that the tail $\sum_{m \geq k} m c_t(m)$ decays exponentially with $k$ before gelation, but as $k^{-1/2}$ afterwards.

The proof of existence and uniqueness to Smoluchowski's equation with a multiplicative kernel and {\em general initial conditions} has only recently been given in \cite{NZ}, and independently in \cite{Rath}. In the latter, it is also shown (for initial conditions with bounded support) that SOC occurs, in the sense described above that the tails of the total mass decay as $k^{-1/2}$ after gelation.

SOC is well-known in the Physics literature, but is seldom proved mathematically. Some models are however known rigorously to exhibit SOC, such as Abelian sandpiles \cite{JaraiAAM}, invasion percolation on regular trees \cite{AngelIPRT}, and Curie-Weiss models \cite{CerfGorny}. 

Though the underlying idea in Smoluchowski's equation is essentially discrete, all that we just said applies to the continuous equation, which is the natural mean-field model to consider. To obtain a discrete picture, we introduced in \cite{MN0} a discrete model for Smoluchowski's equation, described as follows. Start with $N$ particles, and give independent exponential clocks with parameter $1/N$ to each pair of particles. When a clock rings, form the corresponding link between the two particles, {\em as long as they both belong to clusters with sizes below $\a(N)$}, where $(\a(N))$ is some threshold. This threshold should be thought of as when particles become too large and ``fall into the gel'', that is, become inactive. This model is directly inspired by the work of Fournier and Lauren\c{c}ot \cite{FournierMLP}, who first proved that, when the threshold satisfies
\begin{equation} \label{eq:alpha0}
1 \ll \a(N) \ll N, 
\end{equation}
then the concentration of particles of size $m$ converges to the solution \eqref{eq:solsmolunoarm} to Smoluchowski's equation \eqref{eq:smolunoarm}. To be precise, they only consider the concentrations as a jump Markov process, and prove convergence up to a subsequence, since  their paper is anterior to \cite{NZ} and \cite{Rath}. Their result is actually not limited to the multiplicative kernel $\k(m,m') = m m'$, but it is the only one which both exhibits gelation and is described by a truly discrete model as above. R\'ath in \cite{Rath} studies a slightly different yet closely related model, and obtains the same convergence result.

The threshold $(\a(N))$ above which clusters become inactive is fundamental. Informally, Smoluchowski's equation \eqref{eq:smolunoarm} shows no interaction with clusters of infinite mass, so that it should also be the case in a relevant discrete model. In particular, there is also a convergence result for the more simple discrete model with no threshold $(\a(N))$, i.e. when any cluster can interact with any other cluster, see \cite{FournierMLP}. The convergence is then to \emph{Flory's equation}, which is \eqref{eq:smolunoarm}, but with an extra term taking into account interactions with clusters of infinite mass. As it stems from a more simple discrete model, this equation is also easier to study than \eqref{eq:smolunoarm}, see e.g. \cite{NZ} and references therein.

Let us finally describe more precisely the results of \cite{MN0}. The main interest of that model, compared to that of Fournier and Lauren\c{c}ot, is that it retains much more information, in particular the graph structure, if we think of particles as vertices and coagulations as edges being created. Therefore, it allows to study the \emph{shape} of a typical cluster. Under stronger assumptions on the threshold, \cite{MN0} shows that the \emph{shape} a typical cluster converges to  that of a $PGW(t)$ tree for $t \leq 1$, and $PGW(1)$ for $t \geq 1$. In particular, this tree is subcritical before time 1, and critical afterwards, which is a far more striking way to explain SOC. As an aside, this also provides a better explanation for Formula \eqref{eq:solsmolunoarm}, and a way to recover it. Indeed, the concentration of clusters of mass $m$ is exactly the probability that a particle chosen uniformly at random is in an cluster of size $m$, times $1/m$, because of size-bias. From the result just mentioned, this probability converges to the probability that a $PGW(t \wedge 1)$ has size $m$, that is $B(t \wedge 1,m)$, in agreement with Formula \eqref{eq:solsmolunoarm}. See Sections \ref{sec:locconv} and \ref{sec:limconc} for similar reasonings and more details.

\subsection{Smoluchowski's equation with limited aggregations}
 
Physically, one may want to think of a coagulation involving two given particles as the creation of a covalent bond, so that it makes sense to assume that each particle can a priori only form a certain finite number of such bonds. Alternatively, there could be several types of particles, some being able to create more bonds than others. Therefore, it seems reasonable to initially give each particle a certain number of arms, representing the number of covalent bonds it may create. Two arms are used to create a bond, or link, and cannot be reused again. Then, a cluster is characterized by its mass $m$ and its number of free arms $a$, i.e. the total number of arms in that cluster which are not yet bound. This model {\em with limited aggregations} has first been considered by Bertoin in \cite{BertoinTSS}, and gives the evolution of the concentration $c_t(a,m)$ of clusters with $a$ free arms and mass $m$. It is given by an equation resembling Smoluchowski's equation \eqref{eq:smolunoarm}, which we now describe.

Let $S = \bN \times \bN^*$, and for $p=(a,m), p'=(a',m') \in S$,  define $p \cdot p' = a a'$ and $p \circ p' = (a+a'-2, m+m')$. We write $p' \lesssim p$ if $a' \leq a+1$ and $m' \leq m-1$. For $p' \lesssim p$, we define $p \bsl p' = (a+2-a', m-m')$. Then \emph{Smoluchowski's equation with limited aggregations} is given for $p \in S$ by
\begin{equation} \label{eq:smolu0}
\dfdt{} c_t(p)  = \frac12 \sum_{p' \lesssim p} p' \cdot (p \bsl p') c_t(p') c_t(p \bsl p') - \sum_{p' \in S}  p \cdot p' c_t(p) c_t(p').
\end{equation}
The first term in the RHS accounts for the appearance of $(a,m)$-clusters, by coagulation of $(a',m')$-clusters with $(a+2-a', m-m')$-clusters. The second term accounts for the disappearance of $(a,m)$-clusters by coagulation with any other cluster.  

In \cite{BertoinTSS}, Bertoin considers monodisperse initial conditions $c_0(a,m) = \mu'(a) \unn{m=1}$ for a measure\footnote{The prime is to differentiate Bertoin's convention and ours, where he normalizes $\mu$ to have unit moment, whereas we normalize it to be a probability, and we will then drop the prime.} $\mu'$ on $\bN$ with unit mean and a second moment, and such that $\mu'(0) + \mu'(2) \neq 1$. Denote $\g' = \sum_{a \geq 1} a (a-2) \mu'(a)$. Bertoin then shows that \eqref{eq:smolu0} has a unique solution up to time $\Tgel^0$, interpreted as the gelation time, with  
\begin{equation} \label{eq:Tgel0}
\Tgel^0 =
\begin{cases}
\pinf, & \text{if $\g \leq 0$} \\
1/\g', & \text{if $\g > 0$.}
\end{cases}
\end{equation}
More recently, Normand and Zambotti have shown in \cite{NZ}, using analytic techniques, that for general initial conditions (and in particular even when $\g' > 0$), Equation \eqref{eq:smolu0} has a unique solution on the whole of $\R^+$. They also provide explicit solutions; we will come back to this shortly.

One very interesting formula of \cite{BertoinTSS} is the following. Denote $\nu(m) = (m+1) \mu'(m+1)$, which is a probability measure on $\bN$. Then, provided $\g' \leq 0$, there are limits to the concentrations, given for $m \geq 2$ and $a \geq 0$ by
\begin{equation} \label{eq:cinf0}
\cinf(a,m) = \unn{a=0} \frac{1}{m(m-1)} \nu^{*m}(m-2),
\end{equation}
where $\nu^{*m}$ is the $m$-th convolution product of $\nu$. One part of this formula is clear: at the end, all the arms have been used up. To provide an explanation for $\cinf(0,m)$, Bertoin and Sidoravicius \cite{BerSid} study a microscopic model for these limiting concentrations. Without dwelling on technical details, they consider $N$ particles, with an empirical distribution of arms $\mu\N$ converging (weakly and in mean) to $\mu'/\mu'(\bN)$. They then consider a uniform pairing of these arms. Then, in the case when $\g' \leq 0$, they show that, when an arm is selected uniformly at random, the distribution of the cluster to which it belongs converges to that of GW tree started from 2 individuals, and with reproduction law $\nu$. This allows to give a probabilistic interpretation of Formula \eqref{eq:cinf0}, see \cite{BerSid}. We will provide another interpretation of \eqref{eq:cinf0} in Section \ref{sec:locconv}, based on the configuration model and which is more in the flavor of this work.

Let us now explain what happens when gelation does occurs, i.e. in the case when $\g' > 0$. Write $G_{\nu}$ for the generating function of $\nu$, and assume $\nu(0) > 0$. Then it is easy to see that the equation
\begin{equation} \label{eq:c}
c G_{\nu}'(c) = G_{\nu}(c)
\end{equation}
has a unique solution $c \in (0,1)$, and that $\b = 1 / G_{\nu}'(c) > 1$. In \cite{NZ}, it is then proved that, for all $m \geq 2$ and $a \geq 0$,
\begin{equation} \label{eq:cinf}
\cinf(a,m)= \unn{a=0} \frac{1}{m(m-1)} \b^{m-1} \nu^{*m}(m-2).
\end{equation}
This is obviously extremely similar to \eqref{eq:cinf0}, up to the factor $\b^{m-1}$, which seems to favor larger clusters. A question left open in \cite{NZ} is to give a probabilistic interpretation for this formula. We propose to answer this question, and more importantly, study in details a microscopic model for \eqref{eq:smolu0}, which, similarly to the model of \cite{MN0}, exhibits SOC.

A final result in \cite{BerSid} is a dynamic model for \eqref{eq:smolu0}, where each pair of free arms is linked at rate $1/N$. The authors then show that the concentrations of clusters with mass $m$ and $a$ free arms at time $t$ converges to the solution $c_t(a,m)$ of \eqref{eq:smolu0}. Of course, its final state corresponds to a uniform pairing of the arms, in accordance with the reasonings above. This is the natural discrete model underlying \eqref{eq:smolu0}, but, as should be clear from the previous section, it is not quite the right model when gelation occurs.

\subsection{Model}

Considering that we already explained the models of \cite{BerSid} and \cite{MN0}, the model that we are now going to present should not come as a surprise. So let us start initially with $N$ particles, each being equipped with a certain (deterministic) number of arms. We denote $\mu\N$ the \textbf{degree distribution}, that is, there are $N \mu\N(r)$ particles with $r$ arms for every $r \in \bN$. We also fix a threshold $\a(N)$. As time goes by, arms will bind to create clusters. We shall call a cluster \textbf{small} or \textbf{in solution} if it has size less than $\a(N)$, and \textbf{large}\footnote{We use \emph{large} in contrast to \emph{giant}, which is the term usually used in the Random Graphs literature to describe clusters of size of order $N$.} or \textbf{in the gel} if it has size $\a(N)$ or greater. Similarly, we will say that a particle or an arm is in solution if it belongs to a small cluster, and otherwise that it is in the gel. We also say that an arm is \textbf{free} if it is not linked to another arm, so that they are all free at time 0.

Now, each arm is given independently an exponential clock with parameter one $\cE(1)$. When a clock rings, the corresponding arm is made \textbf{active}. If, at this instant, there is another \emph{active} and \emph{free} arm \emph{in solution}, the two arms are linked together to form a bond between the corresponding particles (and are then no more free). Otherwise, nothing else happens.

Note that there can only be zero or one active arm in solution. Another way to see this model is to wait for an arm to activate, then for the next one, and bind them. Continue with the third and fourth activated arms, and so on, until \textbf{gelation} occurs, that is, a large cluster is created. This cluster then precipitates, or ``falls into the gel'', which is just a metaphor to mean that it becomes inactive, and it cannot bind to any more particle. The process continues similarly on the particles in solution, until another cluster falls into the gel, and so on.

The reader might notice two differences with the dynamic version of the model of Bertoin and Sidoravicius, which we explained above. First, of course, there is the appearance of the threshold $\a(N)$, which is fundamental, as for \eqref{eq:smolunoarm}. Without it we would get convergence to the corresponding Flory equation discussed in \cite{NZ}. On the other hand, we put clocks on arms, not on pair of arms. It is clear that the only difference it creates is a change of time, but the underlying Markov chain is the same, since, in both cases, a coagulation occurs by picking a pair of free arms in solution uniformly at random. We choose this model because the microscopic evolution is more clear. Indeed, to know what happens when a clock rings, we just need to check the state of one arm. Otherwise, we need to check the state of both arms in the link, which leads to some cumbersome complications.  As expected, and as we shall see in Section \ref{sec:smolu}, the convergence of the concentrations will be to a \emph{modified} Smoluchwoski equation with limited aggregations, which is just a time-change of \eqref{eq:smolu0}.

\subsection{Configuration model}

Consider the easier version of our model, where we also set i.i.d. $\cE(1)$ clocks on the arms, but there is no threshold $\a(N)$. Hence, an arm is activated when the corresponding clock rings, and at this time, if there exists another free active arm, we bind these two arms together. In other words, we bind the first and second activated arms, the third and fourth activated ones, and so on, up to time $t$. As we mentioned, this is not the right model for \eqref{eq:smolu0}, but it is nonetheless relevant, and it is also easier to study.

At the end of this process $t = \pinf$, we have thus bound all the arms (except maybe one), to obtain a random (multi-)graph. Notice that it would be exactly the same, in law, to successively choose a uniform pair of free arms until they are all picked, or to choose a pairing of the arms, uniform among all the possible pairings. This is the well-known \emph{configuration model} (CM) with degree distribution $\mu\N$, see e.g. \cite{BollobasRG,vdH}.

In this modified model, if we stop at some time $t$, then each arm is activated independently with probability $1 - e^{-t}$, but, conditionally on the activated arms, the order of activation is uniform. We thus still deal with a CM, but whose degree distribution is random. Note however that, by the law of large numbers, this degree distribution is essentially constant. The \emph{process} described above will thus be called a \emph{dynamical} CM (DCM).

There is an obvious coupling between our model and a DCM, by using the same exponential clocks, and these two processes are exactly equal up to gelation, that is, the time when a large cluster appears. Therefore, it seems reasonable that information on (dynamical) CM would give information on our model. For instance, if we know the size of the largest component in a DCM, then we know the gelation time in our model. This is described in more details in Section \ref{sec:firstgel}.

Explaining our results takes same time and preparation, and we shall delay this to Section \ref{sec:outline}. We first conclude this introduction with a short review of the literature, and the description of the plan of our paper.

\subsection{Related works}

\subsubsection{Frozen percolation on trees}

Aldous considers in \cite{AldousFP} a model of frozen percolation on a binary tree. Edges are given independent uniform on $[0,1]$ clocks, and when a clock rings, the edge appears if and only if both vertices belong to finite clusters. A rewording of the latter condition is that any cluster which becomes of infinite size is instantaneously frozen and can no longer interact with other clusters. 

Gelation happens when the first infinite cluster is formed, at time $1/2$. After proving existence of such process past $1/2$, Aldous computes in particular the rate at which the cluster containing a given edge or a given vertex becomes frozen. He further shows that at any time past gelation, conditionally given that an edge $e$ is present and its cluster is finite, clusters of the edge $e$ is a {\em critical} Galton-Watson tree with two ancestors. On the other hand, the distribution of infinite clusters is always that of an incipient infinite cluster (in this case a Galton-Watson with $\Bin(2,1/2)$ offspring distribution conditioned to be infinite).

\subsubsection{Forest-fire models}

R\'ath and T\'oth in \cite{RathToth} prove SOC for a model of forest fires on Erd\H{o}s-R\'enyi random graphs. In their model, edges between each pair of the $N$ particles appear (or re-appear) at rate $1/N$, while lightning strikes particles at a fixed rate $\a(N)/N$, with $(\a(N))$ satisfying Assumption \eqref{eq:alpha0} for the most interesting behavior. The effect of lightning is that edges in the corresponding connected components simply vanish (but vertices remain). The characterization of SOC in \cite{RathToth} is through the fact that past gelation, the limiting concentrations $(c_t(m), m \geq 1)$ satisfy $\sum_{m \geq k} m c_t(m) \sim C(t) k^{-1/2}$, with $C(t) > 0$ for $t \ge 1$; for $t < 1$, the decay is at an exponential rate.

Related to this model, R\'ath considers in \cite{Rath} a model of frozen percolation on the Erd\H{o}s-R\'enyi random graph. Edges between each pair of the $N$ particles appear at rate $1/N$, and vertices are hit at rate $\a(N)/N$. The effect of a vertex being hit is that its connected component becomes instantaneously frozen. 
R\'ath obtains convergence of the concentration of clusters of size $m$ towards the solution to Smoluchowski's equation with a multiplicative kernel. He further proves SOC for this model, in the same sense as \cite{RathToth} above, and obtains limit theorems for the remaining mass and the distribution of the size of the component of a typical vertex. 

Although the models of \cite{MN0} and \cite{Rath} present obvious similarities, they are also quite different: for instance the concentrations in the model of \cite{MN0} converge to the solution of Smoluchowski's equation with a multiplicative kernel, but the ones in \cite{RathToth} do not. Intuitively, this is due to the presence of additional dust created by the burning of clusters, which then also affects the growth of larger clusters. The characterization of SOC in these two papers is also quite different, since \cite{Rath} obtains SOC on the continuous model, while our goal in \cite{MN0}, as in the present paper, is to observe SOC on a \emph{microscopic} level.

\subsection{Organization of the paper}

Following this, Section \ref{sec:outline} is a precise outline of the paper, showing rigorous statements of our results and the main steps and ideas of the reasoning. We also introduce our assumptions and the main notation. For convenience, an index of notation is added in Appendix \ref{ap:indexnot}. Section \ref{sec:outline} contains in particular an explanation of the link between our model and the configuration model, and gives the main idea of how the latter helps study the former.

In Section \ref{sec:smolu}, we state and prove the main results concerning Equation \eqref{eq:smolu}. This section is essentially independent of the rest of the paper, and is merely meant to justify that we indeed get a model for Smoluchowski's equation with limited aggregations. In particular, the details can be skipped in a first reading.

We get back to our model in Section \ref{sec:prelim}. This is a part containing three types of preliminary results. First, that, for any time, there is a still a positive proportion of particles remaining in solution, what allows to use asymptotic results. Secondly, some fairly easy technical results which allow, thanks to our assumptions, to strengthen most convergence results that we get. Finally, we prove some combinatorial results describing the structure of our system at any time in terms of the configuration model, what allows to justify the comparisons between the two models.

The longest part of this article is by far Section \ref{sec:JL}, where we study the CM in the critical window, and prove an improved version of the result of Janson and Luczak. That section is focused on the CM and is entirely self-contained. It has therefore an interest of its own for readers who are more interested in random graphs. We also prove an extension of this result for the percolation on a CM.

These results are applied to \emph{dynamical} CM in Section \ref{sec:gelDCM}, specifically to study when, in a DCM, we observe the appearance of a cluster of size $\a(N)$. We also determine the properties of this cluster. This is a quite straightforward application of the previous section, but the results are slightly tedious to write, and require a lot of notation.

To clarify the link between our model and DCM over several gelation events, we introduce in Section \ref{sec:altmodel} an alternative model. It is built from independent DCM, more precisely from parts of DCM pasted together at the right times. We construct in this way a coupling of our model with another easier model, and we prove, using the results of the preceding section, that they coincide with high probability, at least on compact intervals.

This completes the most technical part of the work. What remains is to use this comparison, and the results on the DCM, to study our process. This is done in Section \ref{sec:limit}. We first prove tightness of our processes, and then show that subsequential limits have to solve a PDE. Despite the unusual nature of this PDE, it turns out to be exactly solvable. This allows to conclude in a classical fashion.

Finally, Section \ref{sec:conclusion} is a conclusion, where we first explain how to recover \eqref{eq:cinf}, and solve our initial problem. We also present some examples of initial distribution for which computations can be carried out explicitly. We conclude by discussing our assumptions and some open problems.

One last note to the reader: \cite{MN0} and the present paper obviously use similar ideas, but the presence of arms creates many complications, both in the reasoning and in the techniques (as the length of the papers can attest). We thus think that \cite{MN0} is a good preliminary reading.

%% file: 2-Outline.tex
\section{Outline and main results} \label{sec:outline}

The goal of this section is to introduce some notation, describe precisely our results, and finally get a flavor of the reasoning route that we will follow, by studying what happens when gelation occurs.

\subsection{Notation and assumption}

For convenience, we summarized all the notation that we use in Appendix \ref{ap:indexnot}. To begin with, for a measure $\mu$ on $\bN = \{0,1,2,\dots\}$ and $i \in \bN$, we let
\[
G_{\mu}(x) = \sum_{k \geq 0} \mu(k) x^k, \quad x \in [0,1],
\]
for its generating function, and 
\[
m_i^{\mu} = G_{\mu}^{(i)} (1) = \sum_{k \geq 0} k (k-1) \dots (k-i+1) \mu(k)
\]
for the $i$-th factorial moment of $\mu$. Weak convergence of measures is denoted by $\to$.

Recall that in our model, we start with $N$ particles indexed by $[N] = \{1,2,\dots,N\}$, and with degree distribution $\mu\N$. Clusters of size $< \a(N)$ are \textbf{small} (or \textbf{in solution}), those of size $\geq \a(N)$ are \textbf{large} (or \textbf{in the gel}). We set i.i.d. exponential clocks with parameter one on each arm, and bind successively the first and second activated \emph{in solution}, the third and fourth \emph{in solution}, and so on. We let $G(t)$ describe the graph created by the particles and their bonds at time $t$. In fact, we will always consider graphs (and trees) up to graph isomorphism, but this is a minor technical issue that we will not mention again.

For the sake of definiteness, we assume all our processes to be c\`adl\`ag. When we talk of the \textbf{degree} of a particle, we always mean its total number of arms (activated or not), which is therefore a constant in time for a given particle.

We will only be interested in particles in solution, so we let $G_S(t)$ be the restriction of $G(t)$ to the particles in solution at time $t$. Note that there are no bounds between $G_S(t)$ and its complement. We also define
\begin{itemize}
\item $N_t$, the number of particles in solution;
\item $n_t = N_t/N$, the concentration of particles in solution;
\item $P_t(k,r)$, the number of particles in solution with $k$ activated arms and degree $r$;
\item $p_t(k,r) = P_t(k,r)/N$, the concentration of particles that are in solution and with $k$ activated arms and degree $r$;
\item $\mu_t(r) = \sum_{k=0}^r P_t(k,r) / N_t$, the distribution of the degree of the particles in solution;
\item $\pi_t(k) = \sum_{r=0}^{\pinf} P_t(k,r) / N_t$, the distribution of the number of activated arms of the particles in solution.
\end{itemize}
With a slight abuse of notation, we have $\mu_0 = \mu$. Note that all these quantities are encoded in $(p_t)$, and that $\mu_t$ and $\pi_t$ are probabilities. When a large cluster appears, we witness a \textbf{gelation event}. The first one will hold a special r\^ole and will just be called \textbf{gelation}. We will denote
\begin{itemize}
\item $(\tau_i)_{i \geq 1}$, the time of the successive gelation events;
\item $\taugel = \tau_1$, the gelation time.
\end{itemize}
When necessary, we will add an exponent $(N)$ to these quantities to mean that we talk of the model on $N$ particles. Finally, we say that a sequence $(E_N)$ of events occurs \textbf{with high probability} (\textbf{w.h.p.}) if $P(E_N) \to 1$ as $N \to \pinf$.

We will need two types of assumptions. One concerns the degree distribution $\mu\N$, the other the threshold $(\a(N))$.

\begin{assume} \label{as:mu4}
There exists a probability $\mu$ such that $\mu\N \to \mu$ and $m_3^{\mu} > 0$. Moreover, there are constants $\eta, M \in (0, \pinf)$ such that
\[
\sup_{N \geq 1} \sum_{r \geq 0} \mu\N(r) r^{4+\eta} \leq M.
\]
\end{assume}

Notice that the second part of the assumption also ensures that $m_i^{\mu\N} \to m_i^{\mu}$ for $i = 1, 2, 3, 4$. It turns out that this assumption will suffice for results on the CM, but that we will need a bit more so as to study \emph{dynamical} CM, and our own model.

\begin{assume} \label{as:mu5}
There exists a probability $\mu \neq \dl_0$ such that $\mu\N \to \mu$ and $m_3^{\mu} > 0$. Moreover, there are constants $\eta, M \in (0, \pinf)$ such that
\[
\sup_{N \geq 1} \sum_{r \geq 0} \mu\N(r) r^{5+\eta} \leq M.
\]
\end{assume}

The only difference is of course the $5+\eta$-th moment assumption. Now, as we mentioned, the result of convergence to Smoluchowski's equation only requires a simple assumption on $(\a(N))$, namely \eqref{eq:alpha0}. As in \cite{MN0}, we will need more on $(\a(N))$ to be able to use results on the configuration model, and actually even more than the ``minimal assumption''
\begin{equation} \label{eq:alpha1}
N^{2/3} \ll \a(N) \ll N
\end{equation}
so as to be able to use CM results repeatedly.

\begin{assume} \label{as:alpha}
As $N \to \pinf$, 
\begin{equation} \label{eq:alpha}
N^{4/5} \ll \a(N) \ll N.
\end{equation}
\end{assume}

The careful reader will notice in the proofs that there is a balance between the assumptions on $(\mu\N)$ and on $(\a(N))$, in that stronger moment conditions allow to take smaller $\a(N)$. We did not try to optimize this.

\subsection{Configuration model} \label{sec:CM}

\subsubsection{Phase transition}

Consider $N$ vertices, with $N \pi\N(i)$ of them of degree $i$ (that is, endowed with $i$ arms, or half-edges). The probability $\pi\N$ is called the \textbf{degree distribution}. Performing a uniform pairing of all the arms (with eventually one half-edge left unpaired, if their total number is odd) gives a random graph, called the  \textbf{configuration model} $CM(N,\pi\N)$. Let us denote by $\cC_1(N,\pi\N)$ and $\cC_2(N,\pi\N)$ the largest and second largest components of the graph $CM(N,\pi\N)$ (chosen uniformly at random if there are several choices). We write $\cC_1\N$ and $\cC_2\N$ if there are no ambiguities. For a given component $\cC$, we let $v_k(\cC)$ be the number of vertices of degree $k$ in $\cC$, and $v(\cC) = \sum_{k \geq 0} v_k(\cC)$ its size.

The configuration model has been extensively studied since its introduction by Bollob\'as in \cite{BollobasCM}. There are in fact several variants for its definition, see for instance Chapters 7 and 10 of \cite{vdH} for a survey. In this paper we take the simplest definition as a uniformly chosen pairing of arms, thus allowing the possible presence of self-loops or multiple edges, and leading to an obvious (multi-)graph structure. We also allow one arm to be left unpaired when there is an odd total number of arms, without affecting the graph structure.

Denote $m\N_i = m_i^{\pi\N}$, and let us assume that $\pi\N$ converges weakly to some probability $\pi$, with $m_1\N \to m_1^{\pi}$ and $m_2\N \to m_2^{\pi}$. An important parameter for the CM is
\begin{equation} \label{eq:gamma}
\g^{\pi} := m_2^{\pi} - m_1^{\pi}.
\end{equation}
An abrupt transition between the subcritical regime $\g^{\pi} \leq 1$, where most vertices are in finite components, and a supercritical regime $\g^{\pi} > 1$, where a positive proportion of vertices belongs to a single giant component, was shown by Molloy and Reed in 1995 \cite{MRPoint}, up to some technical conditions, which were removed in \cite{JL}. We will be more precise shortly, but let us first start by explaining what happens at a microscopic level. This should enlighten the appearance of this parameter $\g$, as well as the following results.

For a measure $\pi \neq \dl_0$ on $\bN$, we will denote $\hat{\pi}$ the \textbf{size-biased, shifted by one} probability defined by
\begin{equation} \label{eq:hat}
\hat{\pi}(k) = \frac{(k+1)\pi(k+1)}{\sum_{i \geq 1} i \pi(i)}, \quad k \in \bN.
\end{equation}
To see why this is a relevant notion, look at a $CM(N,\pi)$ for large $N$. One way to build it is to select a vertex uniformly at random, take one of its half-edges $e$ (chosen in any way), then choose another half-edge $e'$ uniformly at random, and create the link between $e$ and $e'$. Then, select another free half-edge in this cluster (in any way), choose another free half-edge uniformly at random, bind them, and so on. Eventually, we may run out of half-edges in the cluster, and we may then start with another vertex. In this way, we then build a CM component-wise. This construction is actually at the heart of the proof of \cite{JL}, and will be used in Section \ref{sec:JL}.

Now, let us see what happens in this algorithm. First, we pick a vertex $v$ uniformly at random. Its number of half-edges has law $\pi\N \approx \pi$. Then, we pick an \emph{half-edge} uniformly at random, belonging to a vertex $v'$. Forgetting about the half-edge we already chose, it is easy to see that the number of half-edges of $v'$ is distributed according to $\hat{\pi}\N \approx \hat{\pi}$. The difference clearly lies in the fact that we choose a half-edge, rather than choosing a vertex. Repeating the result, we see that, as long as the cluster is not too big, we essentially always pick new vertices, and choose them according to $\hat{\pi}$. Therefore, the cluster we will see will be a tree, whose root has a number of children distributed according to $\pi$, and whose other vertices have a number of children distributed according to $\hat{\pi}$. In branching processes terminology, this is a \textbf{delayed Galton-Watson tree}, that we denote $\GW_{\pi,\hat{\pi}}$.

So what does it have to do with $\g$? Recall \cite{AtNey} that a Galton-Watson tree $\GW_{\pi}$ with reproduction law $\pi$ has extinction probability one if and only if $m_1^{\pi} \leq 1$. Therefore, it is easy to see that a $\GW_{\pi,\hat{\pi}}$ has extinction probability one if and only if $m_1^{\hat{\pi}} \leq 1$. But one readily computes $m_1^{\hat{\pi}} = m_2^{\pi}/m_1^{\pi}$, and therefore there is a.s. extinction if and only if $m_2^{\pi}/m_1^{\pi} \leq 1$, that is, $\g^{\pi} \leq 0$. Though this is informal, this explains the appearance of this quantity $\g^{\pi}$. For instance, a similar reasoning could be done for the Erd\H{o}s-R\'enyi graph.

In the next section, we will mention some results showing an even more striking relation with Galton-Watson trees, but let us first state the phase transition result.

\begin{theorem}[\cite{JansonLC,JL,MRPoint}] \label{th:CMgiant}
Assume that $\pi\N \to \pi$ with $\pi(1) > 0$ and
\[
\sup_{N \geq 1} \sum_{k \geq 0} \pi\N(k) k^2 < \pinf.
\]
Then
\begin{enumerate}
\item if $\g^{\pi} \leq 0$, then $v(\cC_1\N)/N \to 0$ in probability;
\item if $\g^{\pi} > 0$, then $v(\cC_1\N)/N$ converges in probability to the (positive) survival probability of a $\GW_{\pi,\hat{\pi}}$-tree.
\end{enumerate}
If moreover
\[
\sup_{N \geq 1} \sum_{k \geq 0} \pi\N(k) k^{4+\eta} < \pinf
\]
and $\g^{\pi} < 0$, then there is a constant $K$ such that $v(\cC_1\N) \leq K N^{1/(4+\eta)}$ w.h.p.
\end{theorem}

In other words, a giant component exist with positive probability if and only if $\g^{\pi} > 0$. If it does, its size is that of the survival probability of a $\GW_{\pi,\hat{\pi}}$-tree, similarly to the Erd\H{o}s-R\'enyi case \cite{JansonRG,vdH}. This is the classical result of \cite{MRPoint}, with the assumptions of \cite{JL}. These two papers do not make precise the size of the components when $\g^{\pi} < 0$. This is however done in \cite{JansonLC}, under some growth assumptions. Under the assumption done, corresponding to Assumption \ref{as:mu4}, it amounts to the statement above.

In parallel to the Galton-Watson terminology, we will therefore call the cases $\g^{\pi} < 0$, $\g^{\pi} = 0$ and $\g^{\pi} > 0$ respectively \emph{subcritical}, \emph{critical} and \emph{supercritical}. Notice in particular that $\Tgel^0$, defined in \eqref{eq:Tgel0}, is finite if and only if we are in the supercritical case. We will explain this in section \ref{sec:firstgel}.

\subsubsection{Local convergence} \label{sec:locconv}

We explained above what a typical component in a $CM(N,\pi\N)$ should look like, namely a $GW_{\pi,\hat{\pi}}$. This intuition can actually be made precise through the notion of \emph{local convergence}: essentially, a rooted graph converges locally to another rooted graph if they look the same at any fixed distance around the root. For probabilistic results, the natural extension is the concept of \emph{local weak convergence} popularized by Aldous, see e.g \cite{AldousSteele,BordenaveRG}. The following result is exactly Theorem 3.15 of \cite{BordenaveRG}, see also Proposition 2.5 of \cite{DemboMontanari}.

\begin{prop}{\cite{BordenaveRG,DemboMontanari}} \label{prop:locconv}
Assume that $\pi\N \to \pi$ with $m_2\N \to m_2^{\pi}$. Consider a vertex uniformly at random, and the component $\Ctyp\N$ rooted at this vertex. Then, in the sense of local weak convergence
\[
\Ctyp\N \to \GW_{\pi,\hat{\pi}}
\]
as $N \to \pinf$.
\end{prop}

This makes the intuition from the previous section rigorous: a typical component of a CM asymptotically looks like a delayed GW tree. In particular, in the subcritical case, these trees are finite a.s., and the distribution (as a graph) of a typical component converges to that of a $\GW_{\pi,\hat{\pi}}$ tree.

This result of weak local convergence allows to give a more direct explanation to this formula than the one given in \cite{BerSid} that we mentioned in the introduction. Indeed, $\cinf\N(0,m)$ is the concentration of clusters of size $m$ (let us forget about the possible unique free arm remaining) in a $CM(N,\pi\N)$. There is a size bias when one picks a \emph{vertex} uniformly at random, and thus
\[
\cinf\N(0,m) = \frac1m \Q \left ( v(\Ctyp\N) = m \right ),
\]
where $\Q$ is the probability corresponding to the choice of the root vertex of $\Ctyp\N$ (so the RHS above is still random). But under the assumptions of Proposition \ref{prop:locconv}, $\Ctyp\N \to \GW_{\pi,\hat{\pi}}$, and therefore
\[
\Q \left ( v(\Ctyp\N) = m \right ) \to \Q \left ( v(\GW_{\pi,\hat{\pi}} ) = m \right ) = \frac{1}{m-1} \hat{\pi}^{*m}(m-2),
\]
what explains Formula \eqref{eq:cinf0}. The last part of the above computation is a simple application of the celebrated formula of Dwass \cite{Dwass}.

We will use the same type of reasoning to explain Formula \eqref{eq:cinf} in Section \ref{sec:limconc}. However, in this case, we do not know what a typical cluster look like. This is what we will investigate in this article, as explained in the rest of this outline.

\subsubsection{The critical window}

The reader will have noticed that Theorem \ref{th:CMgiant} does not provide much information about the case $\g^{\pi} = 0$. Just as in the case of the Erd\H{o}s-R\'enyi graph, what happens in this \emph{critical window} is more subtle. So let us first state the result of \cite{JL} describing precisely what happens then. We let $\g_N = m_2\N - m_1\N$.

\begin{theorem}[\cite{JL}] \label{th:JL}
Suppose that Assumption \ref{as:mu4} holds for $\pi$. Assume moreover that, for $\g_N = m_2\N - m_1\N$,
\[
N^{-1/3} \ll \g_N \ll 1.
\]
Then, for every $\dl > 0$, the following hold w.h.p.
\begin{enumerate}
\item The size of the largest component verifies
\[
\left | v \left ( \cC_1\N \right ) - 2 \frac{m_1^{\pi}}{m_3^{\pi}} N \g_N \right | \leq \dl N \g_N.
\]
\item For every $k \geq 0$,
\[
\left | v_k \left ( \cC_1\N  \right ) - 2 \frac{k \pi(k)}{m_3^{\pi}} N \g_N \right | \leq \dl N \g_N.
\]
\item The size of the second largest component verifies
\[
v \left ( \cC_2\N \right ) \leq \dl N \g_N.
\]
\end{enumerate}
\end{theorem}

This result describes precisely what happens in the critical window. First, Points 1 and 3 show that there is a unique largest component, which has size about $2 N \g_N m_1^{\pi}/m_3^{\pi}$. Notice that this quantity is of an order strictly between $N^{2/3}$ and $N$, which is the reason for Assumption \ref{eq:alpha1}. The second point means that the vertices in that large component are chosen by a size-bias of their degree. This amounts to saying that, to find the vertices in that component, one chooses about $2 m_1^{\pi} N \g_N / m_3^{\pi}$ \emph{arms} uniformly at random, and picks the corresponding vertices.

\subsection{First gelation time} \label{sec:firstgel}

Let us explain how to use these results to study our model. We assume that the initial distribution of degrees $(\mu\N)$ verifies Assumption \ref{as:mu4} and that $(\a(N))$ verifies \eqref{eq:alpha1}. Recall that $G(t)$ is the graph induced by the particles at time $t$, along with their links. As we already mentioned in the introduction, there is another model coupled with ours that we can consider: link the first and second activated links, the third and fourth, and so on, until time $t$, to get a graph process $(\cG(t))$. In this other process, that we call a \textbf{dynamic configuration model} (DCM), there is no gel and solution, and any cluster can interact with any other cluster.

Recall that $\taugel$ is the gelation time, that is, when we first see a large component. It is clear that
\[
G(t) = \cG(t), \quad t \leq \taugel.
\]
In particular, $\taugel$ is exactly the time when a large component appears in the DCM $\cG$, and the size and properties of that component are the same for both models. Therefore, up to \emph{and including} $\taugel$, it suffices to study $\cG$.

Now, for some fixed time $t$, $\cG(t)$ is a CM on $N$ vertices, with a random degree distribution $\rho\N_t$, where $\rho\N_t$ corresponds to picking an arm or not with probability $p_t = 1 - e^{-t}$. By law of large numbers, it is easy to check that $\rho\N_t \to \rho_t$ a.s. with
\[
G_{\rho_t}(x) = G_{\mu} \left ( (1-e^{-t}) x + e^{-t} \right ),
\]
and even more, than Assumption \ref{as:mu4} holds a.s. for $(\rho_t\N)$. The generating function allows to compute
\[
m_i^{\rho_t} = p_t^i \; m_i^{\mu}, \quad \g^{\rho_t} = p_t \; (p_t \; m_2^{\mu} - m_1^{\mu}).
\]
Clearly, $\g^{\rho_t}$ is an increasing function of $t$ that vanishes at $\Tgel$ whenever it is finite, where
\begin{equation} \label{eq:Tgel}
\Tgel =
\begin{cases}
- \log \left( 1 - \frac{m_1^{\mu}}{m_2^{\mu}} \right) & \text{if $m_2^{\mu} > m_1^{\mu}$} \\
\pinf & \text{otherwise}.
\end{cases}
\end{equation}
Notice that $\Tgel$ is finite if and only if $\Tgel^0$, defined in \eqref{eq:Tgel0} is finite. We can then see different alternatives.

\paragraph{The subcritical case} The subcritical case corresponds to $\g^{\mu} = m_2^{\mu} - m_1^{\mu} < 0$ and is simple. Indeed, when all the arms are activated, the graph $\cG(\infty)$ is a $CM(N,\mu\N)$, which is subcritical. In particular, Theorem \ref{th:CMgiant} implies that its largest component has size $O(N^{1/(4+\eta)}) \ll \a(N)$, since we assume \eqref{eq:alpha1}. This means that w.h.p. there is no gelation event, and the final configuration of clusters is that of this configuration model.

\paragraph{The critical case} The critical case corresponds to $\g^{\mu} = 0$, and is not precise enough to decide whether or not there will be at least one gelation event (for that, one would at least need to know the leading order of $\g_N$ as $N \to \infty$, according to Theorem \ref{th:JL}). Nonetheless, at any time $t \geq 0$, $\g^{\rho_t} < 0$, and therefore the graph $\cG(t)$ is an asymptotically \emph{subcritical} CM. By the exact same reasoning as above, w.h.p. there is no gelation event up to time $t$, and the configuration of clusters at any time $t \ge 0$ is that of a configuration model with $N$ particles and degree distribution $\rho_t$.

Notice that the critical and subcritical cases mean that
\begin{itemize}
\item $\Tgel = \pinf$;
\item $\Tgel^0 = \pinf$, that is, there is no gelation in Smoluchowski's equation \eqref{eq:smolu0};
\item there is w.h.p. no gelation (at a finite time) in our discrete model.
\end{itemize}
In other words, $\taugel\N \to \Tgel = \pinf$. As expected, when $\Tgel < \pinf$, that is, there is gelation in \eqref{eq:smolu}, we will observe gelation in our discrete model as well, as we now explain.

\paragraph{The supercritical case} The supercritical case corresponds to $\g^{\mu} > 0$. Then $\g^{\rho_t} < 0$ for all $t < \Tgel$, and by a reasoning similar as above, there is w.h.p. no gelation event at a time $t < \Tgel$. On the other hand, having no gelation event by time $t > \Tgel$ means that a $CM(N,\rho\N_t)$ has no cluster of size greater than $\a(N) \ll N$. But then $\g^{\rho_t} > 0$, so by Theorem \ref{th:CMgiant}, this event has asymptotically vanishing probability. Hence, this simple argument tells us that
\begin{equation} \label{eq:limTgel}
\taugel\N \to \Tgel,
\end{equation}
provided again that \eqref{eq:alpha1} holds. Clearly, this even just requires that $N^{1/(4+\eta)} \ll \a(N) \ll N$.

Using Janson and Luczak's result in the critical window, Theorem \ref{th:JL}, we can in fact be even more precise in understanding typically when the first gelation event occurs, and what happens then. Let us indeed consider the time
\[
t_N = \Tgel + \frac{\a(N)}{N} \frac{m_3^{\mu}}{2 m_2^{\mu}(m_2^{\mu} - m_1^{\mu})}.
\]
Since $\a(N)/N \to 0$, it is easy to check that
\[
\g_N := \g^{\rho\N_{t_N}} = \frac{\a(N)}{N} \left ( \frac{m_3^{\mu}}{2 m_1^{\mu}} + o(1) \right ).
\]
Now, $(\a(N))$ verifies \eqref{eq:alpha1}, and thus the assumptions of Theorem \ref{th:JL} are verified. The largest component in a $CM(N,\rho\N_{t_N})$ has then size about
\[
2 N \frac{m_1^{\mu}}{m_3^{\mu}} \g_N \approx \a(N).
\]
Therefore
\begin{equation} \label{eq:Tgelexpansion}
\taugel\N = \Tgel + \frac{\a(N)}{N} \frac{m_3^{\mu}}{2 m_2^{\mu}(m_2^{\mu} - m_1^{\mu})} + o \left ( \frac{\a(N)}{N} \right ).
\end{equation}
Moreover, we see that the component that is created at this time $\Tgel\N$ has size exactly $\a(N) + o(\a(N))$, when we can \textit{a priori} only tell that it is between $\a(N)$ and $2 \a(N)$. Theorem \ref{th:JL} even allows to be more precise, and to show that, in that first gel cluster, the particles are chosen by a size-bias of their number of activated arms.

Of course, we are slightly cheating here; the result \eqref{eq:limTgel} holds rigorously, but \eqref{eq:Tgelexpansion} is a bit informal. To begin with, we did not make explicit what $o(\cdot)$ means. We also dealt quite liberally with the error terms. This could be fixed easily by taking times slightly before and after $t_n$, as we will do in Section \ref{sec:gelDCM}. A most important imprecision however, is that we used reasonings on the CM with a random degree distribution, without checking that the assumptions of Theorem \ref{th:JL} hold a.s. with respect to the law of the degree distribution. As will turn out, this will require slightly more to be made precise, to wit Assumption \ref{as:mu5}. We refer to Section \ref{sec:gelDCM} for details.

\subsection{Results}

As should be clear from the earlier description, there are three main protagonists in our play: Smoluchowski's equation with limited aggregations \eqref{eq:smolu0}; our discrete model; and the configuration model. We now describe our results concerning these objects.

\subsubsection{Modified Smoluchowski equation}

To begin with, we will be concerned with Smoluchowski's equation with limited aggregations \eqref{eq:smolu0}, or rather the slight modification that we now introduce. For $f : S \to \R^+$, we write $\la c_t, f(a,m) \ra = \sum_{(a,m) \in S} f(a,m) c_t(a,m)$ and the equation is
\begin{equation} \label{eq:smolu}
\dfdt{} c_t(p)  = \frac{1}{2 \la c_t, a \ra} \left ( \sum_{p' \lesssim p} p' \cdot p \bsl p' c_t(p') c_t(p \bsl p')  - \sum_{p' \in S}  p \cdot p' c_t(p) c_t(p') \right ),
\end{equation}
for $p \in S$. Two natural questions are then raised. First, is this equation well-posed? Secondly, what is its relevance to our model? The two following answers should not come as a surprise.

Let us consider initial conditions $c_0$, with
\[
m_1 = \la c_0, a \ra \in (0, \pinf), \quad m_2 = \la c_0, a(a-1) \ra \in [0,\pinf]
\]
and let
\begin{equation} \label{eq:Tgelsmolu}
\Tgel =
\begin{cases}
- \log \left( 1 - \frac{m_1}{m_2} \right) & \text{if $m_2 > m_1$} \\
\pinf & \text{otherwise}.
\end{cases}
\end{equation}
Clearly, for monodisperse initial conditions $c_0(a,m) = \mu(a) \unn{m=1}$ for a probability $\mu$ on $\bN$, then this is the same as the time $\Tgel$ defined in \eqref{eq:Tgel}. Since, we will always consider monodisperse initial conditions, we keep the same notation. The first result is the well-posedness of \eqref{eq:smolu}. The precise definition of a solution is given in Section \ref{sec:smolu}.

\begin{theorem} \label{th:wellposed}
Consider initial conditions $(c_0(p), p \in S)$ with $\la c_0, a \ra \in (0, \pinf)$, and define $\Tgel$ as in \eqref{eq:Tgel}. Then the modified Smoluchowski's equation with limited aggregations \eqref{eq:smolu} has a unique solution $c$ on $\R^+$. It enjoys the following properties. 
\begin{enumerate}
\item The mass $\la c_t, m \ra$ is continuous on $\R_+$, constant before time $\Tgel$, and decreasing afterwards. In other words, there is gelation at time $\Tgel$. 
\item The total number of arms $A_t = \la c_t ,a \ra$ is continuous, decreasing and remains positive on $\R^+$.
\item The ODE $s'(t) = A_{s(t)}, s(0) = 0$, has a unique solution on $\R_+$. This solution is a $C^1$-diffeomorphism of $\R^+$ and $c_{s(t)}$ is the solution to \eqref{eq:smolu0} with initial condition $c_0$. 
\item If $b$ is the solution to \eqref{eq:smolu0} with initial conditions $c_0$, then for $B_t := \la b_t, a \ra$, the ODE $s'(t) = \frac{1}{B_{s(t)}}$, $s(0)=0$, has a unique solution on $\R^+$. This solution is a $C^1$-diffeomorphism of $\R^+$ and $b_{s(t)}$ is the solution to \eqref{eq:smolu} with initial conditions $c_0$. 
\item The limiting concentrations $\cinf(p) = \lim_{t \to \infty} c_t(p)$ exist for all $p \in S$ and are the same as for \eqref{eq:smolu0}, that is, given by \eqref{eq:cinf}. 
\end{enumerate}
\end{theorem} 

This shows in particular that Equations \eqref{eq:smolu0} and \eqref{eq:smolu} are essentially the same, since they are related by a time-change. Now, as we mentioned, Equation \eqref{eq:smolu0} corresponds informally to a model where pairs of arms are bound at rate 1. Since in our model, we have to wait for two activations to bind a pair of arms, we should observe a time-change at the limit, which turns out to give Equation \eqref{eq:smolu} instead. Indeed, if we denote $c\N_t(a,m)$ for the concentration of clusters with $a$ free arms and mass $m$ in our model, then we have a result similar to that of \cite{FournierMLP}, provided we also assume \eqref{eq:alpha0}. As usual, $\D(\R^+,E)$ is the Skorokhod space of c\`adl\`ag functions on $\R^+$ with values in a Polish space $E$, see \cite{BillingsleyCPM,Kallenberg}.

\begin{theorem} \label{th:convconc}
Assume that \eqref{eq:alpha0} holds and that $\mu\N$ converges weakly to a probability $\mu$, with
\[
\sum_{a \geq 1} a \mu\N(a) \to \sum_{a \geq 1} a \mu(a) \in (0,\pinf).
\]
Then the process $(c\N)$ converges in distribution in $\D(\R^+,\ell^1(S))$ to the unique solution of Smoluchowski's equation \eqref{eq:smolu} started from monodisperse initial conditions $c_0(a,m) = \mu(a) \unn{m=1}$. 
\end{theorem}

This justifies part of the title of this paper, in that our model is indeed a microscopic model for (the modified) Smoluchowski equation with limited aggregations.

\subsubsection{Precise results in the critical window} \label{sec:critwindow}

Our second interest is the configuration model. We already saw in the Section \ref{sec:firstgel} that results on the CM can provide information on our model. More precisely, the divergence between our model and a DCM occurs when a large component appears. It therefore makes sense to study the largest components in a CM. One naturally expects that, in the supercritical case, more than one gelation event will occur, and actually that an order $\a(N)/N$ will. Hence, we will actually repeat the argument of Section \ref{sec:firstgel} over all these gelation events, so that we need to be more precise and control the probability with which they take place. Moreover, we will have to accurately control the degree of the vertices in these components. One of the weakness of the result of \cite{JL}, Theorem \ref{th:JL}, is that it is not at all quantitative and does not allow this. The strength of this result, however, is that the proof can be adapted to show a more precise statement. This is similar to what we do for the Erd\H{o}s-R\'enyi graph in \cite{MN}. 

We shall here consider a CM with parameters $n$ and $\pi\n$. The important difference with Theorem \ref{th:JL} is that we do not really want to see these quantities as being $n \to \pinf$, $\pi\n \to \pi$, but merely as being part of some ``compact'' set, and obtain results for fixed $n$ and $\pi\n$, with probabilities which are uniform for all the parameters in this set.

Let us introduce
\[
\phi_n(x) = \sum_{k \geq 0} v_k(\cC_1(n,\pi\n)) x^k, \quad w_n = \sum_{k \geq 0} k^3 v_k(\cC_1(n,\pi\n)),
\]
and the natural $\| \cdot \|_k$ norm on $C^k$ functions on $[0,1]$ or $[0,1]^2$ (which should be clear from the context) defined by
\begin{equation} \label{eq:Cknorm}
\|f\|_k = \sum_{i=0}^k \sup_{x \in [0,1]} \left | D^i f(x) \right |.
\end{equation}
Denote $G_n = G_{\pi\n}$ the generating function of $\pi\n$, $m\n_i = m_i^{\pi\n}$ its factorial moments, and let
\[
F_n(x) = x G_n'(x) = \sum_{k \geq 1} k \pi\n(k) x^k,
\]
which is, up to $m_1\n$, the generating function of $\hat{\pi}\n$. Finally, let
\[
\g_n = \g^{\pi\n} = m_2\n - m_1\n.
\]

To explain what our ``compact'' set is, we fix two positive sequences $(\eps_n^-)$ and $(\eps_n^+)$ with
\begin{equation} \label{eq:epsnpm}
n^{-1/4} \ll \eps_n^- \ll \eps_n^+ \ll 1,
\end{equation}
as well as three constants $\eta, m, M \in (0,\pinf)$. Our main assumptions will be that
\begin{equation} \label{eq:degreeass}
\eps_n^- \leq \g_n \leq \eps_n^+, \quad m \leq m\n_3, \quad \sum_{k \geq 1} \pi\n(k) k^{4 + \eta} \leq M.
\end{equation}

We may then prove the following result.

\begin{theorem} \label{th:JLimproved}
Consider a sequence $(\pi\n)$ such that \eqref{eq:degreeass} holds. Then, for any $\dl > 0$, there exists a constant $K > 0$, depending only on $\dl$, $\eta$, $m$, $M$ and $(\eps_n^{\pm})$ such that, with probability greater than
\[
1 - K \left ( \frac{1}{\sqrt{n \g_n^3}} + \g_n^{1+\eta} \right )
\]
we have
\[
\left ( 2 \frac{m\n_1}{m\n_3} - \dl \right ) n \g_n \leq v(\cC_1(n,\pi\n)) \leq \left ( 2 \frac{m\n_1}{m\n_3} + \dl \right ) n \g_n, \quad v(\cC_2(n,\pi\n)) \leq \dl n \g_n,
\]
and
\[
\left \|\frac{1}{n\g_n} \phi_n - \frac{2}{m\n_3} F_n \right \|_2 \leq \dl, \quad \frac1n \sum_{k \geq 0} k^3 v_k(\cC_1(n,\pi\n)) \leq \dl.
\]
\end{theorem}

The interpretation of this result is the same as for Theorem \ref{th:JL}. Let us however list the differences with the latter.
\begin{itemize}
\item We require a bit more on $\g_n$, namely that it is bounded below by $n^{-1/4}$, not just $n^{-1/3}$, see \eqref{eq:epsnpm}. There is a trade-off here between this assumption and the moments assumptions on $(\pi\n)$. For instance, if we only have $\eps_n^- \gg n^{-1/3}$, the result holds in the $\| \cdot \|_1$ sense, and we have thus an actual strengthening of the result of \cite{JL}. On the other hand, to get the result in the $\| \cdot \|_2$ sense while just assuming $\eps_n^- \gg n^{-1/3}$, the proof would require $6 + \eta$ moments. In any case, since we will anyway need to assume \eqref{eq:alpha}, this result will be sufficient for us.
\item The properties of the components are given for fixed $n$, not depending on the limiting quantities. Specifically, $\pi\n$ does not need to have a limit, as long as it does not become too degenerate.
\item The number of vertices of degree $k$ is controlled very precisely, since this bound in the $\| \cdot \|_2$ sense translates as
\[
\sum_{k \geq 0} k^2 \left | \frac{1}{n\g_n} v_k(\cC_1(n,\pi\n)) - \frac{2}{m\n_3} k \pi\n(k) \right | \leq \dl.
\]
\item The last part of the result allows to control the third moment, but with much less precision. It is however enough to say that the measure of degrees of the particles which are not in the largest component still have a positive third factorial moment.
\item The probability is uniform for a large set of parameters.
\end{itemize}

These considerations will turn out to be of great importance for several reasons. First, this result allows us to control precisely which vertices fall and which remain in solution after a gelation event. Secondly, we already emphasized that we want to employ this result repeatedly: we will use it not only for $N$, but for all $\eps N \leq n \leq N$ particles (what also explain the switch of notation from $N$ to $n$), as well as for different $\pi\n$, for which we cannot \emph{a priori} tell whether they converge or not. Finally, this result will be of interest to us when $\g_n$ is of order $\a(N)/N$ (see the analysis in Section \ref{sec:firstgel} above), so that we will use it about $N/\a(N)$ times. But in that case, we will get results with probability at least
\[
1 - K \frac{N}{\a(N)} \left ( \frac{1}{\sqrt{N (\a(N)/N)^3}} + \left ( \frac{\a(N)}{N} \right )^{1+\eta} \right ),
\]
and this goes to 1 exactly because we assume \eqref{eq:alpha}. This is the same reason why a stronger condition for the threshold $(\a(N))$ appears in \cite{MN0}.

\subsubsection{Controlling several gelation events: alternative model}

Our final final interest is of course our discrete model. At any given time, the particles along with their bounds naturally define a graph structure. Our main goal is to study the graph involving the particles \emph{in solution}. Note indeed that the way that we deal with the gel is somewhat arbitrary: it consists of a collection of clusters of size between $\a(N)$ and $2 \a(N)$, and we decide that these clusters are inert, but we could also decide to let them coalesce among themselves, what would not affect the dynamic. Their behavior is thus essentially irrelevant. More importantly, the interesting SOC occurs in solution.

With large probability, Theorem \ref{th:JLimproved} allows to control precisely what happens at the first gelation event. The informal idea of the proof is to be able to repeat this argument over all gelation events. To make the argument rigorous, we will {\em couple} our dynamics with that of an alternative model where, at all times, the configuration is that of an appropriate configuration model. 

Details are left to Section \ref{sec:altmodel}, but in short, the alternative model allows us to see our process as pieces of DCM pasted together, every time with different parameters. More precisely, we will condition at the configuration in solution after a gelation event, that is, $G_S(\tau_k)$. The first step is to know precisely what this configuration is. To this end, denote $S(t)$ the \textbf{set of particles in solution} at time $t$, and $B(t)$ the \textbf{total number of activated arms in solution} at $t$. For $S \subset [N] = \{1,2,\dots,N\}$ and $B \in \bN$, we define a configuration model $CM(S,B)$ as a uniform choice of $B$ arms of the particles in $S$, followed by a uniform pairing of these arms. Alternatively, we can choose a uniform sequence of $B$ arms and link the $2 i - 1$-th to the $2i$-th for $i = 1, \dots, \lfloor B /2 \rfloor$. We denote $CM'(S,B)$ the graph $CM(S,B)$ \emph{conditioned on having no large component}, that is we choose a sequence of $B$ arms, uniformly out of all the sequences such that the corresponding pairing does not give a large component. Be wary that this is different from choosing first $B$ arms, then a uniform ordering out of all those that give no giant component.

Finally, we say that a stopping time $\tau$ is a \emph{gelation stopping time} if, informally, it depends only on the particles in the gel. See Section \ref{sec:combin} for a precise definition. What is important is that all the gelation times $\tau_k$ are gelation stopping times. Our main structure result is the following.

\begin{lemma} \label{lem:combiB-int}
If $\tau$ is a gelation stopping time, then conditionally on $S(\tau)$ and $B(\tau)$, the configuration in solution $G_S(\tau)$ has the distribution of a $CM'(S(\tau),B(\tau))$.
\end{lemma}

We can now come back to our reasoning. We want to condition on the configuration $G_S(\tau_k)$ right after a gelation time $\tau_k$, and recognize this configuration as a DCM on $S(\tau_k)$. We therefore look at this DCM at the time when we see $B(\tau_k)$ activated arms, at which we exactly have an \emph{unconditioned} $CM(S(\tau_k), B(\tau_k))$. We have then three more steps to take.
\begin{enumerate}
\item Check that the conditioning has high probability. This can be done because, thanks to Theorem \ref{th:JLimproved}, we know very precisely the particles that remain in solution after a gelation.
\item Study what happens before the next gelation time $\tau_{k+1}$. To figure this out, the previous result allows us to consider our model as a DCM started from an initial configuration with the particles $S(\tau_k)$. We can then use directly Theorem \ref{th:JLimproved}, or actually the dynamic results of Section \ref{sec:gelDCM}.
\item Iterating this reasoning  allows to describe our model in between gelation events, and will yield the results that we explain in the next section.
\end{enumerate}

On a side note, there is a similar result to Lemma \ref{lem:combiB1} and \ref{lem:combiB2}, that will not be used in the bulk of the proofs, and which gives another description of the structure of our graph. Recall that $N_t$ is the number of particles in solution, and $\pi_t$ their empirical distribution of activated arms.

\begin{lemma} \label{lem:combipi-int}
Conditionally on $N_t$ and $\pi_t$, the configuration in solution $G_S(t)$ has the distribution of a $CM(N_t,\pi_t)$, conditioned on having no large cluster.
\end{lemma}

In this case, we deal with the classical CM that we explained before. The meaning of the conditioning is thus clear: we take a uniform pairing of all the arms, out of all the pairings that do not create a large component. We will come back to this result in Section \ref{sec:SOC} below.

\subsubsection{Tightness and limit}

The coupling with the alternative model has two major consequences. First, it allows us to precisely control the time between two gelation events. In particular, we will see that there exists $c>0$ such that, w.h.p.,
\[
\inf_{k \geq 0} \tau_{k+1}\N - \tau_k\N \geq c \frac{\a(N)}{N},
\]
at least for the $\tau_k$ in a compact interval. This easily ensures tightness of our processes, as we now state. For $d \geq 1$, the space $\ell^1_{\b}(\bN^d)$, is the space of sequences in $\bN^d$ with finite $\| \cdot \|_{\b}$-norm, where
\begin{equation} \label{eq:normbeta}
\| u \|_{\b} = \sum_{(k_1,\dots,k_d) \in \bN^d} (1 + k_1^{\b} + \dots + k_d^{\b}) | u(k,r) |,
\end{equation}
endowed with this norm. With no index $\| \cdot \| = \| \cdot \|_0$. Recall that $p\N_t(k,r)$ is the concentration of particles that are in solution at time $t$, with degree $r$ and $k$ activated arms.

\begin{lemma} \label{th:tightness}
Assume that Assumptions \ref{as:mu5} and \ref{as:alpha} hold. Then the sequence $(n\N)$ is tight in $\D(\R^+,\R)$, and any limiting point is locally Lipschitz-continuous. Moreover, for any $\b < 5 + \eta$, the sequence $(p\N)$ is tight in $\D(\R^+,\ell^1_{\b}(\bN^2))$, and any limit point is continuous.
\end{lemma}

This is the penultimate step towards the description of our system, and it only remains to figure out how these quantities $(n_t)$ and $(p_t)$ evolve. The second consequence of the alternative model is that it allows to control the characteristics of the particles in the components that fall into the gel --- this appeals repeatedly to Theorem \ref{th:JLimproved}, but with different parameters each time. We will find that their generating function solves an odd PDE, that is nonetheless fairly easy to solve, see Section \ref{sec:limit}. The result we end up getting is the following.

\begin{theorem} \label{th:limit}
Assume that Assumptions \ref{as:mu5} and \ref{as:alpha} hold, and denote $\nu = \hat{\mu}$. Then the equation
\begin{equation} \label{eq:Q}
(Q(t) - e^{-t}) G_{\nu}'(Q(t)) = G_{\nu}(Q(t))
\end{equation}
has a unique solution on $[\Tgel,\pinf)$. We denote $Q(t) = 1$ for $t < \Tgel$, and $Q$ is then continuous on $\R^+$.

Moreover, for any $\b < 5 + \eta$, $(n\N,p\N)$ converges in $\D(\R^+,\R) \times \D(\R^+,\ell^1_{\b}(\bN^2))$ to a limit $(n,p)$ where
\[
n_t = G_{\mu}(Q(t))
\]
and the generating function $\psi_t$ of $p_t$ is
\[
\psi_t(x,y) = G_{\mu} \left [ ((Q(t) - e^{-t})x + e^{-t})y \right ].
\]
\end{theorem}

We will explain some consequences of this result in the conclusion, Section \ref{sec:conclusion}, and compute some explicit examples. We will for now be content with explaining how this shows SOC in our model.

\subsubsection{Configuration in solution, self-organized criticality} \label{sec:SOC}

We now have all the keys in hand to be able to conclude to the graph structure in solution. On the one hand, Lemma \ref{lem:combipi-int} describes this structure. On the other hand, Theorem \ref{th:limit} gives us the parameters of the model. We essentially just have to deal with the conditioning in Lemma \ref{lem:combipi-int}, but the alternative model will show that it essentially does not matter. The technical details are left to Section \ref{sec:lastres}.

In any case, Theorem \ref{th:limit} ensures that $\pi\N_t \to \pi_t$, where $\pi_t$ has generating function 
\begin{equation} \label{eq:pit}
G_{\pi_t}(x) = \frac{G_{\mu} \left ( (Q(t) - e^{-t})x + e^{-t} \right )}{G_{\mu}(Q(t))}.
\end{equation}
Thanks to Lemma \ref{lem:combipi-int}, we see that to get a sample of our model at time $t$, we can essentially just pick $n_t N$ particles, give them an i.i.d. number of arms according to $\pi_t$, and then create a uniform pairing of the arms. This is a subcritical configuration model for $t < \Tgel$, and a critical one for $t \geq \Tgel$. Check indeed that, by definition of $\pi_t$,
\[
m_2^{\pi_t} = (1 - e^{-t})^2 m_2^{\mu} < (1 - e^{-t}) m_1^{\mu} = m_1^{\mu},
\]
for $t < \Tgel$, and that, for $t \geq \Tgel$,
\[
m_2^{\pi_t} = (Q(t) - e^{-t})^2 G_{\mu}''(Q(t)) = (Q(t) - e^{-t}) G'_{\mu}(Q(t)) = m_1^{\pi_t},
\]
by definition of $Q(t)$. This already proves that our model exhibits some form of SOC. However, it is somewhat awkward to deal with CM with random parameters. For instance, we may want to say that our configuration is close to a critical CM (with deterministic parameters), but this is not really a well-documented notion.

A nice way around this issue is to use the concept of local convergence that we mentioned in Section \ref{sec:locconv}. We can then give the asymptotic distribution of a typical cluster in solution. This is a direct consequence of the alternative model and of the local convergence result, Proposition \ref{prop:locconv}.
 
\begin{theorem} \label{th:typcluster} 
Assume that Assumptions \ref{as:mu5} and \ref{as:alpha} hold. Consider a vertex uniformly at random at some time $t$, and the component $\Ctyp\N(t)$ rooted at this vertex. Then, in the sense of local weak convergence
\[
\Ctyp\N(t) \to \GW_{\pi_t,\hat{\pi}_t}
\]
as $N \to \pinf$.
\end{theorem} 

Since $\hat{\pi}_t$ is critical or subcritical, these trees are finite, and this could be rephrased by saying that, for any finite rooted tree $\cT$,
\begin{equation} \label{eq:typcluster}
\lim_{N \to \pinf} \P \left ( \Ctyp\N(t) = \cT \right ) = \P \left ( \GW_{\pi_t,\hat{\pi}_t} = \cT \right ).
\end{equation}
This Galton-Watson tree is subcritical for $t < \Tgel$, and critical for $t \geq \Tgel$. This shows even more clearly that our model exhibits SOC on a microscopic scale. This is similar to results in \cite{MN0}. As a last step, and somehow \textit{en passant}, we will be able to easily explain the intriguing Formula \eqref{eq:cinf}, answering at last the question of \cite{NZ}; see Section \ref{sec:limconc}.

%% file: 3-Smolu.tex
\section{Smoluchowski's equation} \label{sec:smolu}

The results concerning Equation \eqref{eq:smolu} are essentially already proven in the literature, though with small differences. This part of the work is different from the remaining of the paper, and is merely here to justify that we indeed obtain a model for Smoluchowski's equation. Therefore, we deal with it here once and for all.

\subsection{Well-posedness}

We recall that (the modified) Smoluchowski's equation (with limited aggregations) is written in \eqref{eq:smolu}, and that $\Tgel$ is defined in \eqref{eq:Tgelsmolu}. The well-posedness result is Theorem \ref{th:wellposed}. We just need to make clear what we mean by a solution to \eqref{eq:smolu}.

\begin{defn} \label{def:solsmolu}
We say that a family of nonnegative continuous functions $(c_t(p), p \in S)$ is a solution to \eqref{eq:smolu} with initial conditions $c_0 \in [0, \pinf)^S$ if, for every $t \geq 0$,
\begin{enumerate}[label=(\roman*)]
\item $0 < \inf_{s \in [0,t]} \la c_s, a \ra \leq \sup_{s \in [0,t]} \la c_s, a \ra < \pinf$;
\item $\int_0^t \la c_s, a \ra^2 \ds < \pinf$;
\item for every $p \in S$,
\[
c_t(p) - c_t(0) = \frac12 \int_0^t \frac{1}{\la c_s, a \ra} \left ( \sum_{p' \lesssim p} p' \cdot (p \bsl p') c_s(p') c_s(p \bsl p') - \sum_{p' \in S}  p \cdot p' c_s(p) c_s(p') \right ) \ds.
\]
\end{enumerate} 
\end{defn}

Note that conditions (i) and (ii) ensure that the integral in (iii) is well-defined. Up to the unsurprising condition (i), this is the definition of a solution given in \cite{NZ}. Notice also that the former has an extra condition in the definition of a solution, but it is unnecessary, as was explained for the usual Smoluchowski equation in \cite{MN0}.

We will not prove Theorem \ref{th:wellposed}. Indeed, due to the similarity between equations \eqref{eq:smolu0} and \eqref{eq:smolu}, the proof of the existence and uniqueness of the solution is exactly the same as in \cite{NZ}, as are Points 1 and 2 of the result. Points 3 and 4 show that both equations are just a time-change of one another, and can be checked in a straightforward way. Finally, Point 5 is a direct consequence of this time-change. The analysis of \cite{NZ} also presents a representation formula for the solution to \eqref{eq:smolu0}. The same can be done for \ref{eq:smolu}. However, this gives rise to quite cumbersome formulae, which are not of interest to us.

It is worth comparing the time $\Tgel^0$ defined in \eqref{eq:Tgel0}, to the time $\Tgel$ defined in \eqref{eq:Tgel}. We already mentioned that $\Tgel^0 < \pinf$ if and only if $\Tgel < \pinf$, which is expected as $\Tgel^0$ is the gelation time for \eqref{eq:smolu0}, whereas $\Tgel$ is the gelation time for the time-changed \eqref{eq:smolu}. Moreover, one can explicitly compute that they are related by the time-change described in Theorem \ref{th:wellposed}. In our case of interest, i.e. monodisperse initial conditions $c_0(a,m) = \mu(a) \unn{m=1}$, we also have that $m_1$ and $m_2$ are just the two first factorial moments $m_1^{\mu}$ and $m_2^{\mu}$ of $\mu$, what makes sense in terms of configuration model, as we explained in Section \ref{sec:firstgel}.

\subsection{Convergence to Smoluchowski's equation}

Recall that $c\N_t(a,m)$ is the concentration of clusters with $a$ free arms and mass $m$ in our model, that is, there are $N c\N_t(a,m)$ such clusters at time $t$. Similarly to the result of Fournier and Lauren\c{c}ot \cite{FournierMLP} mentioned above, the only assumption needed for convergence is \eqref{eq:alpha0}. The convergence result is Theorem \ref{th:convconc}.

This result justifies that our model is indeed a discrete model for Smoluchwoski's equation with limited aggregations. Notice also that our model is defined in such a way that the $c_0$ are monodisperse, but it would be easy to modify it to get any initial conditions, and a similar result would hold. However, different initial conditions \emph{do not} allow us to carry out the analysis that will lead to the other results, and which are the main interest of this paper.

We shall now quickly prove Theorem \ref{th:convconc}. The ideas are from \cite{FournierMLP}, already used in \cite{MN0}. There are however two main differences here.
\begin{itemize}
\item There are two parameters $a$ and $m$, and the coagulations depend on $a$, whereas the threshold is a function of $m$. This only creates minor issues.
\item Since one has to wait for two arms to be activated to see a coagulation, the process $(c\N)$ is no more Markovian. The time between two jumps is indeed the sum of two independent exponential times and therefore not exponential. The process is then called \emph{semi-Markov}. Once again, this nonetheless allows to write martingales, and only minor technical issues occur. For the necessary background on semi-Markov processes, see \cite{SemiMarkovBook}.
\end{itemize}

\begin{proof}[Proof of Theorem \ref{th:convconc}]
Let
\[
A\N_t = \sum_{m = 1}^{\a(N) - 1} \sum_{a=1}^{\pinf} a c\N_t(a,m)
\]
so $N A\N_t$ is the total number of arms in solution. We always assume in the following computations that there are at least two arms in solution (in particular $A\N_t \geq 2/N$), else nothing happens. These concentrations evolve as follows. First, one has to wait an exponential time with parameter $A\N_t$ to see an arm activated. Then, another arm is activated after an independent exponential time with parameter $A\N_t-1$, and a coagulation occurs. A coagulation occurs between two different $p$- and $p'$-clusters in solution, with probability
\[
\frac{2 p \cdot p' c\N_t(p) c\N_t(p')}{A\N_t(A\N_t - 1/N)}
\]
if $p \neq p'$, and with probability
\[
\frac{p \cdot p \: c\N_t(p)^2}{A\N_t(A\N_t - 1/N)} - \frac1N \frac{p \cdot p \: c\N_t(p)}{A\N_t (A\N_t - 1/N)}
\]
for $p = p'$. With these probabilities, the new concentrations are $c\N_t + \frac1N \Dl_{p,p'}$, where
\[
\begin{cases}
\Dl_{p,p'}(p) = \Dl_{p,p'}(p') = - 1 & \text{if $p \neq p'$,} \\
\Dl_{p,p'}(p) = - 2 & \text{if $p = p'$,} \\
\Dl_{p,p'}(p \circ p') = + 1.\\
\end{cases}
\]
The only thing that we are missing is when the two arms are on the same $p$-cluster, what happens with probability
\[
\frac1N \frac{p \cdot (p-1) c\N_t(p)}{A\N_t (A\N_t - 1)},
\]
where $p-k := p - (k,0)$ for $k \in \bN$. In this case, the new concentrations are $c\N_t + \frac1N \Dl_p$, with
\[
\Dl_p(p) = - 1, \quad \Dl_{p}(p-2) = +1.
\]

Let $E = \ell^1(S)$ and $I_N = \bN \times \{1, \dots, \a(N) - 1 \}$. At any time, the system in then in a state $\eta \in E$. For $\eta \in E$, we define
\[
A\N_{\eta} = \sum_{m=1}^{\a(N) - 1} \sum_{a \geq 1} a \eta(a,m).
\]
For any bounded measurable mapping $f : \E \to \R$, and taking into account symmetries, let us define the generator of $(c_t\N)$ as
\[
\begin{split}
L\N f(\eta) & = \left ( \frac{1}{N  A\N_{\eta}} + \frac{1}{N A\N_{\eta} - 1} \right )^{-1} \times \\
& \quad \left [ \sum_{p,p' \in I_N^2} \frac{p \cdot p' \eta(p) \eta(p')}{A\N_{\eta} (A\N_{\eta} - 1/N)} \left ( f \left ( \eta + \frac1N \Dl_{p,p'} \right ) - f(\eta) \right ) \right. \\
& \quad - \frac1N \sum_{p \in I_N} \frac{p \cdot p \: \eta(p)}{A\N_{\eta} (A\N_{\eta} - 1/N)} \left ( f \left ( \eta + \frac1N \Dl_{p,p} \right ) - f(\eta) \right ) \\
& \quad + \left. \frac1N  \sum_{p \in I_N} \frac{p \cdot (p-1) \eta(p)}{A\N_{\eta} (A\N_{\eta} - 1/N)} \left ( f \left ( \eta + \frac1N \Dl_p \right ) - f(\eta) \right ) \right ] \\
& = \frac{1}{A\N_{\eta}  -1/(2N)} \left [ \frac12 \sum_{(p,p') \in I_N^2} p \cdot p' \eta(p) \eta(p') N \left ( f \left ( \eta + \frac1N \Dl_{p,p'} \right ) - f(\eta) \right ) \right. \\
& \quad - \frac12 \sum_{p \in I_n} p \cdot p \: \eta(p) \left ( f \left ( \eta + \frac1N \Dl_{p,p} \right ) - f(\eta) \right ) \\
& \quad \left. + \frac12 p \cdot (p - 1) \left ( f \left ( \eta + \frac1N \Dl_p \right ) - f(\eta) \right ) \right ].
\end{split}
\]
As we mentioned, $(c\N_t)$ is not a Markov process since the time between two jumps is not exponential, but it is however still true \cite{SemiMarkovBook} that, for every bounded $f : E \to \R$, the process
\[
\MNf_t = f(c\N_t) - f(c\N_0) - \int_0^t L\N f (c\N_s) \ds
\]
is a martingale, with quadratic variation
\[
\la \MNf_t \ra = \int_0^t \left ( L\N f^2(c\N_s) - 2 f(c\N_s) L\N f (c\N_s) \right ) \ds.
\]

Similarly to \cite{MN0,FournierMLP}, we can see, by taking
\[
f(\eta) = \sum_{a,m \geq 1} (a \wedge b) \eta(a,m)
\]
for some $b > 0$, that
\begin{equation} \label{eq:boundconv}
\E \left ( \int_0^t \frac{1}{\An_s - 1/(2N)} \left ( \sum_{a \geq b} \sum_{m=1}^{\a(N) - 1} a c\N_s(a,m) \right )^2 \ds \right ) \leq K \left ( \frac1b + \frac{\a(N)}{N} \right )
\end{equation}
for some constant $K$ depending only on the initial conditions. This is the fundamental inequality used in the proof, and the rest is just a perusal of the arguments of \cite{MN0,FournierMLP}: first, tightness is easy to obtain, since there is an order $\dl N$ of jumps of size of order $1/N$ on intervals of size $\dl$; then, limits are shown to satisfy \eqref{eq:smolu} by using \eqref{eq:boundconv} and the martingale $\MNf_t$ for linear $f$.
\end{proof}

%% file: 4-Prelim.tex
\section{Preliminary results} \label{sec:prelim}

This section presents some easy preliminary results. The first is that, for any time, there remains a positive fraction of the particles in solution w.h.p., what allows to use asymptotic results. The second series of results allow to strengthen convergence in $\ell^1$ to convergence in $\ell^1_{\b}$, when we have an assumption such as Assumption \ref{as:mu4} or \ref{as:mu5}.

\subsection{Positive concentration}

Denote $\cM_1$ the set of probability measures on $\bN$, and $\cM_1\N$ the subset of $\mu \in \cM_1$ such that $N \mu(k) \in \bN$ for all $k \in \bN$. We consider $\tK\N_{c,m,M}$ the set of all couples
\[
(n,\mu) \in \bigcup_{n \geq 1} \{ n \} \times \cM_1\n
\]
such that
\begin{equation} \label{eq:tKN}
c N \leq n \leq N, \quad m_3^{\mu} \geq m, \quad \sum_{r \geq 0} r^{5+\eta} \mu(r) \leq M,
\end{equation}
and $\K\N_{c,m,M}$ the set of $(n,\mu) \in \tK\N_{c,m,M}$ such that additionally
\begin{equation} \label{eq:KN}
\frac{m^{\mu}_2}{m^{\mu}_1} \geq 1 + m.
\end{equation}
We can a priori ensure that the concentration and empirical measure of degrees in solution remains in $\tK\N_{c,m,M}$, at least on bounded intervals. However, we will also need to show that it remains in $\K\N_{c,m,M}$, and this will require to precisely study how the degrees in solution evolve. Typically, assume that half of the particles have degree 3 at time 0, and half have degree 2. We could imagine that, we are first going to remove the particles of degree 3. When this is done, the empirical measure of degrees in solution will definitely not verify \eqref{eq:KN}. It turns out that this is not how the system evolves, but this will require more work. However, as we mentioned, we have the following result, where we recall that $N_t$ is the number of particles in solution at time $t$, and $\mu_t$ their empirical distribution of degrees.

\begin{lemma} \label{lem:alwaystKN}
Assume that Assumption \ref{as:mu5} holds. Then, for any $T > 0$, there exists $c, m, M \in (0,\pinf)$ such that,
\[
\P \left ( \forall t \in [0,T] \; (N_t, \mu\N_t) \in \tK\N_{c,m,M} \right ) \to 1.
\]
\end{lemma}

\begin{proof}
Since $\mu\N \to \mu$ with $m^{\mu}_3 > 0$, then, by moment assumption, for some $\dl > 0$ and some $k \geq 3$, $\mu\N(k) \geq \dl$ for all $N$ large enough. Denote $X_i\N$ the indicator function that no link on particle $i$ has been activated by time $T$, and $I\N$ the set of particles with $k$ arms. If $i \in I\N$, then
\[
\P(X\N_i = 1) = (1-e^{-t})^k =: p.
\]
Since the $X\N_i$ are independent, by Chebyshev's inequality,
\[
\P \left ( \left | \frac{1}{\# I\N} \sum_{i \in I\N} X_i - p \right | > p/2 \right ) \leq \frac{4}{(\#I \N)^2 p^2} \sum_{i \in I\N} \Var(X\N_i) \leq \frac{4}{\# I\N p}.
\]
Since $\# I\N = N \mu\N(k) \geq \dl N$, then this last term goes to $0$ as $N \to \pinf$. Hence, w.h.p., at least $p \# I\N / 2 \geq p \dl N / 2 := c N$ particles with $k$ arms have no activated arms, and thus are still in solution. Hence, for all $t \in [0,T]$, $n_t \geq c$ and $\mu_t(k) \geq c N / N_t \geq c$. Therefore
\[
\sum_{r \geq 0} r^{5+\eta} \mu_t(r) \leq \frac{N}{N_t} \sum_{r \geq 0} r^{5+\eta} \mu(r) \leq \frac{M}{c},
\]
and
\[
m_3^{\mu_t} \geq k(k-1)(k-2) \mu_t(k) \geq 6 c,
\]
which shows the result.
\end{proof}

\subsection{Stronger convergence}

The next result ensures that all the weak convergence of measures that we will consider can be extended to a far stronger convergence. Recall the definition of $\ell^1_{\b}$ and $\| \cdot \|_{\b}$ from \eqref{eq:normbeta}. We do the proofs in the case $d = 1$ for simplicity, but the extension to $d \geq 2$ is straightforward.

\begin{lemma} \label{lem:strongconvseq}
Assume that $(u\N)$ is a sequence in $\ell^1$ such that $u\N \to u$ in $\ell^1$ and
\[
\sup_{N \geq 1} \| u \N \|_{\g} < \pinf
\]
fr some $\g > 1$. Then, for any $\b \in [1,\g)$, $u\N \to u$ in $\ell^1_{\b}$.
\end{lemma}

\begin{proof}
Let $p \in (1,\pinf)$ such that $\b p < \g$, and let $q \in (1,\pinf)$ such that $p^{-1} + q^{-1} = 1$. Then H\"older's inequality implies that
\[
\sum_{r \geq 0} r^{\b} \left | u\N(r) - u(r) \right | \leq \left ( \sum_{r \geq 0} r^{\b p} \left | u\N(r) - u(r) \right |  \right )^{1/p} \left ( \sum_{r \geq 0} \left | u\N(r) - u(r) \right |  \right )^{1/q} 
\]
and this tends to 0 by the assumptions.
\end{proof}

This can be extended to convergence of $\ell^1$-valued functions as follows.

\begin{lemma} \label{lem:strongconvfunc}
Assume that $(u\N)$ is a sequence in $\D(\R^+,\ell^1)$ such that $u\N \to u$ for some $u \in \D(\R^+,\ell^1)$, and that for some $\g > 1$ and for all $T > 0$,
\[
\sup_{N \geq 1} \sup_{t \in [0,T]} \|u\N_t \|_{\g} < \pinf.
\]
Then, for any $\b \in [1,\g)$, $u\N \to u$ in $\D(\R^+,\ell^1_{\b})$.
\end{lemma}

\begin{proof}
It suffices to recall the definition of the Skorokhod topology \cite{BillingsleyCPM,Kallenberg}: for some sequence $(\l_N)$ of time-changes,
\[
\sup_{s \leq T} \left | \l_N(s) - s \right | + \sup_{s \leq T} \left \| u\N_{\l_N(s)} - u_s \right \| \to 0
\]
for all $T \geq 0$. It is then easy to conclude similarly as in Lemma \ref{lem:strongconvseq}.
\end{proof}

Finally, this readily entails that, for processes verifying this $\g$-th moment assumption, we only need to check usual convergence (or tightness) to get a stronger convergence (or tightness).

\begin{lemma} \label{lem:strongconvproc}
Assume that $(u\N)$ is a process in $\D(\R^+,\ell^1)$ and that, for all $T \geq 0$, there is a $M_T \in (0,\pinf)$ such that
\[
\P \left ( \sup_{t \in [0,T]} \|u\N_t \|_{\g} \leq M_T \right ) \to 1.
\]
Then the following hold.
\begin{enumerate}
\item If $u\N \to u$ in distribution in $\D(\R^+,\ell^1)$, then $u\N \to u$ in distribution in $\D(\R^+,\ell^1_{\b})$ for any $\b \in [0,\g)$.
\item If $(u\N)$ is tight in $\D(\R^+,\ell^1)$, then it is tight in $\D(\R^+,\ell^1_{\b})$ for any $\b \in [0,\g)$.
\end{enumerate}
\end{lemma}

\begin{proof}
The first part follows directly from Lemma \ref{lem:strongconvfunc}, for instance by using Skorokhod's embedding to assume that $u\N \to u$ a.s. The second part means that every subsequence of $(u\N)$ has a subsequence converging in distribution. But then, such a  converging sub-subsequence also converges in distribution in $\D(\R^+,\ell^1_{\b})$ by the first part of the result, which is exactly what we had to check.
\end{proof}

\subsection{Combinatorial results} \label{sec:combin}

Recall that $S(t)$ is the set of particles in solution at time $t$, and that $B(t)$ is the total number of activated arms at this time. A configuration model $CM(S,B)$ is defined by picking a uniform sequence of $B$ arms of the particles in $S$ and joining the first and the second, the third and the fourth, and so on. The conditioned graph $CM'(S,B)$ is this graph conditioned on having no large component. Precisely, it is given by the pairing corresponding to a sequence of $B$ arms in $S$, chosen uniformly among all the sequences such that the corresponding pairing creates no large component. The main combinatorial result is the following.

\begin{lemma} \label{lem:combiB1}
For any $t \geq 0$, conditionally on $S(t)$ and $B(t)$, the configuration in solution $G_S(t)$ has the distribution of a $CM'(S(t),B(t))$.
\end{lemma}

\begin{proof}
Denote by $\s(t)$ the (ordered) sequence of arms activated up to time $t$ (whether they are in solution or not), and $|\s(t)|$ its length, that is, the total number of activated arms. Define $S(\s)$ to be the particles in solution when we activate the arms in the order $\s$ and perform our algorithm, that is, bind the first and second arm in solution, the third and fourth in solution, and so on. For $S \subset [N]$, let also $\s_S(t)$ be the sequence of arms activated on particles of $S$.

For $S \subset [N]$ and $k \geq 0$, denote $L_{S,k}$ the ordered sequences $\ell = (\ell_1, \dots, \ell_k)$ of $k$ distinct arms in $S$, and $L'_{S,k}$ those sequences such that linking $\ell_1$ and $\ell_2$, $\ell_3$ and $\ell_4$ and so on, does not create a large component. Our goal is exactly to show that
\[
\P( S(t) = S, \s_S(t) = \ell)
\]
is independent of $\ell \in L'_{S,k}$.

To begin with, we can split this event according to the arms that are activated in $\bar{S} = [N] \bsl S$, by writing
\begin{align*}
& \P ( S(t) = S, \s_{S}(t) = \ell)\\
& = \sum_{r \geq 0} \sum_{\bar{\ell} \in L_{\bar{S},r}} \P( S(t) = S, \s_S(t) = \ell, \s_{\bar{S}}(t) = \bar{\ell}) \\
& = \sum_{r \geq 0} \sum_{\bar{\ell} \in L_{\bar{S},r}} \P(|\s(t)| = k+r) \P \left ( S(t) = S, \s_S(t) = \ell, \s_{\bar{S}}(t) = \bar{\ell} \middle | |\s(t)| = k+r \right ).
\end{align*}
Now, as we already mentioned, conditionally on $|\s(t)|$, the configuration is given by a uniform ordering of $|\s(t)|$ arms, i.e. $\s(t)$ is uniform in $L_{[N],|\s(t)|}$. Therefore, for $\bar{\ell} \in L_{\bar{S},r}$,
\[
\begin{split}
\P \left ( S(t) = S, \s_S(t) = \ell, \s_{\bar{S}}(t) = \bar{\ell} \right. & \left. \vphantom{\s_{\bar{S}}(t) = \bar{\ell}} \middle | |\s(t)| = k+r \right ) \\
& = \frac{\# \{ \s \in L_{[N],k+r}, \s_S = \ell, \s_{\bar{S}} = \bar{\ell}, S(\s) = S \}}{\# \{ \s \in L_{[N],k+r}, \s_S = \ell, \s_{\bar{S}} = \bar{\ell} \}}.
\end{split}
\]
The denominator clearly only depends on $k$ and $r$. Therefore, it suffices to show that the numerator is independent of $\ell$. So consider another $\ell' \in L'_{S,k}$. There is a clear bijection
\[
\left \{ \s \in L_{[N],k+r}, \s_S = \ell, \s_{\bar{S}} = \bar{\ell}, S(\s) = S \right \} \lra \left \{ \s \in L_{[N],k+r}, \s_S = \ell', \s_{\bar{S}} = \bar{\ell}, S(\s) = S \right \}.
\]
Indeed, take an element $\s$ of the former set, and replace $\ell_i$ by $\ell'_i$ for $i = 1, \dots, k$. This obviously provides a $\s'$ such that $\s'_S = \ell', \s'_{\bar{S}} = \bar{\ell}$, and minute of reflection also allows to check that $S(\s') = S$. There is an obvious inverse mapping, and we have thus a bijection between these two sets, what allows to conclude.
\end{proof}

As in \cite{MN0}, it is possible to prove an extension of this result. We first define a gelation stopping time as in \cite{MN0}. For any $t \ge 0$, the natural filtration of our model is  
\[
\cF_t = \sigma \left ( e_a \unn{e_a \leq t}, a \in A \right ),
\]
where $A$ is the set of arms and $e_a$ the clock on arm $a$. Similarly, for a subset $S$ of $[N]$, we define the filtration generated by clocks attached to particles in $S$ 
\[
\cF^S_t := \sigma \left ( e_a \unn{e_a \leq t}, a \in A_S \right ),
\]
where $A_S$ are the clocks on the arms of the particles of $S$.

\begin{defn}
We say that $\tau$ is a {\em gelation stopping time} if
\begin{itemize} 
\item $\tau$ is a $(\cF_t)_{t \ge 0}$-stopping time,
\item for any $S \subset [N]$ and $t \ge 0$, conditionally on $\{S(t)=S\}$, $\tau \unn{\tau \le t}$ is independent of $\cF_t^S$. 
\end{itemize} 
\end{defn} 

Two important examples of a gelation stopping time are 
\begin{itemize}
\item any given deterministic time $t \geq 0$. 
\item the $k$-th gelation time $\tau_k$ for a given $k \in \bN$. 
\end{itemize} 

The fact that the gelation times $\tau_k$ are gelation stopping times is easy to see. Indeed, conditionally on $\{S(t)=S\}$, $\tau_k\unn{\tau_k \leq t}$ is determined by the $k$-th gelation time in $(\cG_{\bar{S}}(s))_{s \leq t}$, and it is therefore independent of $\cF_t^S$. 

On the other hand, for instance, the first time after $\tau_k$ that a component of size at least $\a(N)/2$ is created is a $(\cF_t)$-stopping time, but is {\em not} a gelation stopping time.

Similarly to what is done in \cite{MN0}, the previous result extends to gelation stopping times. The proof is similar to that of \cite{MN0}. 

\begin{lemma} \label{lem:combiB2}
If $\tau$ is a gelation stopping time, then conditionally on $S(\tau)$ and $B(\tau)$, the configuration in solution $G_S(\tau)$ has the distribution of a $CM'(S(\tau),B(\tau))$.
\end{lemma}

Instead of considering the particles in solution themselves, we can take a lighter conditioning on the number of particles in solution and their degree distribution. We thus define a $CM(M,\mu,B)$ by taking $M$ vertices with degree distribution $\mu$ and then taking the pairing given by a uniform choice of a sequence of $B$ arms. As above, $CM'(M,\mu,B)$ is a $CM(M,\mu,B)$ conditioned on having no large component. By summing the previous result over all the $S$ with the same cardinality and the same measure of degrees, and by exchangeability, we get the following.

\begin{lemma} \label{lem:combiB3}
If $\tau$ is a gelation stopping time, then conditionally on $N_{\tau}$ and $\mu_{\tau}$, the configuration in solution $G_S(\tau)$ has the distribution of a $CM'(N_{\tau},\mu_{\tau},B(\tau))$.
\end{lemma}

Conditioning on the total number of activated arms in solution will turn out to be the most useful to us; see in particular the description of the alternative model in Section \ref{sec:altmodel}. However, we can also condition on $\pi_t$, the empirical distribution of activated arms in solution, to get a more natural result directly related to the CM. The proof is again done in a similar fashion.

\begin{lemma} \label{lem:combipi}
If $\tau$ is a gelation stopping time, then conditionally on $N_{\tau}$ and $\pi_{\tau}$, the configuration in solution $G_S(\tau)$ has the distribution of a $CM(N_{\tau},\pi_{\tau})$, conditioned on having no large cluster.
\end{lemma}

This more natural result will only be useful to us when we want to study typical clusters in solution and use Proposition \ref{prop:locconv}, see Section \ref{sec:lastres}.

%% file: 5-JL.tex
\section{Largest component of a slightly supercritical CM} \label{sec:JL}

This long section is devoted to proving Theorem \ref{th:JLimproved}, and providing some extensions to vertices with ``types'', which should be thought of as the number of free arms of our particles. This section is purely about random graphs, and does not use any specifics of our model. We recall that the notation and the statement of the main result are in Section \ref{sec:CM}.

\subsection{Proof of Theorem \ref{th:JLimproved}}

We outline the argument of \cite{JL} for proving Theorem \ref{th:JLimproved}, and point where it can be slightly improved. 
The beautiful idea in \cite{JL} is to perform simultaneously the random uniform pairings of half-edges and the exploration process of the components, in the following way. Color\footnote{In their article, Janson and Luczak rather label vertices as sleeping and awake, and half-edges as sleeping, active and dead. We use colors here as to not confuse the exploration with the dynamics of our model.} the vertices as white or red and the half-edges as white, yellow or black; white and yellow half-edges are also called bright. Start with all vertices and half-edges white. Pairing one edge with another will be decided by giving the half-edges i.i.d. random maximal ``lifetimes'' $\tau_x$ with distribution $\cE(1)$: namely, each half-edge spontaneously turns black with rate $1$ (unless it was colored black earlier). 

\begin{itemize}
\item[\textbf{C1}] If there are yellow half-edges, go to C2. If there are no yellow half-edges (as in the beginning), select a white vertex by choosing it uniformly at random\footnote{Janson and Luczak rather choose a white half-edge uniformly, but this only changes the order in which we build the components: size-biased here, biased by their number of half-edges in \cite{JL}. This slight modification has the advantage of simplifying a bit the end of the argument.}, color it red and all its half-edges yellow. 
If the chosen vertex has no half-edges, repeat C1, otherwise go to C2.  
\item[\textbf{C2}] Pick a yellow half-edge $e$ (which one does not matter) and color it black.
\item[\textbf{C3}] Wait until the next half-edge $e'$ turns black (spontaneously), and join it with $e$ to form an edge of the graph. 
If the vertex $v$ to which $e'$ belongs is white, change $v$ to red and all other half-edges from $v$ (but $e'$) yellow. Repeat from C1.
\end{itemize}

Let $W(t)$, $Y(t)$, $B(t) = W(t) + Y(t)$ be the number of white, yellow, and bright half-edges at time $t$. We naturally define these processes as c\`adl\`ag. 

As observed in \cite{JL}, this algorithm allows to construct the configuration model compo\-nent-wise, and components are created each time C1 is performed. Note also that the times C1 is performed exactly corresponds to the times where $Y$ cancels. Thus, the exploration of the successive components are performed exactly during the successive excursions of $Y$ above the origin.

\begin{remark} \label{rk:sizebias}
On the other hand, given a graph $\cG$, this algorithm also allows us to explore $\cG$ component by component. Consider indeed that each edge is made up of two half-edges, and perform the same algorithm, up to two slight differences: forget about the exponential clocks, and replace C3 by
\begin{itemize}
\item[\textbf{C3'}] Take $e'$ the other half-edge attached to $e$. If the vertex to which $e'$ belongs is white, change this vertex to red and all other half-edges from this vertex (but $e'$) yellow. Repeat from C1.
\end{itemize}
It is clear that, if $\cG$ is random and obtained by the configuration model, then the sequence of components \emph{explored} by this algorithm has the same distribution as the sequence of components \emph{constructed} by the original algorithm. Since each exploration of a new component is started by choosing a vertex uniformly at random, then the order in which we build the components is actually \emph{biased by their size}.
\end{remark}

Let us fix $t_0 > 0$. In the whole proof, $C$ denotes a constant depending only on $t_0$ and on the fixed parameters $\delta, M, m, \eps_n^{\pm}$, and which may change from line to line.

\paragraph{1.} 
Observe that $(B(t))$ is a death-process starting from $n m\n_1 - 1$. $B(t)$ is decreased by 2 at rate $B(t)$ (except if there remains only one bright half-edge). Lemma 6.2 in \cite{JL} ensures that
\begin{equation} \label{eq:B}
\E \left [ \sup_{t \leq \g_n t_0} \left | B(t) - n m\n_1 \exp(-2t) \right |^2 \right ] \leq C n \g_n.
\end{equation}
Now, let $\cV_k(t)$ be the number of vertices with degree $k \ge 1$ having all their half-edges with lifetimes $\tau_x > t$. Again it is a death process, decreasing by 1 at rate $k \cV_k(t)$, starting from $n \pi\n(k)$ (or $n \pi\n(k) - 1$ for one of them), and as in (6.5) of \cite{JL}, we have
\begin{equation} \label{eq:cVk}
\E \left [ \sup_{t \leq \g_n t_0} \left | \cV_k(t) - n \pi\n(k) \exp(- k t) \right |^2 \right ] \leq C k n \g_n \pi\n(k), \quad k \leq 1 / \g_n.
\end{equation}
The interest in $\cV_k(t)$ is that, at least at the beginning of the construction, it is a very decent approximation of $V_k(t)$, the number of vertices with degree $k \ge 1$ that are still white at time $t$. Better yet, $\cW(t)= \sum_{k \geq 0} k \cV_k(t)$ is also a very decent approximation for the number of white half-edges $W(t) = \sum_{k \ge 0} k V_k(t)$, again at least at the beginning of the exploration. More precisely, one should observe that the difference can only come from half-edges which have been colored at the first step C1 of the algorithm, therefore $\cW$ and $W$ should remain close as long as we have not explored too many components.

Slightly precising the argument of the proof of Lemma 6.3 in \cite{JL}, we find that 
\begin{align*}
& \E \left [ \sup_{t \leq \gamma_n t_0} \left | \cW(t)- n \sum_{k \geq 1} k \pi\n(k) \exp(-kt) \right | \right ] \\
& \leq \sum_{k = 1}^{1/\g_n} k \E \left [ \sup_{t \le \gamma_n t_0} \left | \cV_k(t) - n \pi\n(k) \exp(- k t) \right | \right ]  + \sum_{k > 1/\g_n} k E \left [ \sup_{t \le \gamma_n t_0} \left | \cV_k(t) - n \pi\n(k) \exp(- k t) \right | \right ] \\
& \leq \sum_{k = 1}^{1/\g_n} k \sqrt{C k n \g_n \pi\n(k)} + \sum_{k > 1/\g_n} k n \pi\n(k) \\
& \leq C \sqrt{n \g_n} \sum_{k = 1}^{1/\g_n} \sqrt{k^3 \pi\n(k)} + n \sum_{k > 1/\g_n} k \pi\n(k) \\
& \leq C \sqrt{n \g_n} \left ( \sum_{k = 1}^{1/\g_n} k^{4+\eta} \pi\n(k) \right )^{1/2} \left ( \sum _{k = 1}^{1/\g_n} k^{-1-\eta} \right )^{1/2} + n \g_n^{3 + \eta} \sum_{k > 1/\g_n} k^{4 + \eta} \pi\n(k),
\end{align*}
where, at the second line, we use \eqref{eq:cVk} and Cauchy-Schwarz for $k \leq 1/\g_n$, as well as the trivial bound $\cV_k(t) \leq n \pi\n(k)$ for $k > 1/\g_n$. Finally, our moments assumptions (\ref{eq:degreeass}) allow to conclude that
\begin{equation} \label{eq:tW}
E \left [ \sup_{t \leq \gamma_n t_0} \left | \cW(t)- n \sum_{k \geq 1} k \pi\n(k) \exp(-kt) \right | \right ] \leq C ( \sqrt{n \g_n} + n \g_n^{3 + \eta}).
\end{equation}
In fact, from the proof, it even holds that
\begin{equation} \label{eq:sumkcVk}
E \left [ \sum_{k \geq 1} k \sup_{t \leq \gamma_n t_0} \left | \cV_k(t)- n  \pi\n(k) \exp(-kt) \right | \right ] \leq C ( \sqrt{n \g_n} + n \g_n^{3 + \eta}),
\end{equation}
which will turn out to be useful later on.

\paragraph{2.}
Let now $\cY(t) = B(t) - \cW(t)$, which should be a good approximation to $Y(t) = B(t)-W(t)$, at least at the beginning of the exploration, and define $h_n(t) = m\n_1 e^{-2t} - \sum_{k \geq 1} k \pi\n(k) e^{-k t}$. Then \eqref{eq:B} and \eqref{eq:tW} readily imply that
\begin{equation} \label{eq:tY}
\E \left [ \sup_{t \leq t_0} \left | \frac1n \g_n^{-2} \cY(\g_n t) - \g_n^{-2} h_n(\g_n t) \right | \right ] \leq C \left ( \frac{1}{\sqrt{n \g_n^3}} + \g_n^{1+ \eta} \right ).
\end{equation}
Note that $h_n(t) = F_n(1) e^{-2t} - F_n(e^{-t})$. Since $\g_n \leq \eps_n^+ \to 0$, we may compute the following, where the constants hidden in the $O(\cdot)$ depend only on $(\eps_n^+)$ and $M$, and are uniform in $t \in [0,t_0]$:
\begin{align*}
h_n(\g_n t) & = F_n(1) e^{-2 \g_n t} - F_n(e^{-\g_n t}) \\
& = F_n(1) (1 - 2 \g_n t + 2 \g_n^2 t^2 + O(\g_n^3)) - \left [ F_n(1) + F_n'(1) (e^{- \g_n t} - 1) \right. \\
& \qquad \qquad \qquad \qquad \qquad +{} \frac12 F_n''(1) (e^{- \g_n t} - 1)^2 + \left. O(\g_n^3) \right ] \\
& = F_n(1) ( 1 - 2 \g_n t + 2 \g_n^2 t^2 + O(\g_n^3)) - \left [ F_n(1) + F_n'(1) \left ( - \g_n t + \frac12 \g_n^2 t^2 + O(\g_n^3) \right ) \right. \\
& \qquad \qquad \qquad \qquad \qquad + \left. \frac12 F_n''(1) (- \g_n t + O (\g_n^2))^2 + O(\g_n^3) \right ] \\
& = \g_n t \left ( - 2 m\n_1 + m\n_2 + m\n_1 \right ) \\
& \qquad + \g_n^2 t^2 \left ( 2 m\n_1 - \frac12 \left ( m\n_1 + m\n_2 \right ) - \frac12 \left ( m\n_3 + 2 m\n_2 \right ) \right ) + O (\g_n^3) \\
& = \g_n^2 t + \g_n^2 t^2 \left ( - \frac12 m\n_3 - \frac32 \g_n \right ) + O(\g_n^3) \\
& = \g_n^2 t \left ( 1 - \frac12 m\n_3 t \right ) + O (\g_n^3),
\end{align*}
where we use the definition of the $m\n_i$ and $\g_n = m\n_2 - m\n_1$. In short, we get
\begin{equation} \label{eq:hn}
\sup_{t \leq t_0} \left | \g_n^{-2} h_n(\g_n t) - t \left ( 1 - \frac12 m \n_3 t \right )  \right | \leq C \g_n.
\end{equation}
Note that \eqref{eq:hn} is a purely deterministic statement.

\paragraph{3.}
Now, observe that $V_k(t) \leq \cV_k(t)$, so that $W(t) \leq \cW(t)$, and that moreover, $\cW - W$  can only increase when C1 occurs. Take such a time $t$ when some half-edges are colored due to C1 occurring. Then $Y(t^-) = 0$ and $Y(t)$ is the number of vertices of the vertex chosen at C1, so that $Y(t) \leq \max \{i, \pi\n(i) \neq 0\} \leq (Mn)^{1/(4+\eta)} \leq C n \g_n^3$, where the first inequality comes from \eqref{eq:degreeass}, and the second from the fact that $\eps_n^- \gg n^{-1/4}$. We deduce that 
\[
\cW(t) - W(t) = \cW(t) - B(t) + Y(t) \leq \cW(t) - B(t) + C n \g_n^3.
\]
At other times, $\cW - W$ decreases, thus
\[
\cW(t) - W(t) \leq \sup_{s \leq t} (\cW(t) - B(t)) + C n \g_n^3 = - \inf_{s \leq t} \cY(s) + C n \g_n^3.
\]
For the same reason, up to another factor $\max \{i, \pi\n(i) \neq 0\} \leq C n \g_n^3$, this even actually holds until $t^+$, the first time strictly after $t$ when C1 is performed. To conclude, for all $t \geq 0$,
\begin{equation} \label{eq:error}
0 \leq \sup_{s \leq t^+} \left ( \cW(s) - W(s) \right ) = \sup_{s \leq t^+} \left ( Y(s) - \cY(s) \right ) \leq - \inf_{s \leq t} \cY(s) + C n \g_n^3.
\end{equation}

\paragraph{4.}
Fix now $\dl > 0$ and $\dl' = \dl/(2M)$. Define $t_n = 2/m\n_3$, and notice that $P_n(t) = t(1- t m\n_3/2)$ is positive on $(0,t_n)$ and vanishes on the boundary. For $\eps > 0$ small enough, it is clear that the parabola $P_n$ may be included in a tube of height $\eps$ as in Figure \ref{fig:parabola}. In other words, with $P_n^{\pm} = P_n \pm \eps$, $P_n^-$ vanishes at two points, one on $[0,\dl'/2]$ and one on $[t_n - \dl'/2,t_n]$, whereas $P_n^+ \geq 0$ on $[0,t_n]$ and $P_n^+(t_n+\dl'/2) < 0$. Moreover, one can choose $\eps$ uniform as long as $t_n$ is in a compact interval of $(0,\pinf)$, what holds from our assumptions.

\begin{figure}[htb]
\centering
\includegraphics[width=0.5 \columnwidth]{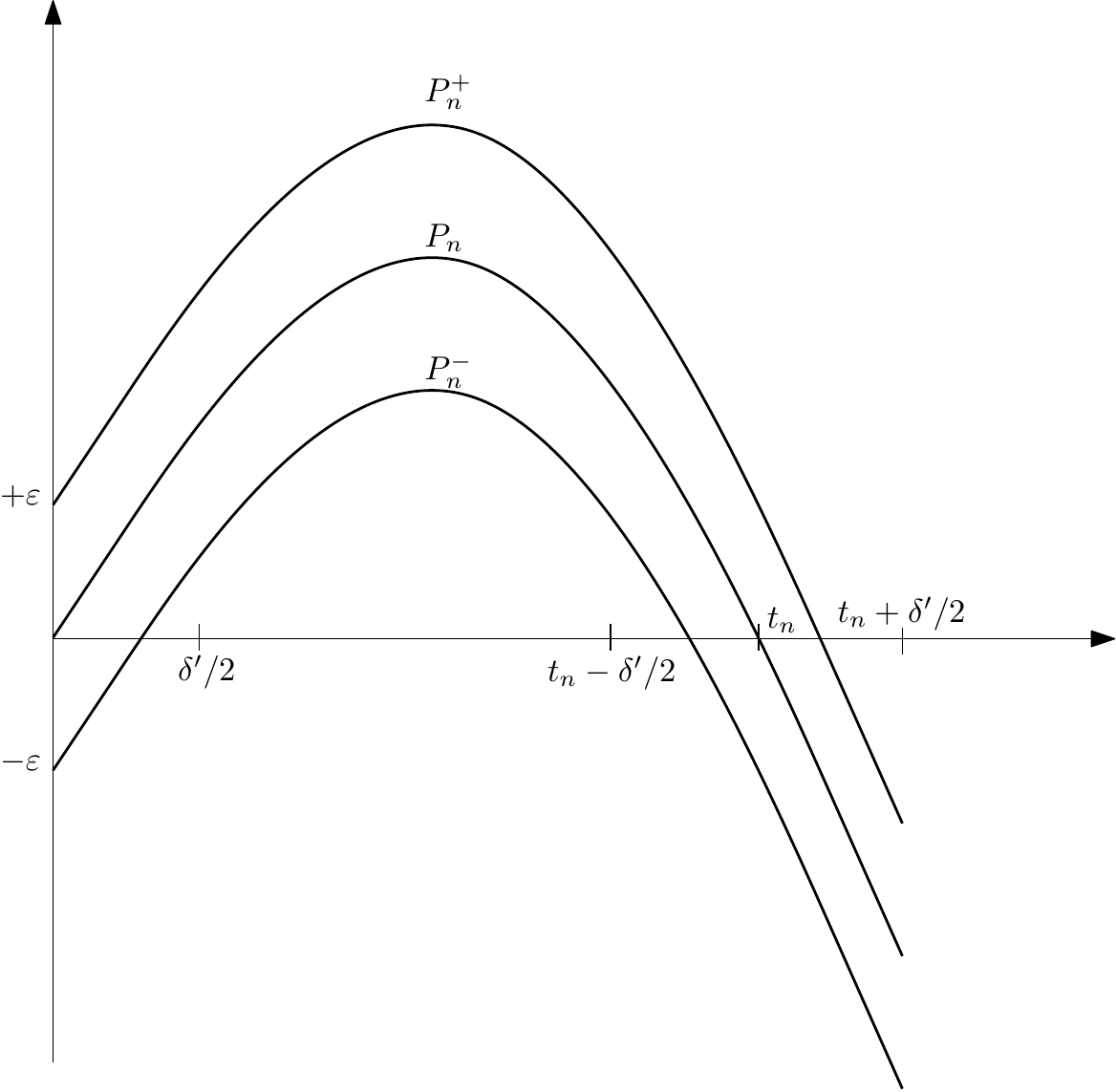}
\caption{The parabolas $P_n^-$, $P_n$, and $P_n^+$.}
\label{fig:parabola}
\end{figure}

Consider the event
\[
E_n = \left \{ \left | \frac1n \g_n^{-2} \cY(\g_n t) - \g_n^{-2} h_n(\g_n t) \right | \leq \frac{\eps}{4} \right \}.
\]
From \eqref{eq:tY} and Markov's inequality, we have $\P(E_n) \geq 1 - p_n$, where
\[
p_n = \frac{4 C}{\eps} \left ( \frac{1}{\sqrt{n \g_n^3}} + \g_n^{1+ \eta} \right ).
\]
From now on, we work on $E_n$, and take $n$ large enough (which can be chosen uniformly) such that $C \g_n \leq \eps/4$ and $C n \g_n^3 \leq \eps/4$. We thus have
\[
\inf_{s \leq t_n} \frac1n \g_n^{-2} \cY(\g_n s) \geq \inf_{s \leq t_n} \frac1n \g_n^{-2} h_n(\g_n s) - \frac{\eps}{4} \geq \inf_{s \leq t_n}  P_n(s) - C \g_n - \frac{\eps}{4} = - C \g_n - \frac{\eps}{4} \geq - \frac{\eps}{2}.
\]
where the second inequality stems from \eqref{eq:hn}. With \eqref{eq:error}, this implies
\begin{equation} \label{eq:error2}
\sup_{s \leq t_n^+} \frac1n \g_n^{-2} \left |  Y(\g_n s) - \cY(\g_n s) \right | \leq - \inf_{s \leq t_n} \frac1n \g_n^{-2} \cY(\g_n s) + C \g_n \leq \frac{3 \eps}{4},
\end{equation}
and from \eqref{eq:hn} and since we work on $E_n$, we get
\begin{equation} \label{eq:Y}
 \sup_{s \leq t_n^+} \left | \frac1n \g_n^{-2} Y(\g_n s) -  P_n(s)  \right |   \leq \frac{3 \eps}{4} + C n \g_n^3 \leq \eps.
\end{equation}
In other words, on $[0,t_n^+]$, $s \mapsto n^{-1} \g_n^{-2} Y(\g_n s)$ is stuck between $P_n^-$ and $P_n^+$. Hence, there must be an excursion of $Y(\g_n \cdot)$ above 0 on an interval $[T_1,T_2]$ containing the interval $[\dl'/2, t_n - \dl'/2]$. We have $T_1 \in [0,\dl'/2]$, and $T_2$ is the first time after $t_n-\dl'/2$ when C1 is performed. If $T_2 > t_n + \dl'/2$, then $t_n^+ > t_n + \dl'/2$, so $n^{-1} \g_n^{-2} Y(\g_n s)$ is below $P_n^+$ at least on the interval $[0,t_n+\dl'/2]$. This forces $Y(t_n+\dl'/2) < 0$, which cannot happen. Hence $T_2 \in [t_n - \dl'/2, t_n + \dl'/2]$, and there is an excursion of $Y$ with size in $[\g_n(t_n - \dl'), \g_n(t_n + \dl')]$. This corresponds to a component $\cC_0$.

\paragraph{5.}
Take now $\g_n \tau_1 \leq \g_n \tau_2 \leq \g_n t_n^+$ be two times when C1 is performed, so that the exploration of a component starts at $\g_n \tau_1$, then one or several components are explored, and an exploration ends at $\g_n \tau_2^-$. Let us compute the size and number $w_k(\tau_1,\tau_2)$ of vertices of degree $k$ in these components. We will first compute the number of vertices $N(\tau_1,\tau_2) = \sum_{k \geq 1} w_k(\tau_1,\tau_2)$ in these components.

To begin with, define
\[
N(t) = \sum_{k \geq 1} V_k(t), \quad \cN(t) = \sum_{k \geq 1} \cV_k(t).
\]
Recall that $V_k \leq \cV_k$, so that $Y - \cY = \cW - W = \sum k |\cV _k - V_k|$, and \eqref{eq:error2} then implies
\begin{equation} \label{eq:NcN}
\sup_{t \leq t_n^+} \frac1n \g_n^{-2} \left | N(\g_n t) - \cN(\g_n t) \right | \leq \frac{3 \eps}{4}.
\end{equation}
Moreover, \eqref{eq:sumkcVk} and Markov's inequality yield
\begin{equation} \label{eq:cNgn}
\sup_{t \leq t_n^+} \left | \frac1n \g_n^{-2} \cN(\g_n t) - \g_n^{-2} g_n(e^{-\g_n t}) \right | \leq \frac{\eps}{4}
\end{equation}
with probability greater than $1 - p_n$. Clearly, $w_k(\tau_1,\tau_2)$ is the number of vertices of degree $k$ which have turned from white to red during $[\g_n \tau_1, \g_n \tau_2)$, i.e. $w_k(\tau_1,\tau_2) = V_k(\g_n \tau_1) - V_k(\g_n \tau_2^-)$, so that $N(\tau_1,\tau_2) = N(\g_n \tau_1) - N(\g_n \tau_2^-)$. Since $\tau_1, \tau_2 \leq t_n^+$, \eqref{eq:NcN} and \eqref{eq:cNgn} imply
\[
\frac1n \g_n^{-2} \left | N(\tau_1,\tau_2) - \left ( g_n(e^{- \g_n \tau_1}) - g_n(e^{- \g_n \tau_2}) \right ) \right | \leq 2 \eps.
\]
A first order Taylor expansion then yields 
\[
g_n(e^{- \g_n \tau_1}) - g_n(e^{- \g_n \tau_2}) = \g_n g_n'(1) (\tau_2 - \tau_1) + O(\g_n^2) = \g_n m\n_1 (\tau_2 - \tau_1) + O(\g_n^2),
\]
where $O(\cdot)$ is uniform in the parameters. The last two equations then yield
\[
\left | \frac1n \g_n^{-1} N(\tau_1,\tau_2) - (\tau_2 - \tau_1) m\n_1 \right | \leq C \g_n. 
\]
We may from now on consider $n$ large enough (uniformly in the parameters), such that $C \g_n \leq \dl/2$. Therefore, we finally obtain
\begin{equation} \label{eq:sizecomp}
\left | \frac1n \g_n^{-1} N(\tau_1,\tau_2) - (\tau_2 - \tau_1) m\n_1 \right | \leq \dl/2.
\end{equation}
In particular, for $\tau_1 = T_1$ and $\tau_2 = T_2$, we get, remembering that $T_2 - T_1 \in [t_n - \dl', t_n + \dl']$, that the number of vertices $N(T_1,T_2)$ of  $|\cC_0|$ verifies
\begin{equation} \label{eq:sizecompC0}
\left | \frac1n \g_n^{-1} |\cC_0| - 2 \frac{m\n_1}{m\n_3} \right | \leq \dl/2 + M \dl' \leq \dl.
\end{equation}

\paragraph{6.}
We have therefore found a component $\cC_0$ of the right size. We shall now prove that with probability at least $1-C p_n$ there is no larger component in the graph. 

First, we know from the above reasoning that, with probability at least $1 - 2p_n$, the following happen.
\begin{itemize}
\item We build, on the interval $[T_1,T_2)$, a component $\cC_0$ of size in $(s_n-\dl,s_n+\dl) n \g_n$, where $s_n = 2 m\n_1/m\n_3$.
\item Any component built before $\cC_0$ has size less than
\[
n \g_n \left ( \frac{\dl}{2} + T_1 m\n_1 \right ) \leq n \g_n \left ( \frac{\dl}{2} + \dl' M \right ) = \dl n \g_n.
\]
\end{itemize}
These equations come from \eqref{eq:sizecomp} and \eqref{eq:sizecompC0}. We shall call this event $\cE_1$. Consider further the events
\begin{itemize}
\item $\cE_2$: we build a component of size $(s_n-\dl,s_n+\dl) n \g_n$ {\em and} at least another one of size greater than $\dl n \g_n$;
\item $\cE_2'$: we build at least one component of size greater than $\dl n \g_n$ {\em before} a component of size $(s_n-\dl,s_n+\dl) n \g_n$. 
\end{itemize} 
Obviously $\cE'_2 \subset \cE_2$, and $\cE_2' \cap \cE_1 = \emptyset$.

As mentioned in Remark \ref{rk:sizebias}, the components are built in a size-biased order, and thus
\[
\P(\cE_2' \vert \cE_2) \geq \dl / (s_n + 2 \dl) \geq p
\]
for some $p > 0$ uniform in the parameters. Hence, 
\[
\P(\cE_2) = \frac{\P(\cE_2' \cap \cE_2)}{\P(\cE_2'|\cE_2)} \leq \frac1p \P(\cE_2') \leq \frac1p (1- \P(\cE_1)) \leq \frac{2 p_n}{p}.
\]
Finally
\[
\P(\cE_1 \bsl \cE_2) \geq \P(\cE_1) - \P(\cE_2) \geq 1 - 2 p_n - \frac{2 p_n}{p} = 1 - 2 \left( \frac1p + 1 \right ) p_n.
\]
The first part of the result follows after noticing that $\cE_1 \bsl \cE_2$ is exactly the event considered there.

\paragraph{7.}
Let us write w.h.p. to mean with probability at least $1 - C p_n$, for some $C$ uniform in the parameters. The reasoning above shows that, w.h.p., $\cC_0$ is indeed $\cC_1(n,\pi\n)$. Hence, to complete the proof, we need to show that
\begin{equation} \label{eq:goal}
\sum_{k \geq 1} k^2 \left | \frac{1}{n \g_n} v_k(\cC_0) - \frac{2}{m_3\n} k \pi\n(k) \right | \leq \dl, \quad \frac1n \sum_{k \geq 1} k^3 v_k(\cC_0) \leq \dl
\end{equation}
w.h.p.

To begin with, w.h.p., $\cC_0$ is the largest component in the graph, and has size at least $c n \g_n$, where $c > 0$ is uniform in the parameters. Since the components are constructed in a size-biased order, there is probability at least $c n \g_n / n = c \g_n$ that it is explored at the first step. If not, there is probability at least $c n \g_n / n' \geq c \g_n$ that it is explored at the second step, where $n - n'$ is the size of the first explored component. Clearly, this shows that the number $K$ of times that C1 is performed before exploring $\cC_0$ is stochastically dominated by a geometric distribution $\cG(1-c\g_n)$ with success parameter $1 - c \g_n$. So take $\b > 1$ such that $2 + \b > 2 / (4 + \eta)$. We have then
\[
\P \left ( \cG(1-c\g_n) > 1/\g_n^{1+\b} \right ) = (1- c \g_n)^{1/\g_n^{1+\b}} \leq \exp - c / \g_n^{\b} \leq C \g_n^{1+\eta}.
\]
Therefore, w.h.p., $ K \leq \g_n^{-(1+\b)}$.

Now, recall from Point 3 above that $0 \leq \cV_k(t) - V_k(t)$, and that this can only increase when C1 occurs. Recall that $\cC_0$ is explored on an interval $[\g_n T_1, \g_n T_2)$, with $T_2$ uniformly bounded. Then, w.h.p.,
\begin{equation} \label{eq:k2VkcVk}
\begin{split}
\sup_{t < \g_n T_2} \sum_{k \geq 1} k^2 \left | V_k(t) - \cV_k(t) \right | & \leq (\max \{i, \pi\n(i) \neq 0\})^2 K \\
& \leq M n^{2/(4+\eta)} \g_n^{-(1+\b)} \\
& \leq \frac{\dl}{2} n \g_n,
\end{split}
\end{equation}
for $n$ large enough, where the last inequality comes from the fact that $\g_n \geq \eps_n^- \gg n^{-1/4}$ and the choice of $\b$. This is when we most crucially use this assumption.

A similar computation as in Point 1 of the proof allows to obtain
\begin{equation} \label{eq:boundk2cVk}
\sum_{k \geq 1} k^2 \E \left [ \sup_{t \leq \g_n T_2} \left | \cV_k(t) - n \pi\n(k) \exp(- k t) \right | \right ] \leq C \left ( \sqrt{\frac{n}{\g_n}} + n \g_n^{2 + \eta} \right ).
\end{equation}
Markov's inequality therefore ensures that
\[
\P \left ( \sum_{k \geq 1} k^2 \E \left [ \sup_{t \leq \g_n T_2} \left | \cV_k(t) - n \pi\n(k) \exp(- k t) \right | \right ] > \frac{\dl}{2} n \g_n \right ) \leq C \left ( \frac{1}{\sqrt{n \g_n^3}} + \g_n^{1 + \eta} \right ).
\]
Along with \eqref{eq:k2VkcVk}, this gives
\[
\sup_{t < \g_n T_2} \sum_{k \geq 1} k^2 \left | V_k(t) - n \pi\n(k) \exp(- k t) \right | \leq \dl n \g_n
\]
w.h.p. But $v_k(\cC_0) = V_k(\g_n T_1) - V_k(\g_n T_2^-)$, so that
\[
\sum_{k \geq 1} k^2 \left | \frac{1}{n \g_n} v_k(\cC_0) - \frac{1}{\g_n} \pi\n(k) \left ( e^{- k \g_n T_1} - e^{- k \g_n T_2} \right ) \right | \leq 2 \dl.
\]
A simple Taylor expansion then allows to conclude to the first part of \eqref{eq:goal}. To get the second part of \eqref{eq:goal}, note that similar computations allow to obtain
\[
\sum_{k \geq 1} k^3 \left | \frac{1}{n} v_k(\cC_0) - \pi\n(k) \left ( e^{- k \g_n T_1} - e^{- k \g_n T_2} \right ) \right | \leq 2 \dl,
\]
so that
\begin{align*}
\frac{1}{n} \sum_{k \geq 1} k^3 v_k(\cC_0) & \leq 2 \dl + \sum_{k \geq 1} k^3 \pi\n(k) \left ( e^{- k \g_n T_1} - e^{- k \g_n T_2} \right ) \\
&  \leq 2 \dl + C \g_n \sum_{k \geq 1} k^4 \pi\n(k)
\end{align*}
and the assumptions \eqref{eq:degreeass} allows to conclude.

All these reasonings are for $n$ large enough, but this can always be done uniformly in the parameters. It therefore suffices to take a larger $K$ in the statement to deal with the small $n$.

\begin{remark}
A $C^1$ estimation of $\phi_n$ could be directly obtained from Point 5 of the proof. However, to get bounds on the second and third moments, we need to use the slightly different techniques of Point 7, essentially due to the fact that the techniques leading to \eqref{eq:NcN} can only provide a bound for the first moment.
\end{remark}

\subsection{Extension to particles with free arms}

\subsubsection{Result}

In our model, coagulations occur thanks to the \emph{activated} arms, but each particle also carries a number of \emph{free} (i.e. not activated yet) arms. Hence, what happens at a gelation event only depends on the distribution of activated arms. However, given the state in solution after a gelation event, the \emph{dynamics} depend on the number of free arms. For this reason, we will need to keep track of the number of free arms. More precisely, we will follow the number of particles with a certain number of active and free arms, and it makes sense to first study this for the configuration model. Rather than the number of free arms, we will instead use the degree as the parameter, since it is a constant for a given particle. Vertices with $k$ activated arms and degree $r$ will just be called $(k,r)$-vertices.

Therefore, let us consider $n$ vertices, with $n \xi\n(k,r)$ of them being $(k,r)$-vertices. Only the active arms are involved in the coagulations, and we therefore define $CM(n,\xi\n)$ as a uniform pairing of the active arms, and let $\cC_1(n,\xi\n)$ and $\cC_2(n,\xi\n)$ be its largest and second largest component. For $k, r \geq 0$, let
\[
\pi\n(k) = \sum_{r \geq 0} \xi\n(k,r), \quad \mu\n(r) = \frac1n \sum_{k = 0}^r \xi\n(k,r).
\]
Therefore, if one forgets about the free arms, then $CM(n,\xi\n)$ has exactly the same distribution as a $CM(n,\pi\n)$ as in the previous section. We further define
\[
\fG_n(x,y) = \sum_{r \geq 0} \sum_{k=0}^r \xi\n(k,r) x^k y^r, \quad \fF_n(x,y) = x \d{\fG_n}{x}(x,y),
\]
as well as $\fm_i\n = m_i^{\pi\n}$ and $\g_n = \g^{\pi\n}$. We reserve this notation with Gothic letters for the ``effective degree''. This will become more relevant for Theorem \ref{th:JLimproved3} below.

Under the same assumptions on $(\pi_n)$, Theorem \ref{th:JLimproved} is then still valid for $CM(n,\xi\n)$. However, it does not tell us about the distribution of the degree of the vertices that belong to the largest component. For this reason, let us denote by $v_{k,r}(\cC)$ the number of $(k,r)$-vertices in a component $\cC$. Introduce
\[
\psi_n(x,y) = \sum_{r \geq 0} \sum_{k=0}^r v_{k,r} \left ( \cC_1(n,\xi\n) \right ) x^k y^r, \quad w_n = \sum_{k, r \geq 0} r^3 v_{k,r} \left ( \cC_1(n,\xi\n) \right ).
\]
We shall need assumptions corresponding to \eqref{eq:degreeass}. Note that the ``$4+\eta$ moment'' assumption on $\pi\n$ is implied by a similar assumption on $\mu\n$. However, we will need to assume $5+\eta$ moments here. We shall therefore fix two positive sequences
\[
n^{-1/4} \ll \eps_n^- \ll \eps_n^+ \ll 1
\]
and three constants $\eta, m, M \in (0,\pinf)$, and assume the following.

\begin{assume}
\begin{equation} \label{eq:degreeass2}
\eps_n^- \leq \g_n \leq \eps_n^+, \quad m \leq \fm_3\n, \quad \sum_{k \geq 1} \mu\n(r) r^{5 + \eta} \leq M.
\end{equation}
\end{assume}
The extension of Theorem \ref{th:JLimproved} corresponding to this model is the following.

\begin{theorem} \label{th:JLimproved2}
Consider a sequence $(\xi\n)$ such that \eqref{eq:degreeass2} holds. Then, for any $\dl > 0$, there exists a constant $K > 0$, depending only on $\dl$, $\eta$, $m$, $M$ and $(\eps_n^{\pm})$ such that, with probability greater than
\[
1 - K \left ( \frac{1}{\sqrt{n \g_n^3}} + \g_n^{1+\eta} \right )
\]
we have
\[
\left ( 2 \frac{\fm_1\n}{\fm_3\n} - \dl \right ) n \g_n \leq |\cC_1(n,\xi\n)| \leq \left ( 2 \frac{\fm_1\n}{\fm_3\n} + \dl \right ) n \g_n, \quad |\cC_2(n,\xi\n)| \leq \dl n \g_n,
\]
and
\[
\left \| \frac{1}{n\g_n} \psi_n - \frac{2}{\fm_3\n} \fF_n \right \|_2 \leq \dl, \quad \frac1n w_n \leq \dl.
\]
\end{theorem}

Of course, the only difference with the previous result is the last part, concerning the number of $(k,r)$-vertices in the largest component. It says as before that the vertices that make up this largest component are chosen according to a size-bias of their number of active arms, whereas the degree is irrelevant.

In particular, as the free arms do not matter for the coagulations, this result could be rephrased using abstract types rather than the total number of arms, but at the price of slightly awkward assumptions. Specifically, we systematically use that $k \leq r$.

A final note is that, because of the $5+\eta$ moment assumption, rather than $4+\eta$ in Theorem \ref{th:JLimproved}, the lower bound on $\g_n$ could be slightly improved. For all intents and purposes, this result will be more than sufficient.

\subsubsection{Proof}

The proof relies on the exact same argument as the one of Theorem \ref{th:JLimproved}. So we perform the algorithm C1, C2, C3 on the active arms, and consider all the same variables and notations. Note that the assumptions \eqref{eq:degreeass} on $\pi\n$ are in force, and thus every single step of the proof is still valid. We wish to prove that
\begin{equation} \label{eq:goal2}
\sum_{r \geq 0} \sum_{k = 0}^r r^2 \left | \frac{1}{n \g_n} v_{k,r}(\cC_0) - \frac{2}{\fm_3\n} k n \xi\n(k,r) \right | \leq \dl
\end{equation}
w.h.p.

To this end, we need to keep track of the extra information provided by the total number of arms, so we additionally introduce the variable $\cV_{k,r}(t)$ for the number of $(k,r)$-vertices that have all \emph{active} half-edges with lifetimes $\tau_x > t$, and similarly $V_{k,r}(t)$ for the number of $(k,r)$-vertices that are still white at time $t$.

The only difference comes from Step 7 of the proof of Theorem \ref{th:JLimproved}. First, $0 \leq \cV_{k,r} - V_{k,r}$, and this can only increase when C1 occurs, which happens $K \leq \g_n^{-(1+\b)}$ times w.h.p. Moreover, w.h.p., the largest component $\cC_0$ is explored on an interval $[\g_n T_1, \g_n T_2)$ with $T_2$ uniformly bounded, and thus
\begin{equation} \label{eq:k2VkrcVkr}
\begin{split}
\sup_{t < \g_n T_2} \sum_{k,r \geq 0} r^2 \left | V_{k,r}(t) - \cV_{k,r}(t) \right | & \leq (\max \{i, \mu\n(i) \neq 0\})^2 K \\
& \leq M n^{2/(5+\eta)} \g_n^{-(1+\b)} \\
& \leq \dl n \g_n.
\end{split}
\end{equation}
This is just \eqref{eq:k2VkcVk} in this context.

We now need an equivalent of \eqref{eq:boundk2cVk}, which we shall get as in Point 1 of the previous proof. As for \eqref{eq:cVk}, $\cV_{k,r}$ is a death process decreasing by one at rate $k \cV_{k,r}(t)$ and starting from $n \xi\n(k,r)$. We therefore have
\begin{equation} \label{eq:cVkr}
\E \left [ \sup_{t \leq \g_n T_2} \left | \cV_{k,r}(t) - n \xi\n(k,r) \exp(- k t) \right |^2 \right ] \leq C k n \g_n \xi\n(k,r), \quad k \leq 1 / \g_n.
\end{equation}
As before, we write
\begin{align*}
\sum_{k, r \geq 0} r^2 & \: \E \left [ \sup_{t \leq \g_n T_2} \left | \cV_{k,r}(t) - n \xi\n(k,r) e^{-k t} \right | \right ] \\
& \leq \g_n^{-1} \sum_{r=0}^{1/\g_n} r \sum_{k=0}^r \E \left [ \sup_{t \leq \g_n T_2} \left | \cV_{k,r}(t) - n \xi\n(k,r) e^{-k t} \right | \right ] \\
& \quad + \sum_{r > 1/\g_n} r^2 \sum_{k=0}^r \E \left [ \sup_{t \leq \g_n T_2} \left | \cV_{k,r}(t) - n \xi\n(k,r) e^{-k t} \right | \right ],
\end{align*}
and deal with each sum on the RHS similarly. First, the second sum on the RHS is bounded by
\[
\sum_{r > 1/\g_n} r^2 \sum_{k=0}^r n \xi\n(k,r) = n \sum_{r > 1/\g_n} r^2 \mu\n(r) \leq n \g_n^{2+\eta} \sum_{r > 1/\g_n} r^{4+\eta} \mu\n(r) \leq M n \g_n^{2+\eta}.
\]
As for the first sum on the RHS, we use \eqref{eq:cVkr} and Cauchy-Schwarz to write that it is bounded by
\begin{align*}
C \sqrt{n \g_n} \sum_{r=0}^{1/\g_n} r \sum_{k=1}^r \sqrt{k \xi\n(k,r)}
& \leq C \sqrt{n \g_n} \sum_{r=0}^{1/\g_n} r \left ( \sum_{k=1}^r k \right )^{1/2} \left ( \sum_{k=1}^r \xi\n(k,r) \right )^{1/2} \\
& \leq C \sqrt{n \g_n} \sum_{r=0}^{1/\g_n} r^2 \sqrt{\mu\n(r)} \\
& \leq C \sqrt{n \g_n} \left ( \sum_{r=0}^{1/\g_n} r^{5+\eta} \mu\n(r) \right )^{1/2} \left ( \sum _{r = 1}^{1/\g_n} r^{-1-\eta} \right )^{1/2} \\
& \leq C \sqrt{n \g_n},
\end{align*}
where the last step comes from \eqref{eq:degreeass2}. To summarize, we obtain
\begin{equation} \label{eq:sumcVk}
\E \left [ \sum_{k, r \geq 0} r^2 \sup_{t \leq \g_n T_2} \left | \cV_{k,r}(t) - n \xi\n(k,r) e^{-k t} \right | \right ]  \leq C \left (\sqrt{\frac{n}{\g_n}}  + n \g_n^{2 + \eta} \right ).
\end{equation}
This is just the equivalent of \eqref{eq:boundk2cVk} in this context. Formula \eqref{eq:goal2} and the rest of the statement can then be obtained as in the previous proof.

\subsection{Percolation on the configuration model}

\subsubsection{Result} \label{sec:percolres}

In our real model, at each time $t$, arms are activated independently with probability $p_t = 1-e^{-t}$. Hence, the degree of a particle is fixed, but its number of activated arms is random. We will thus study the corresponding configuration model in this context.

To be precise, assume that there are $n$ vertices, $n \mu\n(r)$ of them being of degree $r$. Let $G_n = G_{\mu\n}$ be the generating function of $\mu\n$ and define $m_i\n = m_i^{\mu\n}$ the factorial moments of $\mu\n$. Let $p_n \in [0,1]$ and assume that each arm is activated independently with probability $p_n$. We then independently perform a uniform pairing of the activated arms, to get a configuration model $CM(n,\mu\n,p_n)$. We naturally denote $\cC_1(n,\mu\n,p_n)$ and $\cC_2(n,\mu\n,p_n)$ its largest and second largest component, $v_{k,r}(\cC)$ the number of $(k,r)$-vertices in a component $\cC$, and finally
\[
\psi_n(x,y) = \sum_{r \geq 0} \sum_{k=0}^r v_{k,r} \left ( \cC_1(n,\mu\n,p_n) \right ) x^k y^r, \quad w_n = \sum_{k, r \geq 0} r^3 v_{k,r} \left ( \cC_1(n,\mu\n,p_n) \right ).
\]

Of course, conditionally on the activated arms (in a ``quenched'' setting), Theorem \ref{th:JLimproved2} will apply. Two issues however remain: first, one needs to check the assumptions \eqref{eq:degreeass2}. This can be achieved, as we will see, when $p_n$ is chosen so that we get a critical model, namely
\[
p_n = \frac{m_1\n}{m_2\n} (1 + \eps_n).
\]
Secondly the result will then provide information that is conditional on the activated arms, that would not be convenient to use. However, by the law of large number, the initial distribution $\xi\n$ should be almost deterministic. Indeed, if we let
\begin{equation} \label{eq:pkr}
[p]_k^r = \binom{r}{k} p^k (1-p)^{r-k},
\end{equation}
then $\xi\n(k,r)$ should be about $\mu\n(r) [p_n]_k^r$. Therefore, the generating function $\fG_n$ of $\xi\n(k,r)$ should be close to its ``annealed'' version
\[
\bar{\fG}_n(x,y) = \sum_{r \geq 0} \sum_{k=0}^r \mu\n(r) [p_n]_k^r x^k y^r = G_n ((p_n x + 1 - p_n)y).
\]
Ignoring the first order terms, it makes sense to define
\[
q_n(x,y) = \left ( \frac{m\n_1}{m\n_2} x + 1 - \frac{m\n_1}{m\n_2} \right ) y.
\]
With a slight abuse of notation, we will write $x \fG_n$ for the function $x \mapsto x \fG_n(x)$, and so on. We shall again fix two positive sequences
\begin{equation} \label{eq:assepsnpm}
n^{-1/4} \ll \eps_n^- \ll \eps_n^+ \ll 1
\end{equation}
and three constants $\eta, m, M \in (0,\pinf)$, and do the following assumptions.

\begin{assume}
\begin{equation} \label{eq:degreeass3}
m \leq m\n_1 \leq m\n_2, \quad m \leq m\n_3, \quad \eps_n^- \leq \eps_n \leq \eps_n^+, \quad \sum_{k \geq 1} \mu\n(r) r^{5 + \eta} \leq M.
\end{equation}
\end{assume}

We will prove the following consequence of Theorem \ref{th:JLimproved2}.

\begin{theorem} \label{th:JLimproved3}
Consider sequences $(\mu\n)$ and $(p_n)$ such that \eqref{eq:degreeass3} holds. Then, for any $\dl > 0$, there exists a constant $K > 0$, depending only on $\dl$, $\eta$, $m$, $M$ and $(\eps_n^{\pm})$ such that, with probability greater than
\[
1 - K \left ( \frac{1}{\sqrt{n \eps_n^3}} + \eps_n^{1+\eta} \right )
\]
we have
\[
\left ( 2 \frac{m\n_1 m_2\n}{m\n_3} - \dl \right ) n \eps_n \leq |\cC_1(n,\mu\n,p_n)| \leq \left ( 2 \frac{m\n_1 m\n_2}{m\n_3} + \dl \right ) n \eps_n,
\]
and
\[
|\cC_2(n,\mu\n,p_n)| \leq \dl n \eps_n,
\]
and
\[
\left \| \frac{1}{n\eps_n} \psi_n - 2 \frac{m\n_2}{m\n_3} x y G_n'(q_n) \right \|_2 \leq \dl, \quad \frac1n w_n \leq \dl.
\]
\end{theorem}

This shows that, in a configuration model with parameters $n$ and $\mu\n$, the right percolation threshold to obtain a largest component of size of order $n \eps_n$ is $m\n_1/m\n_2 (1 + \eps_n)$. Moreover, taking $x=1$ or $y=1$ in the second part of the statement shows that the particles belonging to the largest component are chosen both by a size-bias of their number of activated arms, and by a size-bias of their total number of arms. This will prove fundamental in our analysis.

\subsubsection{Proof}

Let us proceed with the proof of this result. As we mentioned, the number $n \xi\n(k,r)$ of $(k,r)$-vertices should be close to $\mu\n(r) [p_n]_k^r$. More precisely, we will prove the following inequality. We denote here $\xi_p$ for the empirical distribution of $(k,r)$-vertices when each arm is activated independently with probability $p$, and $\Q$ for the probability under which the activated arms are chosen.

\begin{lemma} \label{lem:momentcomp}
For any $\dl > 0$, there is a constant $K \in (0,\pinf)$, depending only on $\dl$ and $M$, such that for any $p \in [0,1]$, $n \geq 1$ and $a_n > 0$,
\begin{equation} \label{eq:2moment}
\Q \left ( \sum_{r \geq 0} \sum_{k=0}^r r^2 \left | \xi_p\n(k,r) - \mu\n(r) [p]_k^r \right | > \dl a_n \right ) \leq K \left( a_n^{1+\eta} + \frac{1}{\sqrt{n a_n^3}} \right )
\end{equation}
and
\begin{equation} \label{eq:3moment}
\Q \left ( \sum_{r \geq 0} \sum_{k=0}^r r^3 \left | \xi_p\n(k,r) - \mu\n(r) [p]_k^r \right | > \dl \right ) \leq K \left( a_n^{1+\eta} + \frac{1}{\sqrt{n a_n^3}} \right ).
\end{equation}
\end{lemma}

We delay the proof of this result to the end of this section, and shall first use it to prove Theorem \ref{th:JLimproved3}.

\begin{proof}[Proof of Theorem \ref{th:JLimproved3}]
Let us use the same notation $\pi\n$, $\fm_i\n$, $\g_n$, $\fG_n$, $\fF_n$ as in Theorem \ref{th:JLimproved2}. Note that these quantities are random, but conditionally on their values, Theorem \ref{th:JLimproved2} applies, provided the assumptions \eqref{eq:degreeass2} hold. We shall therefore start by checking them.

To begin with, the last assumption of \eqref{eq:degreeass2} is just the same as in \eqref{eq:degreeass3}. Now, from \eqref{eq:2moment} and \eqref{eq:3moment} with $a_n = \eps_n$, $p = p_n$ and $\dl$ small enough (to be fixed later), we have
\begin{equation} \label{eq:xin}
\sum_{r \geq 0} \sum_{k=0}^r r^2 \left | \xi\n(k,r) - \mu\n(r) [p_n]_k^r \right | \leq \dl \eps_n, \quad \sum_{k, r \geq 0}^r r^3 \left | \xi\n(k,r) - \mu\n(r) [p_n]_k^r \right | \leq \dl
\end{equation}
with probability at least
\[
1 - 2 K \left( \eps_n^{1+\eta} + \frac{1}{\sqrt{n \eps_n^3}} \right ).
\]
As usual, we will just write w.h.p. for such a probability, where $K$ depends only on the fixed parameters, and $C$ will be a constant which may change from line to line, but depends only on  the fixed parameters (but not $\dl$).

To begin with, the first equation of \eqref{eq:xin} implies that
\[
\fm_1\n = \sum_{r \geq 1} \sum_{k=1}^r k \xi\n(k,r),
\]
is, up to $\dl \eps_n$,
\[
\sum_{r \geq 1} \sum_{k=1}^r k \mu\n(r) [p_n]_k^r = \sum_{r \geq 1} p_n r \mu\n(r) = p_n m\n_1.
\]
Similar computations hold for the second and third factorial moments, and thus
\begin{equation} \label{eq:momentcomp}
\left | \fm_1\n - p_n m\n_1 \right | \leq C \dl \eps_n, \quad \left | \fm_2\n - p_n^2 m\n_2 \right | \leq C \dl \eps_n, \quad \left | \fm_3\n - p_n^3 m\n_3 \right | \leq C \dl.
\end{equation}
A couple easy Taylor expansions then ensure that the corresponding $\g_n$ verifies
\begin{equation} \label{eq:gammaeps}
\left | \g_n - \frac{(m\n_1)^2}{m\n_2} \eps_n \right | \leq C \dl \eps_n.
\end{equation}
The assumptions \eqref{eq:degreeass3} easily show that $(p_n)$ is uniformly bounded below by a positive constant, and it is then easy to deduce from \eqref{eq:momentcomp} and \eqref{eq:gammaeps} that the assumptions \eqref{eq:degreeass2} of Theorem \ref{th:JLimproved2} are verified, provided we chose $\dl$ small enough. Moreover \eqref{eq:gammaeps} shows that $\eps_n$ and $\g_n$ differ only by a multiplicative constant that is bounded above and below. Consequently, in the notations of Theorem \ref{th:JLimproved2}, w.h.p., we have
\[
\left ( 2 \frac{\fm_1\n}{\fm_3\n} - \dl \right ) n \g_n \leq |\cC_1(n,\xi\n)| \leq \left ( 2 \frac{\fm_1\n}{\fm_3\n} + \dl \right ) n \g_n, \quad |\cC_2(n,\xi\n)| \leq \dl n \g_n,
\]
and
\[
\left \| \frac{1}{n\g_n} \psi_n - \frac{2}{\fm_3\n} \fF_n \right \|_2 \leq \dl, \quad \frac1n w_n \leq \dl.
\]
This already proves the very last part of the result. On the other hand, using \eqref{eq:momentcomp} and \eqref{eq:gammaeps} allows to replace, up to a factor $C \dl \eps_n$, $\fm_i\n$ by $p_n m\n_i$ and $\g_n$ by $(m\n_1)^2/m\n_2 \eps_n$, and a couple lines of computation provide the first part of the statement. To conclude, notice that \eqref{eq:xin} ensures that
\[
\left \| \fG_n - \bar{\fG}_n \right \|_3 \leq \dl
\]
w.h.p. Hence, up to a factor $\dl$, we may replace, in the above equation, $\fF_n$ by
\begin{align*}
x \d{\bar{\fG}_n}{x}(x,y) & = p_n x y G_n'((p_n x + 1 - p_n)y) \\
& = \frac{m_1\n}{m_2\n} x y G_n' \left ( \left ( \frac{m\n_1}{m\n_2} x + 1 - \frac{m\n_1}{m\n_2} + \eps_n \frac{m\n_1}{m\n_2} (x-1) \right ) y \right ) + O(\eps_n) \\
& = \frac{m_1\n}{m_2\n} x y G_n'(q_n(x,y)) + O(\eps_n),
\end{align*}
where both lines are easy Taylor expansions and $O(\cdot)$ is uniform in the parameters. This readily allows to conclude.
\end{proof}

We finally just need to prove Lemma \ref{lem:momentcomp}.

\begin{proof}[Proof of Lemma \ref{lem:momentcomp}]
Let us prove the first inequality, the second being similar. First split the sum on $r \leq 1/a_n$ and $r > 1/a_n$. To deal with the second sum, just write
\[
\sum_{k=0}^r r^2 \left | \xi_p\n(k,r) - \mu\n(r) [p]_k^r \right | \leq r^2 \sum_{k=1}^r \xi_p\n(k,r) = r^2 \mu\n(r)
\]
to obtain
\begin{align*}
\E_{\Q} \left [ \sum_{r > 1/a_n} \sum_{k=0}^r r^2 \left | \xi_p\n(k,r) - \mu\n(r) [p]_k^r \right | \right ] & \leq \sum_{r > 1/a_n} r^2 \mu\n(r) \\
& \leq 2 a_n^{2+\dl} \sum_{r > 1/a_n} r^{4 + \eta} \mu\n(r) \\
& \leq 2 M a_n^{2+\dl},
\end{align*}
so that Markov's inequality yields
\begin{equation} \label{eq:thirdmoment}
\Q \left ( \sum_{r > 1/a_n} \sum_{k=0}^r r^2 \left | \xi_p\n(k,r) - \mu\n(r) [p]_k^r \right | > \frac{\dl}{2} a_n \right ) \leq \frac{4 M}{\dl} a_n^{1+\dl}.
\end{equation}

For the $r \leq 1/a_n$ term, first note that, for each $r$, the vector $(n \xi_p\n(k,r))_{0 \leq k \leq r}$ is multinomial with parameters $n \mu\n(r)$ and $([p]_k^r)_{0 \leq k \leq r}$. In particular, $n \xi_p\n(k,r)$ has a binomial distribution with parameters $n \mu\n(r)$ and $[p]_k^r$, so that Chebychev's inequality yields
\[
\E_{\Q} \left [ \left | n \xi_p\n(k,r) - n \mu\n(r) [p]_k^r \right | \right ] \leq \sqrt{n \mu\n(r) [p]_k^r (1-[p]_k^r)} \leq \sqrt{n \mu\n(r) [p]_k^r}.
\]
We can then compute, for a constant $C$ depending only on $M$ and $\eta$, that
\begin{align*}
\E_{\Q} \left [ \sum_{r \leq 1/a_n} \sum_{k=0}^r r^2 \left | \xi_p\n(k,r) - \mu\n(r) [p]_k^r \right | \right ] & \leq \frac1n \sum_{r \leq 1/a_n} r^2 \sum_{k=0}^r \E_{\Q} \left [ \left | n \xi_p\n(k,r) - n \mu\n(r) [p]_k^r \right | \right ] \\
& \leq \frac{1}{\sqrt{n}} \sum_{r \leq 1/a_n} r^2 \sum_{k=0}^r \sqrt{\mu\n(r) [p]_k^r} \\
& \leq \frac{1}{\sqrt{n}} \sum_{r \leq 1/a_n} r^2 \sqrt{\mu\n(r)} (\sum_{k=0}^r 1 )^{1/2} (\sum_{k=0}^r [p]_k^r)^{1/2} \\
& \leq C \frac{1}{\sqrt{n}} \sum_{r \leq 1/a_n} r^{5/2} \sqrt{\mu\n(r)} \\
& \leq C \frac{1}{\sqrt{n a_n}} \sum_{r \leq 1/a_n} r^2 \sqrt{\mu\n(r)} \\
& \leq C \frac{1}{\sqrt{n a_n}} \left ( \sum_{r \leq 1/a_n} r^{5+\eta} \mu\n(r) \right )^{1/2} \left ( \sum_{r \leq 1/a_n} r^{-(1+\eta)} \right )^{1/2} \\
& \leq C \frac{1}{\sqrt{n a_n}},
\end{align*}
where we use Cauchy-Schwarz twice. We conclude with Markov's inequality that
\[
\Q \left ( \sum_{r \leq 1/a_n} \sum_{k=0}^r r^2 \left | \xi_p\n(k,r) - \mu\n(r) [p]_k^r \right | > \frac{\dl}{2} a_n \right ) \leq \frac{2 C}{\dl} \frac{1}{\sqrt{n a_n^3}}.
\]
what, along with \eqref{eq:thirdmoment}, provides the result. The proof of the second statement is similar, by merely writing $r^{7/2} \leq a_n^{-3/2} r^2$ in the second step of the reasoning.
\end{proof}

%% file: 6-GeltimesDCM.tex
\section{Gelation in a dynamic configuration model} \label{sec:gelDCM}

\subsection{Introduction}

As already explained in Section \ref{sec:firstgel}, we can compare our model to a dynamic CM, at least until the first gelation time. More generally, we will introduce in the next section an alternative model that will allow us to compare our model to a DCM at any time, at least for the particles in solution. In particular, we will have to consider an arbitrary number $n$ of particles with a degree distribution $\mu$, where $n \leq N$ and $\mu$ can be any (reasonable) distribution. However, the threshold for the size of the components that fall into the gel remains unchanged, still $\a(N)$.

So let us consider a DCM described as follows. Consider $n$ vertices, whose degree distribution is given by $\mu$, and let $\th = (n, \mu)$. Set i.i.d. clocks on the arms with distribution $\cE(1)$, and link the first arm that rings with the second one, the third with the fourth, and so on. This therefore provides a graph process $(\cG_t(\th))$. If, at some time, there is only one activated arm, we assume that it does not matter in the graph structure. As usual, we define $m^{\mu}_i$ as the factorial moments of $\mu$, $G_{\mu}$ as its generating function, and we introduce
\[
q_{\mu}(x,y) = \left ( \frac{m^{\mu}_1}{m^{\mu}_2} x + 1 - \frac{m^{\mu}_1}{m^{\mu}_2} \right ) y, \quad F_{\mu}(y) = y G'_{\mu}(y).
\]

As in the previous section, we will take constants $c, m, M \in (0,\pinf)$, and we will assume that $\th \in \K\N_{c,m,M}$, which was defined in Section \ref{sec:prelim}. We will then obtain results with a precision $\k > 0$. The results will only depend on these quantities, and we classically write $K = K(c,m,M,\k)$ to mean that $K$ is a constant that depends only on $c$, $m$, $M$ and $\k$. Our main assumption on $(\a(N))$ will be Assumption \ref{as:alpha}, and this will remain fixed, so we will not include it as a parameter. We will write
\begin{equation} \label{eq:pN}
p_N = \frac{N}{\a(N)^{3/2}} + \left ( \frac{\a(N)}{N} \right )^{1+\eta}
\end{equation}
and all our results will be obtained with probability at least $1 - K p_N$, what we will shorten as w.h.p. in the proofs.

Because $\th \in \K\N_{c,m,M}$, the generating functions $G_{\mu}$ have a fifth moment that is uniformly bounded above. In particular, if we do a Taylor expansion of $G_{\mu}$ (up to the fourth order), the constants hidden in the $O(\cdot)$ can be chosen uniformly in all the fixed parameters, and similarly for expansions of $G_{\mu}'$ up to the third order, and so on. Similarly, the bounds on the $m^{\mu}_i$ will readily ensure that all the expressions involving the $m^{\mu}_i$ that we will consider are uniformly bounded above. We will implicitly always consider $N$ large enough such that everything that we write makes sense. Specifically, since we assume \eqref{eq:alpha}, then any expression that we will consider multiplied by $\a(N)/N$ will be small, for large enough $N$. For the same reason, this $N$ can be chosen uniform, and we can change the probabilities in the proofs to make the statements trivial for these small $N$. We are then able to perform Taylor expansions at the first order in $\a(N)/N$ (or $\a(N)/n$), and whatever constant hidden inside the $O(\cdot)$ will depend only on the fixed parameters $c,m,M$.

\subsection{Evolution of the random graph}

Before diving into the heart of our problem, let us study the activated arms in the graph $\cG_t(\th)$. At each time $t$, each arm is activated independently with probability $p_t = 1-e^{-t}$. The law of large numbers then readily tells approximately how many arms are activated on the vertices of degree $r$, how many activated arms there are in total, and so on. To be precise, we will be interested in the following quantities.
\begin{itemize}
\item $\fV_{k,r}(\th,t)$ is the number of $(k,r)$ vertices at time $t$;
\item $\fG_{\th}(t)$ is its generating function;
\item $\fB_{\th}(t) = \d{\fG\n_{\th}(t)}{x}(1,1)$ is the number of activated arms at time $t$;
\item $\vs(\th,B)$ is the first time $t$ when $\fB_{\th}(t) = B$.
\end{itemize}
Finally, we let $q_t(x,y) = (p_t x + 1 - p_t) y$.

\begin{lemma} \label{lem:charactgraph}
Assume that \eqref{eq:alpha} holds, and consider $\th = (n,\mu) \in \K\N_{c,m,M}$ and $\dl > 0$. Then, for some constant $K = K(c,m,M,\k)$, with probability at least $1 - K p_N$, it holds that
\[
\left \| \frac1n \fG_{\th}(t) - G_{\mu}(q_t)  \right \|_2 \leq \dl \frac{\a(N)}{N}
\]
and
\[
\left | \frac1n \fB_{\th}(t) - p_t m^{\mu}_1 \right | \leq  \dl \frac{\a(N)}{N}.
\]
In particular, for any $B \in [0, (1-m) m_1^{\mu} n]$,
\[
\left | \vs(\th,B) + \log \left ( 1 - \frac{B}{n m_1^{\mu}}\right ) \right | \leq \dl \frac{\a(N)}{N}.
\]
\end{lemma}

Note that the last quantity is only relevant when $B \in [0,n m_1^{\mu}]$, and we need to assume furthermore that $B/(n m_1^{\mu})$ is bounded away from 1 to allow uniform Taylor expansions. 

\begin{proof}
This is again Lemma \ref{lem:momentcomp} with $a_n = \a(N)/N$, and $p = p_t$. It implies that
\[
\sum_{r \geq 1} \sum_{k=0}^r r^2 \left | \frac1n \fV_{k,r}(\th,t) - \mu\n(r) [p_t]_k^r \right | \leq \dl \frac{\a(N)}{N}
\]
w.h.p. The first part is then a simple computation and the second and third part follow easily.
\end{proof}

\subsection{Gelation time and properties of the gel}

The following result concerns the time of the appearance of the large cluster, and its structure. It will be our main tool to prove precise estimates on the gelation times in our model, and on the evolution of the characteristics of the particles remaining in solution. We consider the following quantities.
\begin{itemize}
\item $\s\N(\th)$ is the time when a large component, i.e. a component of size $\geq \a(N)$ appears;
\item $s\N(\th)$ is the size of that large component;
\item $v_{k,r}\N(\th)$ is the number of $(k,r)$ vertices in that large component, and $\psi\N_{\th}$ its generating function;
\item $\phi\N_{\th} = \psi\N_{\th}(1,\cdot)$ is the generating function of the number of vertices with degree $r$ in that large component;
\item $b\N(\th) = \d{\psi_{\th}\N}{x}(1,1)$ is the number of activated arms in that large component;
\item $w\N(\th) = \sum_{k,r \geq 0} r^3 v_{k,r}\N(\th)$.
\end{itemize}

\begin{prop} \label{prop:atgel}
Assume that \eqref{eq:alpha} holds, and consider $\th = (n,\mu) \in \K\N_{c,m,M}$ and $\k > 0$. Then, for some constant $K = K(c,m,M,\k)$, with probability at least $1 - K p_N$, the following statements hold.
\begin{enumerate}
\item The gelation time verifies
\[
\left | \s\N(\th) + \log \left ( 1 - \frac{m^{\mu}_1}{m^{\mu}_2} \right ) - \frac{\a(N)}{n} \frac{m^{\mu}_3}{2 m^{\mu}_2 (m^{\mu}_2 - m^{\mu}_1)} \right | \leq \k \frac{\a(N)}{N}.
\]
\item In particular, the characteristics of the particles at this time is described by
\[
\left \| \frac1n \fG_{\th}(\s\N(\th)) - G_{\mu}(q_{\mu}) - \frac{\a(N)}{n} \frac{m^{\mu}_3}{2 (m^{\mu}_2)^2} (x-1) y G_{\mu}'(q_{\mu}) \right \| \leq \k \frac{\a(N)}{N},
\]
and the total number of activated arms verifies
\[
\left | \frac1n \fB_{\th}(\s\N(\th)) - \frac{(m^{\mu}_1)^2}{m^{\mu}_2} - \frac{\a(N)}{n} \frac{m^{\mu}_1 m^{\mu}_3}{2 (m^{\mu}_2)^2} \right | \leq \k \frac{\a(N)}{N}.
\]
\item The size of the large component enjoys
\[
\a(N) \leq s\N(\th) \leq (1 + \k) \a(N). 
\]
\item The structure of the particles in the large component is described by
\[
\left \| \frac{1}{\a(N)} \psi\N_{\th} - \frac{1}{m^{\mu}_1} x y G_{\mu}'(q_{\mu}) \right \|_2 \leq \k.
\]
In particular,
\[
\left \| \frac{1}{\a(N)} \phi\N_{\th} - \frac{1}{m^{\mu}_1} F_{\mu} \right \|_2 \leq \k
\]
and
\[
\left | \frac{1}{\a(N)} b\N(\th) - 2 \right | \leq \k.
\]
\item Finally
\[
\left | \frac1n w\N(\th) \right | \leq \k.
\]
\end{enumerate}
\end{prop}

\begin{proof}
Let
\[
\eps_n^- = \frac{m}{4 M^2} \frac{1}{n} \inf_{n \leq k \leq n/c} \a(k), \quad \eps_n^+ = \frac{M}{m^2} \frac{1}{n} \sup_{n \leq k \leq n/c} \a(k).
\]
In the following, we shall consider $\eps_n$ of the form
\[
\eps_n = \eps_{n,\pm} = \frac{\a(N)}{n} \frac{m^{\mu}_3}{2m^{\mu}_1m^{\mu}_2} (1 \pm \k)
\]
where we can assume $\k \in (0,1/2)$. According to the moment assumptions, $\eps_n^- \leq \eps_n \leq \eps_n^+$. Moreover, since $(\a(N))$ is such that \eqref{eq:alpha} holds, then $(\eps_n^{\pm})$ verifies Assumptions \eqref{eq:assepsnpm}. As usual, $C$ will be a constant that might change from line to line but only depends on the fixed parameters, and we let $\dl = \k / (2C)$ and
\[
\s\N_{\pm}(\th) = - \log \left ( 1 - \frac{m^{\mu}_1}{m^{\mu}_2}\right ) - \log \left ( 1 - (1 \pm \dl) \frac{\a(N)}{n} \frac{m^{\mu}_3}{2 m^{\mu}_2 (m^{\mu}_2 - m^{\mu}_1)} \right ).
\]
We may then take $K$ as in Theorem \ref{th:JLimproved3}, which depends only on $\dl$, $\eta$, $m$, $M$ and $(\eps_n^{\pm})$, and thus only on $\k$, $\eta$, $m$, $M$, $c$ and $(\a(N))$. All the conclusions of the corollary will thus hold with probability at least
\[
1 - K \left ( \frac{1}{\sqrt{n \eps_n^3}} + \eps_n^{1+\eta} \right ),
\]
and thus w.h.p. 

\begin{enumerate}[fullwidth]
\item Denote
\[
p^{\pm}_n = p_{\s\N_{\pm}(\th)} = \frac{m^{\mu}_1}{m^{\mu}_2} (1 + \eps_{n,\pm}).
\]
Now, note that at any time, $\cG_t(\th)$ is exactly a $CM(n,\mu,p_t)$ graph. Therefore, $\s\N(\th) \geq \s\N_+(\th)$ means that $CM(n,\mu,p_n^+)$ has no large component. But Theorem \ref{th:JLimproved3} implies that, w.h.p., 
\begin{align*}
\left | \cC_1(n,\mu,p_n^+) \right | & \geq \left ( 2 \frac{m^{\mu}_1}{m^{\mu}_2 m^{\mu}_3} - \dl \right ) n \frac{\a(N)}{n} \frac{m^{\mu}_3}{2 m^{\mu}_1 m^{\mu}_2} (1 + \k) \\
& = \a(N) \left ( 1 + \k - \dl (1 + \k) \frac{m^{\mu}_3}{2 m^{\mu}_1 m^{\mu}_2} \right ) \\
& \geq \a(N) \left ( 1 + \k - C \dl \right ) \\
& = \a(N) ( 1 + \k / 2).
\end{align*}
Hence, $\s\N(\th) \leq \s\N_+(\th)$ w.h.p., and similarly, $\s\N(\th) \geq \s\N_-(\th)$ w.h.p. Then, an easy Taylor expansion allows to write
\[
\s\N_{\pm}(\th) = - \log \left ( 1 - \frac{m^{\mu}_1}{m^{\mu}_2}\right ) + (1 \pm \dl) \frac{\a(N)}{n} \frac{m^{\mu}_3}{2 m^{\mu}_2 (m^{\mu}_2 - m^{\mu}_1)} + O \left ( \frac{\a(N)}{N} \right )^2,
\]
and Point 1 of the result follows.

\item Since $\fG_{\th}$ is increasing in $t$, from the point above, w.h.p.
\[
\fG_{\th}(\s\N_-(\th)) \leq \fG_{\th}(\s\N(\th)) \leq \fG_{\th}(\s\N_+(\th)).
\]
On the other hand, by Lemma \ref{lem:charactgraph}, w.h.p.
\[
\frac1n \fG_{\th}(\s\N_+(\th)) \leq G_{\mu}(q_{\s\N_+(\th)}) + \k \frac{\a(N)}{N}.
\]
But a simple Taylor expansion shows that
\[
q_{\s\N_+(\th)}(x,y) = q_{\mu}(x,y) + \frac{\a(N)}{n} \frac{m^{\mu}_3}{2 (m^{\mu}_2)^2} (x - 1) y + \k \: O \left ( \frac{\a(N)}{N} \right ).
\]
Another Taylor expansion ensures that
\[
G_{\mu}(q_{\s\N_+(\th)}(x,y)) = G_{\mu}(q_{\mu}(x,y)) + \frac{\a(N)}{n} \frac{m^{\mu}_3}{2 (m^{\mu}_2)^2} (x - 1) y G_{\mu}' (q_{\mu}(x,y)) + \k \: O \left ( \frac{\a(N)}{N} \right ).
\]
But these computations actualy hold in the $\| \cdot \|_2$ sense, as well as for $\fG_{\th}(\s\N_-(\th))$, and the result for $\fG_{\th}(\s\N(\th))$ follows. The second part is then a simple computation.

\item Note that the size of the largest component in $\cG_t(\th)$ is increasing with $t$. But we just saw that w.h.p., $\s\N(\th) \leq \s\N_+(\th)$, so that
\[
\s\N(\th) \leq \left | \cC_1(n,\mu,p_n^+) \right |
\]
w.h.p. A similar computation as in Point 1 above ensures that this size is bounded by $1+\k + C\dl \leq 1 + 2 \k$ w.h.p.

\item Denote by $\psi_{\th}^{\pm}$ the generating function of $(v_{k,r} ( \cC_1(n,\mu,p_n^{\pm})))$. Theorem \ref{th:JLimproved3} implies that w.h.p.,
\[
\left \| \frac{1}{n \eps_{n,+}} \psi_{\th}^+ - 2 \frac{m^{\mu}_2}{m^{\mu}_3} xy G_{\mu}'(q_{\mu}) \right \|_2 \leq \dl.
\]
Rearranging, it is easy to see that
\[
\left \| \frac{1}{\a(N)} \psi_{\th}^+ - \frac{1}{m^{\mu}_1} xy G_{\mu}'(q_{\mu}) \right \|_2 \leq C \dl \leq \k.
\]
A similar inequality holds for $\psi_{\th}^-$. But
\[
v_{k,r} \left ( \cC_1(n,\mu,p_n^-) \right ) \leq v_{k,r}\N(\th) \leq v_{k,r} \left ( \cC_1(n,\mu,p_n^+) \right )
\]
w.h.p., so that
\[
\left \| \frac{1}{\a(N)} \psi\N_{\th} - \frac{1}{m^{\mu}_1} xy G_{\mu}'(q_{\mu}) \right \|_2 \leq \k.
\]
The other formulas are obtained after easy computations.

\item Finally, for the same reason,
\[
\frac1nw\N(\th) \leq \frac1n \sum_{k,r \geq 0} r^3 v_{k,r}(\cC_1(n,\mu,p_n^+)) \leq \k
\]
w.h.p., and the proof is over. \qedhere
\end{enumerate}
\end{proof}

\subsection{After gelation}

The previous result precisely describes when gelation occurs and what happens for the gel at this time. It therefore allows us to explain what remains \emph{in solution} after a gelation event. Let us then start by defining the following quantities. Note that the notation is the same as above, but with capital letters.
\begin{itemize}
\item $S\N(\th) = n - s\N(\th)$ is the number of vertices remaining in solution at $\s\n$;
\item $V_{k,r}\N(\th) = \fV_{k,r}(\th,\s\N(\th)) - v_{k,r}\N(\th)$ is the number of $(k,r)$ vertices in solution at $s\N(\th)$, and $\Psi\N_{\th}$ is its generating function;
\item $\Phi\N_{\th} = \Psi\N_{\th}(1,\cdot)$ is the generating function of the number of vertices with degree $r$ remaining in solution at $\s\n(\th)$;
\item $B\N(\th) = \d{\Psi\N_{\th}}{x}(1,1)$ is the number of activated arms remaining in solution at $\s\n(\th)$.
\end{itemize}
We then have the following result, which is a direct consequence of Lemma \ref{lem:charactgraph} and Proposition \ref{prop:atgel}.

\begin{lemma} \label{lem:aftergel}
Assume that \eqref{eq:alpha} holds, and consider $\th = (n,\mu) \in \K\N_{c,m,M}$ and $\k > 0$. Then, for some constant $K = K(c,m,M,\k)$, with probability at least $1 - K p_N$, the following statements hold.
\begin{enumerate}
\item The number of particle remaining in solution enjoys
\[
n - (1 + \k) \a(N) \leq S\N(\th) \leq n - \a(N). 
\]
\item The structure of the particles remaining in solution is described by
\[
\begin{split}
\left \| \frac1n \Psi\N(\th) - G_{\mu}(q_{\mu}) - \frac{\a(N)}{n} \left ( \frac{m^{\mu}_3}{2 (m^{\mu}_2)^2} (x-1)y G_{\mu}'(q_{\mu}) \right. \right. & \left. \left. \vphantom{\frac1n \Psi\N(\th) - G_{\mu}(q_{\mu}) - \frac{\a(N)}{n} \frac{m^{\mu}_3}{2 (m^{\mu}_2)^2} (x-1)y G_{\mu}'(q_{\mu})} - \frac{1}{m^{\mu}_1} xy G_{\mu}'(q_{\mu}) \right ) \right \|_2 \\
& \leq \k \frac{\a(N)}{N}.
\end{split}
\]
\item In particular,
\[
\left \| \frac1n \Phi\N(\th) - G_{\mu} + \frac{\a(N)}{n} \frac{1}{m^{\mu}_1} F_{\mu} \right \|_2 \leq \k \frac{\a(N)}{N}
\]
and
\[
\left | \frac1n B\N(\th) - \frac{(m^{\mu}_1)^2}{m^{\mu}_2} - \frac{\a(N)}{n} \left ( \frac{m^{\mu}_1 m^{\mu}_3}{2 (m^{\mu}_2)^2} - 2 \right ) \right | \leq \k \frac{\a(N)}{N}.
\]
\end{enumerate}
\end{lemma}

\subsection{Properties of the new model}

After gelation, we are left with particles in solution described by the above result. In particular, there are $S\N(\th)$ of them, and their total degree is described by $\Psi\N(\th)$. It thus makes sense to ask what happens for a configuration model with these parameters. It will turn out to be especially useful when we consider the alternative model in the next section. To be precise, we consider $S\N(\th)$ particles, with degree distribution $\bar{\mu}\N(\th)$, which has generating function $\bar{G}\N_{\th} = \Phi\N(\th)/S\N(\th)$. We write $\bar{\th}\N = (S\N(\th),\bar{\mu}\N(\th))$. The factorial moments of $\bar{\mu}\N(\th)$ are written $\bar{m}\N_i(\th)$, and we denote the time when a giant component appears in $\cG(\bar{\th}\N)$ as $\bar{\s}\N(\th)$. We can then prove the following result.

\begin{lemma} \label{lem:mubar}
Assume that \eqref{eq:alpha} holds, and consider $\th = (n,\mu) \in \K\N_{c,m,M}$ and $\k > 0$. Then, for some constant $K = K(c,m,M,\k)$, with probability at least $1 - K p_N$, the following statements hold.
\begin{enumerate}
\item The generating function of $\bar{\mu}\N(\th)$ verifies
\[
\left \| \bar{G}\N_{\th} - G_{\mu} - \frac{\a(N)}{n} \left ( G_{\mu} - \frac{1}{m^{\mu}_1} F_{\mu} \right ) \right \|_2 \leq \k \frac{\a(N)}{N}.
\]
\item The factorial moments of $\bar{\mu}\N(\th)$ enjoy
\[
\left | \bar{m}\N_1(\th) - m^{\mu}_1 - \frac{\a(N)}{n} \left ( m^{\mu}_1 -  \frac{m^{\mu}_1 + m^{\mu}_2}{m^{\mu}_1} \right ) \right | \leq \k \frac{\a(N)}{N},
\]
and
\[
\left | \bar{m}\N_2(\th) - m^{\mu}_2 - \frac{\a(N)}{n} \left ( m^{\mu}_2 - \frac{ 2 m^{\mu}_2 + m^{\mu}_3}{m^{\mu}_1} \right ) \right | \leq \k \frac{\a(N)}{N},
\]
and
\[
\left | \bar{m}\N_3(\th) - m^{\mu}_3 \right | \leq \k.
\]
\item The time when a giant component appears in $\cG(\bar{\th}\N)$ verifies
\[
\left | \bar{\s}\N(\th) + \log \left ( 1 - \frac{m^{\mu}_1}{m^{\mu}_2} \right ) - \frac{\a(N)}{n} \left ( \frac{3m^{\mu}_3}{2m^{\mu}_2 ( m^{\mu}_2 - m^{\mu}_1) }- \frac{1}{m^{\mu}_1} \right ) \right | \leq \k \frac{\a(N)}{N}.
\]

\end{enumerate}
\end{lemma}

\begin{proof}
\begin{enumerate}[fullwidth]
\item We know from the previous result that w.h.p.
\begin{align*}
\bar{G}\N_{\th} & = \frac{\Phi\N(\th)}{S\N(\th)} \\
& = \frac{1}{S\N(\th)/n} \frac{1}{n} \Phi\N(\th) \\
& = \frac{1}{1 - \a(N)/n + \k \: O(\a(N)/n)} \left ( G_{\mu} - \frac{\a(N)}{n} \frac{1}{m^{\mu}_1} F_{\mu} + \k \: O \left (\frac{\a(N)}{n} \right ) \right ) \\
& = \left ( 1 + \frac{\a(N)}{n} + \k \: O(\a(N)/n) \right ) \left ( G_{\mu} - \frac{\a(N)}{n} \frac{1}{m^{\mu}_1} F_{\mu} + \k \: O \left (\frac{\a(N)}{n} \right ) \right ) \\
& = G_{\mu} + \frac{\a(N)}{n} \left ( G_{\mu} - \frac{1}{m^{\mu}_1} F_{\mu} \right ) + \k \: O \left (\frac{\a(N)}{n} \right ),
\end{align*}
and this holds in the $\| \cdot \|_2$ sense. The first part of the result follows. The estimations of $\bar{m}\N_1(\th) = (\bar{G}_{\th}\N)'(1)$ and $\bar{m}\N_2(\th) = (\bar{G}_{\th}\N)''(1)$ are then easy computations.

\item For the third moment, note that by Proposition \ref{prop:atgel},
\begin{align*}
\bar{m}\N_3(\th) & = \frac{1}{S\N(\th)} \sum_{k,r \geq 0} r^3 V_{k,r}\N(\th) \\
& =  \frac{1}{S\N(\th)} \sum_{k,r \geq 0} r^3 \fV_{k,r}(\th,\s\N(\th)) - \frac{1}{S\n} \sum_{k,r \geq 0} r^3 v_{k,r}\N(\th) \\
& = \frac{n}{S\N(\th)} \left ( m^{\mu}_3 - \frac1n  w\N(\th) \right ) \\
& = \left ( 1 + \k \: O (1) \right ) \left ( m^{\mu}_3 + \k \: O(1) \right ) \\
& = m^{\mu}_3 + \k \: O (1),
\end{align*}
and the result follows.

\item Finally, the estimation of $S\N(\th)$ and $\bar{m}\N_i(\th)$ above readily shows that the assumptions of Proposition \ref{prop:atgel} are still in force, maybe at the price of considering different constants. Therefore, Proposition \ref{prop:atgel} applies, and shows that w.h.p.
\[
\bar{\s}\N(\th) = - \log \left ( 1 - \frac{\bar{m}\N_1(\th)}{\bar{m}\N_2(\th)} \right ) + \frac{\a(N)}{S\N(\th)} \frac{\bar{m}\N_3(\th)}{2 \bar{m}\N_2(\th) (\bar{m}\N_2(\th) - \bar{m}\N_1(\th))} + \k \: O \left ( \frac{\a(N)}{N} \right ).
\]
A couple lines of computation using the previous results show that
\[
\frac{\bar{m}\N_1(\th)}{\bar{m}\N_2(\th)} = \frac{m^{\mu}_1}{m^{\mu}_2} \left ( 1 + \frac{\a(N)}{n} \left ( \frac{1}{m^{\mu}_1} - \frac{m^{\mu}_2}{(m^{\mu}_1)^2} + \frac{m^{\mu}_3}{m^{\mu}_1 m^{\mu}_2} + \k \: O(1) \right ) \right ).
\]
A couple more lines ensure that
\[
\log \left ( 1 - \frac{\bar{m}\N_1(\th)}{\bar{m}\N_2(\th)} \right ) = \log \left ( 1 - \frac{m^{\mu}_1}{m^{\mu}_2} \right ) + \frac{\a(N)}{n} \left ( \frac{1}{m^{\mu}_1} - \frac{m^{\mu}_3}{m^{\mu}_2 ( m^{\mu}_2 - m^{\mu}_1)} \right ) + \k \: O \left ( \frac{\a(N)}{N} \right ).
\]
The corrections on the second term of $\bar{\s}\N(\th)$ are of the second order, and thus, one may remove the bars there at the price of another factor $\k \: O (\a(N)/n)$. An easy computation then allows to conclude.
\end{enumerate}
\end{proof}

\subsection{Non-degenerate conditioning}

Recall from Lemma \ref{lem:combiB2} that, right after a gelation event, conditionally on the particles remaining in solution and the number of activated arms, the structure in solution is that of a CM on these particles with this number of activated arms, conditioned on having no large component. It therefore makes sense to ask whether this conditioning matters or not. The main interest of the previous result is that, w.h.p., this conditioning does not matter. It will be an easy consequence of the following computation, which we will reuse later.

\begin{lemma} \label{lem:diffgeltimes}
Assume that \eqref{eq:alpha} holds, and consider $\th = (n,\mu) \in \K\N_{c,m,M}$ and $\k > 0$. Then, for some constant $K = K(c,m,M,\k)$, with probability at least $1 - K p_N$,
\[
\left | \vs(\bar{\th},B\N(\th)) + \log \left ( 1 - \frac{m_1^{\mu}}{m_2^{\mu}} \right ) - \frac{\a(N)}{n} \left ( \frac{m_3^{\mu}}{2 m_2^{\mu} (m_2^{\mu} - m_1^{\mu})} - \frac{1}{m_1^{\mu}} \right ) \right | \leq \k \frac{\a(N)}{N}.
\]
In particular,
\[
\left | \bar{\s}\N(\th) - \vs(\bar{\th},B\N(\th)) - \frac{\a(N)}{n} \frac{m_3^{\mu}}{m_2^{\mu}(m_2^{\mu} - m_1^{\mu})} \right | \leq \k \frac{\a(N)}{N}.
\]
\end{lemma}

\begin{proof}
From Lemma \ref{lem:mubar}, we know $B\N(\th)$, $S\N(\th)$ and $\bar{m}\N_i(\th)$ with good precision. Easy computations allow to obtain
\[
\frac{B\N(\th)}{S\N(\th) \bar{m}\N_1(\th)} = \frac{m_1^{\mu}}{m_2^{\mu}}  \left ( 1 + \frac{\a(N)}{n} \left ( \frac{m_3^{\mu}}{2 m_1^{\mu} m_2^{\mu}} + \frac{m_1^{\mu} - m_2^{\mu}}{(m_1^{\mu})^2} \right ) \right ) + \k O \left ( \frac{\a(N)}{N} \right ).
\]
Therefore, w.h.p., and maybe at the price of taking smaller constants,
\[
B\N(\th)/ (S\N(\th) \bar{m}\N_1(\th)) < 1 - m,
\]
and, as we mentioned in the previous proof, $\bar{\th} \in \K\N_{c,m,M}$. We may then use Lemma \ref{lem:charactgraph} to write
\[
\vs(\bar{\th},B\N(\th)) = - \log \left ( 1 - \frac{B\N(\th)}{S\N(\th) \bar{m}_1\N(\th)} \right ) + \k O \left ( \frac{\a(N)}{N} \right )
\]
and the result follows after easy computations using Lemma \ref{lem:aftergel}. The second part is then a direct consequence of Point 3 of Lemma \ref{lem:aftergel}.
\end{proof}

\begin{prop} \label{prop:trivcond}
Assume that \eqref{eq:alpha} holds, and consider $\th = (n,\mu) \in \K\N_{c,m,M}$ and $\k > 0$. Then, for some constant $K = K(c,m,M,\k)$, conditionally on $\bar{\th}\N$ and $B\N(\th)$, a configuration model $CM(\bar{\th}\N,B\N(\th))$ has no large component with probability at least $1 - K p_N$.
\end{prop}

\begin{proof}
From the previous result, w.h.p.
\[
\bar{\s}\N(\bar{\th}) - \vs(\bar{\th},B\N(\th)) \geq \frac{\a(N)}{n} \left ( \frac{m_3^{\mu}}{m_2^{\mu}(m_2^{\mu} - m_1^{\mu})} - \k \right ).
\]
The fact that $\th \in \K\N_{c,m,M}$ ensures that, if $\k$ is chosen small enough, then this quantity is positive. Therefore, w.h.p., gelation in a $\cG(\bar{\th}\N)$ occurs strictly after the time when $B\N(\th)$ arms are activated, and thus a $CM(\bar{\th}\N,B\N(\th))$ has w.h.p. no large component.
\end{proof}

%% file: 7-Altmodel.tex
\section{Alternative model} \label{sec:altmodel}

After these two sections concerning random graphs, we come back to our model, and start putting the pieces together. The goal of this section is to present an alternative model which is easier to study than ours since it has a much nicer combinatorial structure. We will prove however that w.h.p. the two models have the same distribution on compact time intervals. 

According to Lemma \ref{lem:combiB3}, at any gelation time $\tau$ and conditionally on the number of particles in solution $n = |S(\tau)|$, their empirical measure of degrees $\mu = \mu_{\tau}$ and their number of activated arms $B = B(\tau)$, the configuration in solution is that of conditioned configuration model $CM'(n,\mu,B)$. Informally speaking, our goal is to somehow get rid of the conditioning. 

We introduce below a slightly more complicated process, which is however merely built from DCM. Encoded in that process are two processes: our original one (at least, its state in solution), and another one which can be more easily studied. In other words, we have a coupling between our process and a more simple one: this is the purely combinatorial Lemma \ref{lem:equivproc}. The fundamental fact is that the processes are actually equal w.h.p., as proved in Lemma \ref{lem:equivasymp}. This second statement makes heavy use of what was proven in Proposition \ref{prop:trivcond}. This construction and the results are very similar to those presented in Section 5 of \cite{MN0}, to which we refer for a more elementary exposition.

For each $N \geq 1$, consider a family $(\cG\N(\th,k))$ of independent DCM, indexed by $k \geq 1$ and $\th = (n,\mu)$, with $n \in [N]$ and $\mu \in \cM\n$ (these were defined in Section \ref{sec:prelim}). Note that there is only a countable number of parameters and there is thus no issue considering this family. For each such $\th$ and $k$, $\cG\N(\th,k)$ has the distribution of $\cG(\th)$. We shall use the same notation as in Section \ref{sec:JL}, but adding a parameter $k$ to insist that we are dealing with $\cG\N(\th,k)$. Specifically, we let
\begin{itemize}
\item $\vs\N(\th,R,k)$ to be the first time when there are $R$ activated arms in $\cG\N(\th,k)$;
\item $\s\N(\th,k)$ be the time when a component of size $\geq \a(N)$ appears in $\cG\N(\th,k)$;
\item $s\N(\th,k)$ the size of that large component;
\item $B\N(\th,k)$ the number of activated arms remaining in solution at $\s\N(\th,k)$;
\item $\bar{\mu}\N(\th,k)$ the empirical distribution of degree of the particles remaining on solution;
\item $\bar{\th}\N(k) = (S\N(\th,k),\bar{\mu}\N(\th,k))$.
\end{itemize}

Similarly as in \cite{MN0}, let us construct the process $(Y(t))$ as follows. 
\begin{description}
\item[\textbf{Step 0}] Let $\rho_0 = (N,\mu\N)$, $t_0 = 0$, $D_0 = 0$, and $K(0) = 1$. Consider $\cG\N(\rho_0,K(0))$, up to the time $t_1 = \s(\rho_0,K(0))$ when a large component appears. Then let $\rho_1 = \bar{\rho_0}\N(K(0))$, $D_1 = B\N(\rho_0,K(0))$, define $Y(s) = \cG\N_s(\rho_0,K(0))$ for $s \in [0,t_1)$, and go to step 1.
\item[\textbf{Step i}] Consider here the graph processes $\cG\N(\rho_i,k)$ for $k \geq 1$. Define $K(i)$ to be the first $k$ such that $\cG\N_{\vs\N(\rho_i,D_i,k)}(\th_i,k)$ has no large component. Then take
\begin{equation} \label{eq:tip1}
t_{i+1} = t_i + \s\N(\rho_i,K(i)) - \vs\N(\rho_i,D_i,K(i)),
\end{equation}
as well as $\rho_{i+1} = \bar{\rho_i}\N(K(i))$, $D_{i+1} = B\N(\rho_i,K(i))$, define
\begin{equation} \label{eq:Ys}
Y(s) = \cG\N_{\vs\N(\rho_i,D_i,k) - t_i + s}(\rho_i,K(i)), \quad s \in [t_i,t_{i+1}),
\end{equation} 
and go to step $i+1$.
\end{description}
Obviously, we stop and let $t_{i+1} = \pinf$ when no more large component can be created.

An important thing to notice, and the main difference with \cite{MN0}, are Equations \eqref{eq:tip1} and \eqref{eq:Ys}. To understand where it comes from, let us explain the construction in more details. At Step 0, we are just looking at a DCM $\cG\N(\rho_0,1)$ until the gelation time $t_1$, just like our original process. Thanks to Lemma \ref{lem:combiB3}, we know that, right after gelation, what remains in solution has the distribution of a $CM'(\rho_1,D_1)$. Now, we want to replicate this latter graph, thanks to some DCM $\cG\N(\rho_1,k)$. To this end, we look at these processes up to the time $\vs\N(\rho_1,D_1,k)$ when $D_1$ arms are activated. At this time, we obtain graphs with distribution $CM(\rho_1,D_1)$. We choose the $K(1)$-th, the first one which has no giant component, and which has therefore the conditioned distribution $CM'(\rho_1,D_1)$. By the Markov property, the evolution of our original process and of $\cG\N(\rho_1,K(1))$ afterwards are similar.

In particular, note that we are dealing with two different time lines here: the one of $\cG\N(\rho_0,1)$ and the one of $\cG\N(\rho_1,K(1))$. We look at the first process on the interval $[0,t_1)$, and at the second one on the interval $[\vs\N(\rho_1,D_1,k),\s(\rho_i,K(i)))$. Therefore, Equations \eqref{eq:tip1} and \eqref{eq:Ys} merely allow to define our new process $Y$ by pasting these two parts one right after the other. In the definition of \cite{MN0}, the times were automatically matching and no such treatment was necessary.

\begin{remark}
It might seem more natural to condition on $N_t$ and $\pi_t$ and use Lemma \ref{lem:combipi}. However, the whole trick of the previous construction is to be able to replicate a $CM'(\rho_1,D_1)$ thanks to a DCM, so as to be able to study the dynamic afterwards. As we saw, it suffices to stop at the right time. But, even replicating a $CM(N_t,\pi_t)$ does not seem possible, since there is possibly no time when we get exactly the distribution of activated arms $\pi_t$. This explains the slightly bizarre conditioning on the total number of activated arms.
\end{remark}

By iterating this reasoning, it should be clear that we have constructed a process which evolves in the same way as our original process in solution, at least in between two gelation times. Our construction does not ensure continuity (in terms of graph structure) at the gelation events, but it does not matter for all intents and purposes. Doing so is possible, but would further complicate the definition. Therefore, with a slight abuse of notation, for two processes $U, V$ and two sequences of random times $(a_i), (b_i)$, we define
\[
((a_i), U) \eqlaw ((b_i), V)
\]
if 
\[
((a_i), U_{(a_i)})) \eqlaw ((b_i), V_{(b_i)})),
\]
where, for an interval $I$, $U_I = (U(t), t \in I)$ and $U_{(a_i)} = (U_{[a_i,a_{i+1})},i \geq 0)$. In the following result, recall that $(G_S(t))$ is the configuration in solution of our process and $(\tau_i)_{i \geq 1}$ is the sequence of gelation times. We also let $\tau_0 = 0$.

\begin{lemma} \label{lem:equivproc}
The equality
\[
((\tau_i),G_S) \eqlaw ((t_i),Y)
\]
holds.
\end{lemma}

\begin{proof}
Let us write the parameters of our original model as $\xi_i = (N_{\tau_i}, \mu_{\tau_i})$, $B_i = B _{\tau_i}$ and $C_i = (G_S)_{[\tau_i,\tau_{i+1})}$. In particular $\xi_0 = (N,\mu\N) = \rho_0$. Now, our process is, before gelation, just a $\cG(\xi_0)$, so has the distribution of $\cG\N(\rho_0,1)$. In particular, the time of the appearance of a large component and the properties of what remains in solution are the same as for $Y$, i.e.
\[
[\tau_0, \tau_1, \xi_1, B_1, C_0] \eqlaw [t_0, t_1, \rho_1, D_1, Y_{[t_0,t_1)}].
\]
Let us then prove by induction that for every $i \geq 1$, $\cH_i$ holds, where $\cH_i$ is the assumption

\[
([\tau_{j-1}, \tau_j, \xi_j, B_j, C_{j-1}, 1 \leq j \leq i) \eqlaw ([t_{j-1}, t_j, \rho_j, D_j, Y_{[t_{j-1},t_j)}], 1 \leq j \leq i).
\]
We just checked $\cH_1$, so assume that $\cH_i$ holds for some $i \geq 1$. On the one hand, we know from Lemma \ref{lem:combiB3} that
\[
G_S(\tau_i) \eqlaw CM'(\xi_i,B_i).
\]
Now, from the construction, $Y(t_i)$  is constructed from the graphs $\cG\N_{\vs\N(\rho_i,D_i,k)}(\rho_i,k)$. Each has distribution $CM(\rho_i,D_i)$, and we choose the $K(i)$-th graph, the first which has no large component. It is thus distributed as a $CM'(\rho_i,D_i)$ graph. By $\cH_i$, this has the same distribution as a $CM'(\xi_i,B_i)$ graph, and thus
\[
Y(t_i) \eqlaw CM'(\rho_i,D_i) \eqlaw CM'(\xi_i,B_i) \eqlaw G_S(\tau_i).
\]
Then, by Markov property, $G_S(\tau_i + t)_{t \geq 0}$ and $Y(t_i+t)_{t \geq 0}$ evolve in the same way until the next gelation event, and $\cH_{i+1}$ follows.
\end{proof}

Hence, from now on, we can assume that $(G_S(t))$ is constructed as above. In particular, it is coupled with the process $(Z(t))$ that we introduce now. We write $\cG\N(\th) = \cG\N(\th,1)$.
\begin{description}
\item[\textbf{Step 0}] Let $\th_0 = (N,\mu\N)$, $\s_0 = 0$ and $A_0 = 0$. Consider $\cG\N(\th_0)$, up to the time $\s_1 = \s(\th_0,1)$ when a large component appears. Then let $\th_1 = \bar{\th_0}\N(1)$, $A_1 = B\N(\th_0,1)$, define $Z(t) = \cG\N_t(\th_0)$ for $t \in [0,\s_1)$, and go to step 1.
\item[\textbf{Step i}] Consider here the graph processes $\cG\N(\th_i)$. Then take
\begin{equation} \label{eq:sigmaip1}
\s_{i+1} = \s_i + \left (\s\N(\th_i,1) - \vs\N(\th_i,A_i,1) \right ) \vee 0,
\end{equation}
as well as $\th_{i+1} = \bar{\th_i}\N(1)$, $A_{i+1} = B\N(\th_i,1)$, and define
\begin{equation} \label{eq:Zs}
Z(s) = \cG\N_{\vs\N(\rho_i,D_i,1) - \s_i + s}(\th_i,1), \quad s \in [\s_i,\s_{i+1}),
\end{equation}
where this interval might be empty if $\s_{i+1} = \s_i$. Then go to step $i+1$.
\end{description}
In other words, we consider the same model as above, but taking always $K(i) = 1$. It should be clear that $(Z(t))$ is easier to study than $(Y(t))$. Our next claim is that under \eqref{eq:alpha}, $Z$ and $Y$ are barely any different, at least on compact intervals. To this end, define
\[
I = \inf \{ i \geq 1, K(i) \neq 1 \}, \quad E(T) = \{t_I > T\},
\]
so that $E(T)$ is the event that $K(i) = 1$ for all the $\cG\N$ we consider before time $T$. Hence, on $E(T)$, the processes $(Y(t), t \in [0,T])$ and $(Z(t), t \in [0,T])$ are equal. As usual, we will write $E\N(T)$ when we want to insist on the dependence on $N$. We will prove the following.

\begin{lemma} \label{lem:equivasymp}
If Assumptions \ref{as:mu5} and \ref{as:alpha} hold, then, for all $T \geq 0$, $\P(E\N(T)) \to 1$.
\end{lemma}

\begin{proof}
Take $c$, $m$, $M$ as in Lemma \ref{lem:alwaystKN}. We may assume that $m$ is small enough such that
\begin{equation} \label{eq:choicem}
- \log \left ( 1 - \frac{1}{1+m} \right ) - \frac{\a(N)}{N} > T.
\end{equation}
Consider $\k = 1$ for instance, and the corresponding $K = K(c,m/2,M,\k)$ such that all the results of Section \ref{sec:gelDCM} hold. Note that we take the constant corresponding to $m/2$ to leave ourselves a bit of leeway, what we will use to study $F\N(T,2)$ below. All that we write from now on is implied to happen with probability at least $1 - K p_N$, where we recall that $p_N$ is defined in Equation \eqref{eq:pN}.

We want to prove that $\P(E\N(T)) \to 1$. Define
\[
J = \inf \{ i \geq 1, \rho_i \notin \K\N_{c,m,M} \},
\]
and the event
\[
E\N(T,1) = \left \{ \forall t \in [0,T] \; (N_t, \mu\N_t) \in \tK\N_{c,m,M} \right \}
\]
considered in Lemma \ref{lem:alwaystKN}. There are three ways things can go wrong, i.e. an outcome $\om$ is such that $t_I(\om) \leq T$:
\begin{itemize}
\item either $E\N(T,1)$ does not occur;
\item or $E\N(T,1)$ occurs, but $J < I$ and $t_I \leq T$;
\item or $E\N(T,1)$ occurs, $J \geq I$, but $K(i) \neq 1$ for some $i \leq I$.
\end{itemize}
Let us call these events respectively $F\N(T,1) = E\N(T,1)^{\complement}$, $F\N(T,2)$ and $F\N(T,3)$. By choice of the constants $c$, $m$, $M$ as in Lemma \ref{lem:alwaystKN}, we already know that
\[
\P(F\N(T,1)) \to 0.
\]

Let us now consider $F\N(T,2)$. In this case, since $t_J \leq t_I \leq T$ and $E\N(T,1)$ occurs, then $\rho_J = (n_J,\mu_J)$ is in $\tK\N_{c,m,M}$, but not in $\K\N_{c,m,M}$, i.e. it verifies $n_J \geq c N$, $\sum_{r \geq 1} \mu_J(r) r^{5 + \eta} \leq M$, $\fm^{\mu_J}_3 \geq m$, but $\fm_2^{\mu_J}/\fm_1^{\mu_J} \leq 1 + m$. Moreover $\rho_{J-1} \in \K\N_{c,m,M}$. But recall Lemma \ref{lem:mubar}. It easily implies that there is a constant $C$ such that, for any $\rho =(n,\mu) \in \K\N_{c,m,M}$,
\[
\frac{\bar{\fm}\N_1(\rho)}{\bar{\fm}\N_2(\rho)} \geq \frac{\fm^{\mu}_1}{\fm^{\mu}_2} \left ( 1 - C \frac{\a(N)}{N} \right ).
\]
Since $\fm^{\mu}_2 / \fm^{\mu}_1 \geq 1 + m$, then $\bar{\fm}\N_1(\rho) / \bar{\fm}\N_2(\rho) \geq 1 + m/2$ for $N$ large enough. Therefore, since $\rho_{J-1} \in \K\N_{c,m,M}$, then
\[
\frac{\fm_2^{\mu_J}}{\fm_1^{\mu_J}} = \frac{\bar{\fm}\N_1(\rho_J)}{ \bar{\fm}\N_2(\rho_J)} \geq 1 + m/2.
\]
Hence $\rho_J \in \K\N_{c,m/2,M}$. According to Proposition \ref{prop:atgel} (with $\k = 1$), the corresponding gelation time verifies
\[
\s\N(\rho_J) \geq - \log \left ( 1 - \frac{\fm_1^{\mu_J}}{\fm_2^{\mu_J}} \right ) - \frac{\a(N)}{N}.
\]
As we mentioned, this is our only reason to choose the constant $K$ corresponding to $m/2$, not $m$. Now, since $\fm_2^{\mu_J}/\fm_1^{\mu_J} \leq 1 + m$ as we just mentioned, then $\s\N(\rho_J) > T$ by the choice of $m$ as in \eqref{eq:choicem}, so that $t_I > T$. To conclude, $\P(F\N(T,2)) \leq 1 - K p_N \to 0$.

Finally, on $F\N(T,3)$, all the $\rho_i$ we consider are in $\K\N_{c,m,M}$. Note that $K(i) = 1$ means that $\cG\N_{\vs(\rho_i,D_i,k)}(\rho_i,1)$ has no large component. But
\[
\cG\N_{\vs(\rho_i,D_i,k)}(\rho_i,1) \eqlaw CM(\rho_i,D_i) = CM(\bar{\rho_{i-1}}\N,B\N(\rho_{i-1}))
\]
and we know by Proposition \ref{prop:trivcond} that this graph has no large component with probability greater than $1 - K p_N$. Since there are at most $N/\a(N)$ gelation events, and recalling the definition of $p_N$, we get
\[
P(F\N(T,3)) \leq K \frac{N}{\a(N)} p_N = K \left ( \left ( \frac{N^{4/5}}{\a(N)} \right )^{5/2} + \left ( \frac{\a(N)}{N} \right )^{\eta} \right ) \to 0
\]
by \eqref{eq:alpha}. The result follows.
\end{proof}

Notice that in the very last computation we crucially use \eqref{eq:alpha}. This is the main reason for our assumption, just as in \cite{MN0}, we needed to assume slightly more on the threshold $(\a(N))$. As a by-product of the above proof, namely the fact that $\P(F\N(T,2)) \to 0$, we see that all the $(n_t,\mu_t)$ that we consider before time $T$ are well-behaved, in that they are in $\K\N_{c,m,M}$. Coupled with Lemma \ref{lem:equivproc}, this allows to obtain the equivalence of our model and the alternative model on any compact interval, along with the good behavior of the relevant quantities.

\begin{theorem} \label{th:equiv}
If Assumptions \ref{as:mu5} and \ref{as:alpha} hold, then, for all $T > 0$,
\[
\P \left ( (G_S(t), t \in [0,T]) = (Z(t), t \in [0,T]) \right ) \to 1.
\]
Moreover, there exists $c, m, M \in (0,\pinf)$ such that
\[
\P \left ( \forall t \in [0,T] \;  (N_t,\mu_t) \in \K\N_{c,m,M} \right ) \to 1.
\]
\end{theorem}

%% file: 8-Limit.tex
\section{Limit} \label{sec:limit}

This section puts together the results on the DCM and the alternative model to study precisely how our model evolves between gelation events. This culminates with the proof of the convergence of $p\N$, and the explicit characterization of its limit.

\subsection{Microscopic evolution}

Let us recall that $P_t(k,r)$ is the number of particles in solution at time $t$, with $k$ activated arms and degree $r$, and $\mu_t$ is the empirical distribution of degrees in solution at time $t$. Define
\[
\psi_t(x,y) = \frac1N \sum_{r \geq 0} \sum_{k=0}^r P_t(k,r) x^k y^r, \quad \phi_t(y) = \psi_t(1,y) = \frac{N_t}{N} G_{\mu_t}.
\]
Thanks to the alternative model and the results of Section \ref{sec:gelDCM}, we are able to describe precisely how $\psi_t$ and $\phi_t$ evolve at the microscopic scale. 

Recall that $(\tau_k)$ is the sequence of gelation times, and with a slight abuse of notation, denote $N_k = N_{\tau_k}$, $\mu_k = \mu_{\tau_k}$, $m_i^k = m_i^{\mu_k}$ and $\phi_k = \phi_{\tau_k}$. Note that all these quantities are constant on $[\tau_k,\tau_{k+1})$. We finally let $\psi_k = \psi_{\tau_{k+1}^-}$, that is, the left limit of $\psi$ at $\tau_{k+1}$.

In all this section, we will use freely the alternative model and Theorem \ref{th:equiv}. The most important point to remember is that, with high probability, our original model and the alternative model coincide before time $T$, all the $\th_k$ we consider before this time are in $\K\N_{c,m,M}$, and all the $K(i)$ are equal to 1. In particular, our model on the interval $[\tau_k,\tau_{k+1})$ corresponds to a DCM with parameter $\th_k$ on an interval $[\vs\N (\th_k,B_k,1),\s\N(\th_k,1))$. In all the proofs, $O ( \cdot )$ is implied to be independent the $k$ that we consider, and we imply that everything happens w.h.p.

\begin{lemma} \label{lem:evolution}
Assume that Assumptions \ref{as:mu5} and \ref{as:alpha} hold. Then, the following hold. For any $\k > 0$ and $T > 0$, there is a $K > 0$ such that, with high probability, for all $k \geq 1$ such that $\tau_k \leq T$, we have
\[
\frac{\a(N)}{N_k} \left ( \frac{m_3^{k-1}}{m_2^{k-1}(m_2^{k-1} - m_1^{k-1})} - \k \right ) \leq \tau_{k+1} - \tau_k \leq \frac{\a(N)}{N_k} \left ( \frac{m_3^{k-1}}{m_2^{k-1}(m_2^{k-1} - m_1^{k-1})} - \k \right )
\]
and
\[
\left \| \phi_{k+1} - \phi_k + \frac{\a(N)}{N_k} \frac{1}{m_1^{k-1}} y \phi_k'(y) \right \|_2 \leq \k \frac{\a(N)}{N}
\]
and
\[
\left \| \psi_k - \phi_k \left ( \left ( \frac{m_2^k}{m_1^k} x + 1 - \frac{m_2^k}{m_1^k} \right ) y \right ) \right \|_2 \leq K \left ( \frac{\a(N)}{N} \right ).
\]
In particular, for any $T > 0$, there are constants $d, D \in (0, \pinf)$ such that, with high probability,
\[
d \frac{\a(N)}{N} \leq \tau_{k+1} - \tau_k \leq D \frac{\a(N)}{N}
\]
for all $k \geq 1$ such that $\tau_k \leq T$. 
\end{lemma}

\begin{proof}
A first consequence of the alternative model is that we have
\[
\tau_{k+1} - \tau_k = \s\N(\th_k) - \vs\N (\th_k,B_k) = \s\N(\bar{\th}_{k-1}\N) - \vs\N  \left ( \bar{\th}_{k-1}\N,B\N(\bar{\th}_{k-1}\N) \right ),
\]
and from Lemma \ref{lem:diffgeltimes}, this is
\[
\frac{\a(N)}{N_k} \left ( \frac{m_3^{k-1}}{m_2^{k-1}(m_2^{k-1} - m_1^{k-1})} \right ) + \k O \left ( \frac{\a(N)}{N} \right ).
\]

The second consequence is that
\[
\phi_{k+1} = \frac{N_{k+1}}{N} \bar{G}\N_{\mu_k}
\]
and Lemma \ref{lem:mubar} shows that this is
\[
\frac{N_{k+1}}{N} \left ( G_{\mu_k} - \frac{\a(N)}{N_k} \left ( G_{\mu_k} - \frac{1}{m_1^k} F_{\mu_k} + \k O(1) \right ) \right ).
\]
By Lemma \ref{lem:aftergel}, $N_{k+1} = N_k - \a(N) + \k O(1)$, and it is then an easy computation to conclude.

As a third consequence, we obtain
\[
\psi_{\tau_{k+1}^-} = \frac{1}{N} \fG_{\th_k} \left ( \s\N(\th_k) \right ).
\]
According to Proposition \ref{prop:atgel}, and this is
\[
G_{\mu_k} (q_{\mu_k}) + O \left ( \frac{\a(N)}{N} \right ) = \phi_k \left ( \left ( \frac{m_2^k}{m_1^k} x + 1 - \frac{m_2^k}{m_1^k} \right ) y \right ) + O \left ( \frac{\a(N)}{N} \right ).
\]
Since there are less than $N/\a(N)$ gelation events, these relations are true for all $k$ with $\tau_k \leq T$, with probability at least $1 - K N p_N / \a(N)$, and as we saw, this tends to one because of \eqref{eq:alpha}.

The last part of the result is a direct consequence of the first inequality and the fact that all the $\th_k$ that we consider are in $\K\N_{c,m,M}$ w.h.p.
\end{proof}

\subsection{Tightness and limits}

The result above allows to describe precisely the evolution of $(\tau_k)$, $(\phi_k)$ and $(\psi_k)$. It is therefore not surprising that it allows to describe limits. Classically, we first prove tightness, then show that subsequential limits solve an equation, and finally solve this equation. Recall that the space $C^k$ is the space of $k$ times continuously differentiable functions on $[0,1]$ or $[0,1]^2$ (the context should make obvious which one) endowed with the $\| \cdot \|_k$ norm defined in \eqref{eq:Cknorm}.

\begin{lemma} \label{lem:tightness}
Assume that Assumptions \ref{as:mu5} and \ref{as:alpha} hold. Then the sequence $(n\N)$ is tight in $\D(\R^+,\R)$, and any limiting point is locally Lipschitz-continuous. Moreover, the sequence $(\phi\N)$ is tight in $\D(\R^+,C^5)$, and any limit point is continuous.
\end{lemma}

\begin{proof}
Concerning $(n\N)$ and $(\phi\N)$, this can be done exactly as in \cite{MN0}, using the classical criterion of \cite{BillingsleyCPM}. From the previous result, the number of gelation events on an interval of size $s$ is at most $\a(N)/N s/d$, and each makes $n\N$ (resp. $\phi\N$ in $\| \cdot \|_0$ norm) change by at most $2 \a(N)/N$, see \cite{MN0} for details.

Note that this only shows tightness of $\phi\N$ in $\D(\R^+,C^0)$. However, this implies the tightness of $(\mu\N)$ in $\D(\R^+,\ell^1)$, so in $\ell^1_5(\bN)$ thanks to Lemma \ref{lem:strongconvproc}, and this, along with the tightness of $(n\N)$, implies in turn tightness of $\phi\N$ in $\D(\R^+,C^5)$.
\end{proof}

\begin{prop} \label{prop:PDE}
Assume that Assumptions \ref{as:mu5} and \ref{as:alpha} hold. Then any limit point $(n,\phi)$ of $(n\N,\phi\N)$ in $\D(\R^+,\R) \times \D(\R^+,C^5)$ verifies
\[
n_t = 1, \quad \phi_t = G_{\mu}, \quad t \leq \Tgel,
\]
and
\[
\dfdt n_t = - n_t s(t), \quad \dfdt \phi_t = - n_t s(t) \frac{1}{\phi_t'(1)} y \phi_t'(y), \quad t \geq \Tgel,
\]
where
\[
s(t) = \frac{1}{n_t} \frac{\phi_t''(1) (\phi_t''(1) - \phi_t'(1))}{\phi_t''(1)}.
\]
Moreover, whenever $(\phi\N)$ converges on a subsequence in $\D(\R^+,C^5)$, so does $(\psi\N)$, and its limit $\psi$ then verifies
\[
\psi_t (x,y) = \phi_t \left ( \left ( \frac{\phi_t''(1)}{\phi_t'(1)} x + 1 - \frac{\phi_t''(1)}{\phi_t'(1)} \right ) y \right ), \quad x, y \in [0,1].
\]
\end{prop}

\begin{proof}
The part of the result before $\Tgel$ is clear since we know that w.h.p., no gelation event has occurred before $\Tgel$ (thanks e.g. to \eqref{eq:limTgel}), and therefore $n_t$ and $\phi_t$ are constant on $[0,\Tgel)$. We extend to $\Tgel$ by continuity. The formula for $\phi_t$ is then a consequence of Lemma \ref{lem:charactgraph}.

For the part after $\Tgel$, fix $\Tgel < t < t + s$, and define
\[
r\N_-(t,t+s) = \inf_{u \in [t,t+s]} \frac{m_3^{\mu_u}}{m_2^{\mu_u}(m_2^{\mu_u} - m_1^{\mu_u})}, \quad r\N_+(t,t+s) = \sup_{u \in [t,t+s]} \frac{m_3^{\mu_u}}{m_2^{\mu_u}(m_2^{\mu_u} - m_1^{\mu_u})}.
\]
All the following events will happen w.h.p. First, at least one gelation event has happened by time $t$ by \eqref{eq:limTgel}. Moreover, by Lemma \ref{lem:evolution} and by monotonicity of $(N_t)$, we have, if $t < \tau_i < t+s$, 
\[
\frac{\a(N)}{N_t} \left ( r\N_-(t,t+s) - \k \right ) \leq \tau_{i+1} - \tau_i \leq \frac{\a(N)}{N_{t+s}} \left ( r\N_+(t,t+s) + \k \right ).
\]
The number of gelation events on $[t,t+s]$ is thus at least
\[
\frac{N_{t+s}}{\a(N)} \frac{1}{r\N_+(t,t+s) + \k}
\]
and at most
\[
1 + \frac{N_t}{\a(N)} \frac{1}{r\N_-(t,t+s) - \k}.
\]
By Proposition \ref{prop:atgel}, a gelation event makes at most $(1+\k) \a(N)$ particles falls into the gel. Consequently
\[
- \frac{N_{t+s}}{N} \frac{1}{r\N_+(t,t+s) + \k} \leq n\N_{t+s} - n\N_t \leq - \left ( \frac{\a(N)}{N} + \frac{N_t}{N} \frac{1}{r\N_-(t,t+s) - \k} \right ) (1 + \k).
\]
But $\phi\N \to \phi$ in $\D(\R^+,C^5)$, and $\phi$ is continuous (in time), so there is actually uniform converge on the compact sets (see \cite{BillingsleyCPM}). Therefore
\[
r\N_-(t,t+s) \to \inf_{u \in [t,t+s]} 1/s(u), \quad r\N_+(t,t+s) \to \sup_{u \in [t,t+s]} 1/s(u)
\]
and, passing to the limit in the equality above then shows that
\[
- n_{t+s} \frac{1}{\sup_{u \in [t,t+s]} 1/s(u) + \k} \leq n_{t+s} - n_t \leq - n_t \frac{1}{\inf_{u \in [t,t+s]} 1/s(u) - \k} (1 + \k).
\]
Having $\k \to 0$, then $s \to 0$ and using continuity allows to conclude. The computation for $\phi$ is similar.

Concerning $(\psi\N)$, assume that $(\phi\N)$ converges along some subsequence to some $\phi$. For $t \geq 0$, take $k\N(t)$ the $k$ such that $t \in [\tau\N_k,\tau\N_{k+1})$. Clearly, $\psi$ is an increasing function in time on each interval $[\tau\N_k,\tau\N_{k+1})$, and cannot change by more than $\a(N)/N$ (in $\| \cdot \|_0$ sense) at each gelation event. Therefore, for all $t \geq 0$,
\begin{equation} \label{eq:psitaukNt}
\psi_{k\N(t)-1} - 2 \frac{\a(N)}{N} \leq \psi_{\tau_k\N(t)} \leq \psi_t \leq \psi_{k\N(t)}.
\end{equation}
By Lemma \ref{lem:evolution}, the RHS is
\begin{equation} \label{eq:phikNt}
\begin{split}
\psi_{k\N(t)}(x,y) & = \phi_{k\N(t)} \left ( \left ( \frac{m_2^{k\N(t)}}{m_1^{k\N(t)}} x + 1 - \frac{m_2^{k\N(t)}}{m_1^{k\N(t)}} \right ) y \right ) + O \left ( \frac{\a(N)}{N} \right ) \\
& = \left ( \left ( \frac{\phi''_{k\N(t)}(1)}{\phi'_{k\N(t)}(1)} x + 1 - \frac{\phi''_{k\N(t)}(1)}{\phi'_{k\N(t)}(1)} \right ) y \right ) + O \left ( \frac{\a(N)}{N} \right )
\end{split}
\end{equation}
and a similar formula for the LHS. Again by uniform convergence on compact sets, we have $\phi\N_{k\N(t)} = \phi\N_t \to \phi_t$. On the other hand, $\tau_{k\N(t)} - \tau_{k\N(t) - 1} \to 0$ by Lemma \ref{lem:evolution}, so that $\phi\N_{k\N(t) - 1} \to \phi_t$ for the same reason. Passing to the limit in \eqref{eq:psitaukNt} and \eqref{eq:phikNt} shows the result. This actually just shows pointwise convergence, but \eqref{eq:psitaukNt} and the lower bound on $\tau_{k+1} - \tau_k$ of Lemma \ref{lem:evolution} easily ensure that the variations of $\psi$ are of order $s$ on intervals of size $s$, what allows to conclude classically to tightness, as in Lemma \ref{lem:tightness}.
\end{proof}

We can finally conclude to the convergence of all the relevant quantities by solving the previous equation.

\begin{proof}[Proof of Theorem \ref{th:limit}]
First, assume that $u_t(x)$ is a differentiable function that verifies
\begin{equation} \label{eq:dut}
\rmd u_t(x) = - x q(t) u_t'(x), \quad t \geq T, \quad x \in [0,1],
\end{equation}
where $q$ is a continuous function. This can be solved by method of characteristics. Indeed, define $Q(t) = \exp - \int_T^t q(s) \ds$. Then it is easy to check that $u_t(x/Q(t))$ is constant, and therefore equal to $u_T(x)$, and thus \eqref{eq:dut} has a unique solution given by
\begin{equation} \label{eq:solut}
u_t(x) = u_T(x Q(t)).
\end{equation}

Now, let us define $u_t(x) = \phi_t'(x)/(x \phi_t''(x))$. On the one hand, we have
\begin{equation} \label{eq:utprime}
u_t'(x) = \frac1x - \frac{\phi_t'(x)}{x \phi_t''(x)} - \frac{\phi_t'(x) \phi_t'''(x)}{x^2 \phi_t''(x)^2}.
\end{equation}
Let us define $q(t) = - n_t s(t)/\phi_t'(1)$. A simple computation shows that
\begin{equation} \label{eq:q}
q(t) = \frac{u_t(1) - 1}{u_t'(1) + u_t(1) - 1}.
\end{equation}
But on the other hand, using the PDE solved by $\phi$ of Proposition \ref{prop:PDE} for $T = \Tgel$, it is easy to check that $u$ solves \eqref{eq:dut}, so that $u_t = u_T(x Q(t))$. Plugging this in \eqref{eq:q}, we see that $Q$ verifies
\[
q(t) = - \frac{Q'(t)}{Q(t)} = \frac{u_T(Q(t)) - 1}{Q(t) u_T'(Q(t)) + u_T(Q(t)) - 1}.
\]
This shows that $R(t) = Q(t) (1 - u_T(Q(t)))$ verifies
\begin{equation} \label{eq:R}
R'(t) = - R(t), \quad t \geq \Tgel.
\end{equation}
But, by continuity, $\phi_T = G_{\mu}$, so that
\[
u_T(x) = \frac{G'_{\mu}(x)}{x G''_{\mu}(x)} = \frac{G_{\nu}(x)}{x G'_{\nu}(x)}.
\]
Therefore
\[
R(T) = Q(T) (1 - u_T(Q(T))) = - u_T(1) = - \frac{m_2^{\mu}}{m_1^{\mu}},
\]
which, with \eqref{eq:R} shows that $R(t) = e^{-t}$ for $t \geq \Tgel$. By definition of $R$, this is just saying that $Q(t)$ solves \eqref{eq:Q}.

Finally, since $\phi$ also solves \eqref{eq:dut}, we get from \eqref{eq:solut} that
\begin{equation} \label{eq:phit}
\phi_t(x) = \phi_T(x Q(t)) = G_{\mu}(x Q(t)),
\end{equation}
whence we deduce the solution for $n_t = \phi_t(1)$ and $\psi_t$ given in the previous result.
\end{proof}

%% file: 9-Conclusion.tex
\section{Conclusion} \label{sec:conclusion}

\subsection{Typical clusters} \label{sec:lastres}

With all these tools in hand, proving Theorem \ref{th:typcluster} is an easy task. According to Theorem \ref{th:equiv}, it suffices to do it for the alternative model. Recall that, at each time, the alternative model $(Z(t))$ is the realization of a DCM at some time. In part, for each $t$, $Z(t)$ is a CM with some random parameters $M_t$ and $\l\N_t$. By Theorems \ref{th:limit} and \ref{th:equiv}, we have
\[
\frac{M_t}{N} \to n_t, \quad \l\N_t \to \pi_t.
\]
The second convergence holds in $\ell^1_5$. Hence, $(\l\N_t)$ verifies the assumptions of Proposition \ref{prop:locconv} with limit $\pi_t$. Proposition \ref{prop:locconv} then applies for $Z(t)$, conditionally on $M_t$ and $\l\N_t$, and it then suffices to de-condition. To be precise, one can use Skorokhod representation and prove the result as \eqref{eq:typcluster}. The easy details are left to the reader.

\subsection{Limiting concentrations} \label{sec:limconc}

Let us now explain Formula \eqref{eq:cinf}. We assume that we are in the supercritical case $m_2^{\mu} > m_1^{\mu}$. The reasoning is again done as in Section \ref{sec:locconv}, and we consider the same notation.

First, note that $Q(t)$ lives in $[0,1]$, and that any limit point $c$ has to solve \eqref{eq:c}. This equation has a unique solution, and thus $Q(t) \to x$ as $t \to \pinf$. Hence, taking limits in \eqref{eq:pit}, we see that $\pi_t \to \rho$, where $\rho$ is a probability with generating function
\[
G_{\rho}(x) = \frac{G_{\mu} (c x)}{G_{\mu}(c)}.
\]

Similarly to the reasoning done in Section \ref{sec:locconv}, we have
\[
\cinf\N(0,m) = \frac1m \Q \left ( v(\Ctyp\N(\infty)) = m \right ).
\]
According to Theorem \ref{th:typcluster}, $\Ctyp\N(\infty)$ should be a $\GW_{\rho,\hat{\rho}}$. This is of course not rigorously true since we only have convergence on compacts sets. Notwithstanding,
\[
\lim_{N \to \pinf} \Q \left ( v(\Ctyp\N(\infty)) = m \right ) = \Q  \left ( \GW_{\rho,\hat{\rho}} = m \right ) = \frac{1}{m-1} \b^{m-1} \nu^{*m}(m-2),
\]
where, again, the last equality is an application of Dwass' formula \cite{Dwass}.

This finally explains the appearance of the term $\b$ in Formula \eqref{eq:cinf}, compared to Formula \eqref{eq:cinf0}. In the latter case, there is no gelation, and typical clusters in the final state are subcritical Galton-Watson trees. In the former case, there is gelation, and typical clusters in the final state are critical Galton-Watson trees $\GW_{\rho,\hat{\rho}}$. Their reproduction law is obtained from the original subcritical law $\mu$ by an exponential tilt, which turns it into a critical law.

\subsection{Typical particles in solution}

Theorems \ref{th:limit} also shows how typical particles in solution evolve. Consider a particle $p$ uniformly chosen at some time $t \geq \Tgel$. Then, in the asymptotics $N \to \infty$, 
\begin{itemize} 
\item $p$ is still in solution with probability $n_t = G_{\mu}(Q(t))$;
\item $p$ is in solution and has degree $n$ with probability $\mu(n) Q(t)^n$ (the coefficient of $y^n$ in $\psi_t(1,y)$); 
\item knowing that $p$ has degree $n$, it is still in solution at time $t$ with probability $Q(t)^n$.
\end{itemize}

The last point suggests that a form of asymptotic independence
holds throughout the whole process between the distinct arms of $p$. Everything indeed happens as if each of these arms causes, independently of the others, the fall of $p$ into the gel by time $t$ with probability $1-Q(t)$. Another way to see this is that the growth of the different subclusters attached to each arm of a given  particle most likely do not interfere with one another, so that the particle falls into the gel when one (and only one) of this clusters reaches size about $\a(N)$.

\subsection{Particular cases}

\subsubsection{Poisson distribution}

Assume that $\mu$ is a Poisson distribution $\cP(\l)$ with parameter $\l > 1$.  Then one readily obtains
\[
\Tgel = - \log \left ( 1 - \frac{1}{\l} \right ).
\]
Moreover, $\nu = \mu$, and for $t \geq \Tgel$, $Q(t) = e^{-t} + 1/\l$,
\[
n_t = \exp \left ( 1 - \l + \l e^{-t} \right ),
\]
and $\pi_t = \rho := \cP(1)$. Hence, in this case, the distribution in solution is essentially that of a configuration model on $n_t N$ particles with degree distribution $\rho$. A typical cluster in solution is a $\GW_{\rho}$, and has thus a stationary distribution. This should not be surprising, since the configuration model with a $\cP(\l)$ distribution corresponds to the Erd\H{o}s-R\'enyi graph with parameter $\l$. This agrees with the results of \cite{MN0}, where typical clusters in solution are such $\GW_{\rho}$ trees post-gelation.

\subsubsection{Fixed degree}

Assume now $\mu = \dl_d$ for some $d \geq 3$, that is, all particles initially have degree $d$. Then
\[
\Tgel = - \log \frac{d-2}{d-1}.
\]
Moreover, $\nu = \dl_{d-1}$, and for $t \geq \Tgel$, $Q(t) = e^{-t} (d-1)/(d-2)$,
\[
n_t = \left ( \frac{d-1}{d-2} e^{-t} \right )^d,
\]
and $\pi_t = \rho := \mathrm{Bin}(d,1/(d-1))$. Finally, $\hat{\rho} = \mathrm{Bin}(d-1,1/(d-1))$. In this case, the distribution in solution is essentially that of a configuration model on $n_t N$ particles with degree distribution $\rho$. A typical cluster in solution is thus a $\GW_{\rho,\hat{\rho}}$, and is again stationary.

Notice that in this case, $n_t \to 0$, which makes sense since there are no particles with degree 1, so that, in some sense, clusters cannot ``close'' by attaching such particles at their extremity.

\subsubsection{Binomial distribution}

Assume finally that $\mu = \mathrm{Bin}(d,p)$ for $d \geq 3$ and $p > 1 / (d-1)$. Then
\[
\Tgel = - \log \left ( 1 - \frac{1}{(d-1) p}\right )
\]
and $\nu = \mathrm{Bin}(d - 1,p)$. The computations are slightly more tedious in this case, but it is nonetheless easy to obtain that, for $t \geq \Tgel$,
\[
n_t = \left ( \frac{d-1}{d-2} (1-p+pe^{-t}) \right )^d,
\]
as well as $\pi_t = \rho := \mathrm{Bin}(d,1/(d-1))$ and $\hat{\rho} = \mathrm{Bin}(d-1,1/(d-1))$.

We therefore have the same typical clusters in solution $\GW_{\rho,\hat{\rho}}$ as in the previous case. This is easy to explain. Note indeed that, before gelation, the number of activated of the particles has distribution $\mathrm{Bin}(d,(1- e^{-t}))$ in the previous case, and $\mathrm{Bin}(d,p(1- e^{-t}))$ in the current one. At gelation, typical clusters are therefore $\GW_{\rho,\hat{\rho}}$ in both cases, and then the evolutions are similar. There is however a non-trivial time-change due to the difference in degrees. In the binomial case, there is moreover a positive proportion of particles with one arm, and thus there remains a positive mass at the end. In the final state, we essentially just observe a realization of a $CM(n_{\infty}N, \rho)$, whose typical clusters are $\GW_{\rho,\hat{\rho}}$ trees.

All these examples show a stationary distribution for the typical clusters. However, this distribution is not stationary after gelation for any other example that we tried, but the computations are not as nice and do not make well-known distributions appear.

%% file: A-Index-notation.tex
\section{Index of notation} \label{ap:indexnot}

\paragraph{Sets}

\begin{itemize}
\item $\bN = \{0,1,\dots\}$;
\item $[N] = \{1,2,\dots,N\}$;
\item $\R^+ = [0,\pinf)$.
\end{itemize}

\paragraph{Measures}

\begin{itemize}
\item $\cM_1$: probability measures on $\bN$;
\item $\cM_1\N$: probability measures $\mu$ on $\bN$ such that $N \mu(k) \in \bN$ for all $k \in \bN$;
\item for $\mu \in \cM_1$,
\[
G_{\mu}(y) = \sum_{r \geq 0} \mu(r) y^r, \quad y \in [0,1],
\]
is the generating function of $\mu$;
\item for $i \in \bN$,
\[
m_i^{\mu} = G_{\mu}^{(i)}(1)
\]
are the factorial moments of $\mu$;
\item $F_{\mu}(x) = x G_{\mu}'(x)$;
\item for $\mu \in \cM_1$ with $m_2^{\mu} > 0$,
\[
q_{\mu}(x,y) = \left ( \frac{m^{\mu}_1}{m^{\mu}_2} x + 1 - \frac{m^{\mu}_1}{m^{\mu}_2} \right ) y;
\]
\item for $\mu \neq 0$,
\[
\hat{\mu}(k) = \frac{(k+1)\mu(k+1)}{\sum_{i \geq 1} i \mu(i)}, \quad k \in \bN.
\]
\end{itemize}

\paragraph{Topologies}

\begin{itemize}
\item $\D(\R^+,S)$: c\`adl\`ag processes with values in a Polish space $S$, endowed with the Skorokhod topology (see \cite{BillingsleyCPM,Kallenberg});
\item for $\b > 1$, $\ell^1_{\b}(\bN^d)$: the space of sequences in $\bN^d$ with finite $\| \cdot \|_{\b}$-norm, where
\[
\| u \|_{\b} = \sum_{(k_1,\dots,k_d) \in \bN^d} (1 + k_1^{\b} + \dots + k_d^{\b}) | u(k,r) |;
\]
\item for $k \in \bN$, $C^k$: space of $k$ times continuously differentiable functions on $[0,1]$ or $[0,1]^2$, endowed with the norm
\[
\| f \|_{\g} = \sum_{i=0}^{k} \sup_{x \in [0,1]} \left | D^i f (x) \right |.
\]
\end{itemize}

\paragraph{Smoluchowski equation}

We write $S = \bN \times \bN^*$ and for $p=(a,m), p'=(a',m') \in S$,
\begin{itemize}
\item $p \cdot p' = a a'$;
\item $p \circ p' = (a+a'-2, m+m')$;
\item $p' \lesssim p$ if $a' \leq a+1$ and $m' \leq m-1$;
\item if $p' \lesssim p$, $p \bsl p' = (a+2-a', m-m')$.
\end{itemize}
For two nonnegative families $(c(p), p \in S)$, $(f(p),p \in S)$, \[
\la c, f \ra = \sum_{p \in S} c(p) f(p).
\]
With a slight abuse of notation, $\la c, a \ra = \la c, f \ra$ for $f(a,m) = a$, and so on.

\paragraph{Discrete model}

\begin{itemize}
\item $N_t$: number of particles in solution at $t$;
\item $n_t = N_t/N$: concentration in solution at $t$;
\item $P_t(k,r)$: number of particles in solution $k$ activated arms and degree $r$;
\item $p_t(k,r) = P_t(k,r)/N$: concentration of particles that are in solution and with $k$ activated arms and degree $r$;
\item $\mu_t(r) = \sum_{k=0}^r P_t(k,r) / N_t$: distribution of the degree of the particles in solution;
\item $\pi_t(k) = \sum_{r=0}^{\pinf} P_t(k,r) / N_t$: distribution of the number of activated arms of the particles in solution.
\item $\psi_t$: generating function of $p_t$;
\item $\phi_t(y) = \psi_t(1,y) = \frac{N_t}{N} G_{\mu_t}$;
\item $B(t) = \sum_{k,r \geq 0} k P_t(k,r)$: number of activated arms in solution;
\item $\tau_i$: $i$-th gelation time.
\end{itemize}

\paragraph{Other quantities}

\begin{itemize}
\item $CM(n,\pi\n)$: configuration model on $n$ vertices with degree distribution $\pi\n$;
\item $p_t = 1 - e^{-t}$;
\item $q_t(x,y) = (p_t x + 1 - p_t) y$;
\item $p_N = N / \a(N)^{3/2} + (\a(N)/N)^{1+\eta}$;
\item $\tK\N_{c,m,M}$: the set of all couples
\[
(n,\mu) \in \bigcup_{n \geq 1} \{ n \} \times \cM_1\n
\]
such that
\[
c N \leq n \leq N, \quad m_3^{\mu} \geq m, \quad \sum_{r \geq 0} r^{5+\eta} \mu(r) \leq M;
\]
\item $\K\N_{c,m,M}$: set of $(n,\mu) \in \tK\N_{c,m,M}$ such that additionally
\[
\frac{m^{\mu}_2}{m^{\mu}_1} \geq 1 + m.
\]
\end{itemize}